\documentclass[11pt]{amsart} 
\usepackage{amsmath,amsthm}
\usepackage{xypic}
\usepackage{enumerate}

\usepackage{amscd}
\usepackage{amsfonts,amssymb,latexsym} 
\usepackage{mathrsfs}
\usepackage[all]{xy} 
\xyoption{arc}
\CompileMatrices

\newcommand{\udots}{\mathinner{\mskip1mu\raise1pt\vbox{\kern7pt\hbox{.}}
\mskip2mu\raise4pt\hbox{.}\mskip2mu\raise7pt\hbox{.}\mskip1mu}}

\newcommand{\SD}{{\mathcal{D}}}
\newcommand{\SE}{{\mathcal{E}}}

\newcommand{\SL}{{\mathcal{L}}}
\newcommand{\SM}{{\mathcal{M}}}

\newcommand{\SO}{{\mathcal{O}}}

\newcommand{\SQ}{{\mathcal{Q}}}
\newcommand{\SR}{{\mathcal{R}}}

\newcommand{\SW}{{\mathcal{W}}}
\newcommand{\SX}{{\mathcal{X}}}

\newcommand{\GG}{\mathbb{G}}
\newcommand{\PP}{\mathbb{P}}
\newcommand{\ZZ}{\mathbb{Z}}

\newcommand{\CC}{\mathbb{C}}
\newcommand{\RR}{\mathbb{R}}

\newcommand{\Gr}{\operatorname{Gr}}
\newcommand{\Spec}{\operatorname{Spec}}

\newcommand{\Quot}{\operatorname{Quot}}
\newcommand{\Grass}{\operatorname{Grass}}

\newcommand{\Sch}{\operatorname{Sch}}
\newcommand{\Sets}{\operatorname{Sets}}
\newcommand{\Sym}{\operatorname{Sym}}
\newcommand{\id}{\operatorname{Id}}
\newcommand{\Ker}{\operatorname{Ker}}
\newcommand{\im}{\operatorname{Im}}

\newcommand{\rk}{\operatorname{rk}}
\newcommand{\pdeg}{\operatorname{pardeg}}
\newcommand{\pmu}{\operatorname{par-\mu}}
\newcommand{\pchi}{\operatorname{par-\chi}}
\newcommand{\parP}{\operatorname{par-P}}
\newcommand{\SParHom}{\operatorname{SParHom}}
\newcommand{\ParHom}{\operatorname{ParHom}}
\newcommand{\SParEnd}{\operatorname{SParEnd}}

\newcommand{\End}{\operatorname{End}}
\newcommand{\owt}{\operatorname{wt}}

\newcommand{\tr}{\operatorname{tr}}
\newcommand{\GL}{\operatorname{GL}}

\newcommand{\gitq}{/\!\!/}

\newcommand{\Res}{\operatorname{Res}}
\newcommand{\fprod}[3]{{\prod_{{#1}}\!{}_{\substack{\\ {#2}}}^{{#3}}}}

\newcommand{\op}{\operatorname}

\newtheorem{proposition}{Proposition}[section]
\newtheorem{theorem}[proposition]{Theorem}
\newtheorem{definition}[proposition]{Definition}

\newtheorem{lemma}[proposition]{Lemma}

\newtheorem{corollary}[proposition]{Corollary}

\numberwithin{equation}{section}

\title[Moduli parabolic $\Lambda$-modules]{Moduli space of parabolic $\Lambda$-modules over a curve}
\author[D. Alfaya]{David Alfaya}

\date{}

\address{Instituto de Ciencias Matem\'aticas (CSIC-UAM-UC3M-UCM),
Nicol\'as Cabrera 15, Campus Cantoblanco UAM, 28049 Madrid, Spain}

\email{david.alfaya@icmat.es}


\begin{document}

\begin{abstract}
Simpson, in 1994, introduced the notion of $\Lambda$-modules and constructed the corresponding moduli space, where $\Lambda$ is a sheaf of rings of differential operators. Higgs bundles, connections and $\lambda$-connections (as defined by Delgine) are particular cases of $\Lambda$-modules. In this article the concept of parabolic $\Lambda$-modules over a curve is introduced and their moduli space is built. As an application, we construct the parabolic Hodge moduli space parameterizing parabolic $\lambda$-connections.
\end{abstract}

\maketitle

\section{Introduction}

Simpson \cite{Si2} developed the concept of $\Lambda$-modules as a theoretical framework that unified the notions of vector bundle, Higgs bundle, integrable connection and other similar geometric structures. The main idea is to consider the corresponding Higgs field or connection as an action of a certain sheaf of rings of differential operators on a coherent sheaf. For example, if we have a Higgs field $\varphi:E\to E\otimes K$ over a coherent sheaf $E$ with $\varphi \wedge \varphi=0$, it induces a morphism $\varphi':K^\lor \otimes E\to E$ that extends, by composition, to a morphism $\varphi'': \Sym^\bullet(K^\lor) \otimes E \to E$. Therefore, providing a Higgs field is equivalent to defining a left action of the sheaf of algebras $\Lambda^{\op{Higgs}}:=\Sym^\bullet(K^\lor)$ on $E$.

Similarly, sheafs with an integrable connection, described as a sheaf $E$ together with a $\CC$-linear morphism $\nabla:E\to E\otimes K$ satisfying the Leibniz rule such that $\nabla^2=0$, are in correspondence with $\SD_X$-modules, i.e., sheafs $E$ with a left action of the sheaf of differential operators $\Lambda^{\op{DR}}:=\SD_X$. This approach had been studied by \cite{Bernstein83} and motivated the definition given by Simpson of sheaf of rings of differential operators. A sheaf of rings of differential operators over $X$ is a filtered $\SO_X$-algebras satisfying some conditions resembling the main properties of $\SD$; the left and right action of $\SO_X$ on the graduate are the same (the algebra of symbols of operators of a certain degree is commutative), the graduate at each point is coherent (the algebra of symbols of operators of a given order is finite-dimensional) and the graded algebra is generated by the first step of the filtration (the algebra of differential operators is generated by operators of order one).

A $\Lambda$-module is a left module $E$ for the sheaf of rings $\Lambda$ where the $\SO_X$-module structure coming from $\SO_X\hookrightarrow \Lambda$ coincides with the $\SO_X$-module structure of $E$, i.e., it is an $\SO_X$-module $E$ endowed with an action
$$\varphi:\Lambda\otimes_{\SO_X} E\longrightarrow E$$

Simpson proved that for every $\Lambda$ satisfying the previous properties there exists a quasi-projective moduli space of semistable $\Lambda$-modules for a certain natural semi-stability condition. Some important examples of moduli spaces that can be constructed as instances of this general theorem include the moduli spaces of vector bundles, Higgs bundles (or, in general, Hitchin pairs/twisted Higgs bundles), connections, logarithmic connections or $\lambda$-connections.

On the other hand, let $C$ be a smooth complex projective curve and let $D$ be a finite set of points in $C$ that we will consider as punctures on a Riemann surface. We are interested in studying variants of the previous geometric contraptions over $C$ where we allow the existence of logarithmic singularities over the punctures in $D$, modulated by a ``parabolic'' structure over $D$, i.e., a filtration of the fibers of the underlying sheaf at each of the punctures preserved by the action of the Higgs field or connection. The moduli space of parabolic vector bundles over a curve was described by Mehta and Seshadri \cite{MS}.  Maruyama and Yokogawa generalized the concept of parabolic sheaf to arbitrary dimension and proved the existence of a coarse moduli space of parabolic sheafs \cite{MY92}. Later on, Yokogawa  built the moduli space of parabolic Higgs bundles \cite{Yokogawa93} . The moduli space of logarithmic connections (without a parabolic structure) has been built by Nitsure \cite{Ni2} and a notion of moduli space of parabolic connections was developed in \cite{IIS} in order to study solutions to the Painlev\'e VI equation on $\PP^1$. In this paper, we adapt the approach of $\Lambda$-modules introduced by Simpson to the parabolic scenario in order to unify the previous results in a single theoretical framework and build some similar, yet unknown, moduli spaces such as the parabolic Hodge moduli space, parameterizing parabolic $\lambda$-connections.

A parabolic $\Lambda$-module is a $\Lambda$-module $(E,\varphi)$ together with a filtration of the fiber $E|_x$ over the each parabolic point $x\in D$
$$E|_x=E_{x,1}\supsetneq E_{x,2}\supsetneq \cdots \supsetneq E_{x,l_x+1}=0$$
and a sequence of real weights $0\le \alpha_{x,1}<\alpha_{x,2}<\ldots <\alpha_{x,l_x}<1$ such that the action of $\Lambda$ preserves the filtration in a certain sense. The stability for $\Lambda$-modules is substituted by a notion of stability depending on the system of weights $\alpha=\{\alpha_{x,i}\}$ and the filtration $E_\bullet=\{E_{x,i}\}$. The new definition is a natural generalization of existing ones for parabolic vector bundles, parabolic Higgs bundles and parabolic connections and admits the usual contraptions such as the Harder-Narasimhan and Jordan-H\"older filtrations. The main result obtained in this work is the following (Theorem \ref{thm:LambdaModModuli})

\begin{theorem}
Let $\Lambda$ be a sheaf of rings of differential operators on $X=C\times S$ over $S$ such that $\Lambda|_{D\times S}$ is locally free. Then there exist a coarse moduli space parameterizing S-equivalence classes of semistable parabolic $\Lambda$-modules over $(C,D)$ and an open subset parametrizing isomorphism classes of stable parabolic $\Lambda$-modules.
\end{theorem}

The first part of the article is devoted to reviewing the notion of sheaf of rings of differential operators and $\Lambda$-modules as introduced by Simpson \cite[\S 2]{Si2} and generalizing its properties to the parabolic scenario. Parabolic $\Lambda$-modules are defined and we give a notion of stability for parabolic $\Lambda$-modules both for complex schemes $X$ of the form $X=C\times S$, over $S$, where $C$ is a complex projective curve. Versions of the Harder-Narasimhan and Jordan-H\"older filtrations for parabolic $\Lambda$-modules are constructed.

The main question treated in section \ref{section:boundnessTheorems} is the boundedness of the family of semistable parabolic $\Lambda$-modules. We provide uniform bounds for the Mumford-Castelnuovo regularity of both semistable parabolic $\Lambda$-modules and destabilizing subsheafs of (possibly unstable) parabolic $\Lambda$-modules. We prove several technical lemmas introducing inequalities over the sections of twists of subsheafs of parabolic $\Lambda$-modules.

Section \ref{section:parameterSpace} describes the construction of a parameterizing space $R^{ss}$ for the family of semistable parabolic $\Lambda$-modules. First, we describe a projective scheme parameterizing parabolic quotients of a given sheaf such that the filtrations have a given fixed type. Then, starting from Simpson's rigidification of $\Lambda$-modules as quotients of $\Lambda_r\otimes \SO_X(-N) \otimes_\CC \CC^{P(N)}$ for a suitable $N$, we use this ``filtered quot scheme'' to incorporate the filtration to the parameter space. Finally we prove that the space is a quasi-projective variety that can be embedded into a product of Grassmanians over $S$ using Grothendieck's embedding of the Quot scheme \cite{GrothQuot}.

In section \ref{section:GIT}, we use Geometric Invariant Theory to construct a universal categorical quotient of the previous parameterizing space which corepresents the moduli functor of families of semistable parabolic $\Lambda$-modules over $X$. GIT-semi-stability conditions are computed for the natural action of $\op{SL}(V)$, where $V$ is a complex vector space $V$, on the product of Grassmanians of the form $\Grass(V\otimes W,p)$ for some vector space $W$. We use this numerical criterion to describe GIT-semistable parabolic points of $R^{ss}$ and we prove that GIT-semi-stability coincides with slope-stability over the parameter space.

When dealing with parabolic Higgs bundles or parabolic connections, we have a natural notion of residue of the Higgs field or the logarithmic connection at each parabolic point $x\in D$ as the ``$-1$ coefficient'' of the Laurent expansion of the field near the point. In both cases, the residue must preserve the parabolic filtration. Moreover, when we study the geometry of the moduli space of parabolic vector bundles a condition over the residue of the Higgs field or the connection respectively arises naturally. In the case of parabolic Higgs bundles, we usually prescribe the fields to be ``strongly parabolic'', so all the eigenvalues are zero. In the case of parabolic connections, if we want them to correspond to ``strongly parabolic'' Higgs bundles through Simpson's correspondence \cite{SimpsonNonCompact} then the eigenvalues of the residue of the connection must be required to be equal to the corresponding parabolic weight. As $\Lambda$-modules are a generalization of these concepts, in section \ref{section:residualLambdaModules} we aim to generalize these kinds of requisites to other classes of $\Lambda$-modules.

We define the concept of ``total residue'' of a parabolic $\Lambda$-module $(E,E_\bullet,\varphi)$ as the morphism
$$\op{Res}(\varphi,x):\Lambda|_{\{x\}\times S} \otimes_{\SO_S} E|_{\{x\}\times S} \longrightarrow E|_{\{x\}\times S}$$
induced by $\varphi:\Lambda\otimes E\to E$ at the parabolic points. Our definition of parabolic $\Lambda$-modules ensures that this map is well defined and preserves the parabolic filtration in the sense that
$$\op{Res}(\varphi,x)\left(\Lambda|_{\{x\}\times S} \otimes_{\SO_S} E_{x,i}\right) \subseteq E_{x,i}$$
Then, for every section $R\in H^0(S,\Lambda|_{\{x\}\times S})$, the ``total residue'' induces an endomorphism of the fiber $\op{Res}_R(\varphi,x)\in \End(E|_{\{x\}\times S})$.
We prove that the usual notions of residue for parabolic Higgs bundles and parabolic connections can be recovered within this theoretical framework. Then, we define ``residual $\Lambda$-modules'' as the parabolic $\Lambda$-modules that satisfying a certain additional condition on the residue analogous to the control of the eigenvalues appearing in parabolic Higgs bundles or parabolic connections. The moduli of ``residual $\Lambda$-modules'' is built as a closed subscheme of the moduli of parabolic $\Lambda$-modules, obtaining the following theorem (Theorem \ref{thm:ResLambdaModModuli})
\begin{theorem}
There exist a coarse moduli scheme parametrizing S-equivalence classes of semistable ``residual'' parabolic $\Lambda$-modules and an open subset parametrizing isomorphism classes of stable ones.
\end{theorem}

In general, the schemes constructed in sections \ref{section:GIT} and \ref{section:residualLambdaModules} are only coarse moduli spaces for the corresponding moduli problems. In section \ref{section:fineModuli} we provide a numerical condition which, when satisfied, implies that the subschemes parameterizing stable objects admit a universal family and, therefore, they are fine moduli spaces for their corresponding moduli problems. In particular, we prove the following result (Corollary \ref{cor:fineModuli} of Theorem \ref{thm:fineModuli})
\begin{theorem}
If the system of weights $\alpha$ is full flag, then the moduli spaces of stable parabolic $\Lambda$-modules and stable residual parabolic $\Lambda$-modules are fine, i.e., they admit a universal family.
\end{theorem}

Finally, in section \ref{section:lambdaConnections} we apply the previous theorems to the construction of the moduli space of parabolic $\lambda$-connections for the group $\op{SL}_r(\CC)$ (Theorem \ref{thm:moduliLambdaConn}). We use the deformation to the graduate of the De Rham sheaf of logarithmic differential operators $\Lambda^{\op{DR},\log D}$ over $C$ with poles over $D$ to obtain a sheaf of differential operators $\Lambda^{\op{DR},\log D,R}$ over $C\times \mathbb{A}^1$, such that residual parabolic $\Lambda^{\op{DR},\log D,R}$-modules over $\op{SL}$ correspond to parabolic $\lambda$-connections. The fiber over $\lambda=1$ of $\Lambda^{\op{DR},\log D,R}$ coincides with $\Lambda^{\op{DR},\log D}$ and the fiber over $\lambda=0$ is $\Gr(\Lambda^{\op{DR},\log D})\cong \Sym(K^\lor(D))\cong \Lambda^{\op{Higgs},\log D}$. We conclude that the constructed moduli space is a quasi projective variety over $\mathbb{A}^1$ such that its fiber over $0$ coincides with the parabolic Higgs moduli space and the fiber over $1$ (in fact, over every nonzero $\lambda$) is isomorphic to the moduli space of parabolic connections.

\noindent\textbf{Acknowledgments.} 
This research was funded by MINECO (grant MTM2016-79400-P and ICMAT Severo Ochoa project SEV-2015-0554) and the 7th European Union Framework Programme (Marie Curie IRSES grant 612534 project MODULI). The author was also supported by a predoctoral grant from Fundaci\'on La Caixa -- Severo Ochoa International Ph.D. Program.
I would really like to thank Tom\'as G\'omez for all the useful discussions held during the development of this work and for reviewing the paper.

\section{Parabolic $\Lambda$-modules}
\label{section:ParLambdaModules}
Let $p:X\longrightarrow S$ be any relative smooth projective variety over a complex scheme $S$.

\begin{definition}[Sheaf of rings of differential operators]
A sheaf of rings of differential operators on $X$ over $S$ is a sheaf of $\SO_X$-algebras $\Lambda$ over $X$, with a filtration by sub-algebras $\Lambda_0 \subseteq \Lambda_1 \subseteq \ldots$ which satisfies the following properties
\begin{enumerate}
\item $\Lambda=\bigcup_{i=0}^{\infty} \Lambda_i$ and for every $i$ and $j$, $\Lambda_i \cdot \Lambda_j \subseteq \Lambda_{i+j}$
\item The image of the morphism $\SO_X \to \Lambda$ is equal to $\Lambda_0$.
\item The image of $p^{-1}(\SO_S)$ in $\SO_X$ is contained in the center of $\Lambda$.
\item The left and ring $\SO_X$-module structures on $\Gr_i(\Lambda):=\Lambda_i/\Lambda_{i-1}$ are equal.
\item The sheafs of $\SO_X$-modules $\Gr_i(\Lambda)$ are coherent.
\item The morphism of sheaves
$$\Gr_1(\Lambda) \otimes \cdots \otimes \Gr_1(\Lambda) \to \Gr_i(\Lambda)$$
induced by the product is surjective.
\end{enumerate}
\end{definition}

We will denote by $\Lambda^{\op{DR}}=\SD_{X/S}$ the sheaf of differential operators over $X$ relative to $S$ \cite{Bernstein83}. It represents the main example of sheaf of rings of differential operators and, in fact, the previous set of properties are meant to be an abstraction of the principal characteristics of $\SD_{X/S}$. Its graduate $\Lambda^{\op{Higgs}}=\Gr_\bullet(\SD_{X/S})$ with the induced sheaf of algebras structure and its deformation to the graduate $\Lambda^{\op{DR},R}$ are additional examples.

\begin{lemma}
\label{lemma:sdoSurj}
Let $\Lambda$ be a sheaf of rings of differential operators over $X$. Then for every $i,j\ge 0$
$$\Lambda_i \cdot \Lambda_j = \Lambda_{i+j}$$
\end{lemma}

\begin{proof}
It is enough to prove that $\underbrace{\Lambda_1\cdots \Lambda_1}_{i} = \Lambda_i$, as then
$$\Lambda_i \cdot \Lambda_j = \underbrace{\Lambda_1 \cdots \Lambda_1}_i \cdot \underbrace{\Lambda_1 \cdots \Lambda_1}_j = \underbrace{\Lambda_1 \cdots \Lambda_1}_{i+j} = \Lambda_{i+j}$$
By induction, it is enough to prove that for every $i$, $\Lambda_i \cdot \Lambda_1 = \Lambda_{i+1}$, i.e., that the morphism $\Lambda_i \otimes \Lambda_1 \to \Lambda_{i+1}$ is surjective.
Let $U\subseteq X$ be open. Let $v\in \Lambda_{i+1}(U)$. As $\underbrace{\Gr_1(\Lambda(U)) \otimes \cdots \otimes \Gr_1(\Lambda(U))}_{i+1} \to \Gr_{i+1}(\Lambda(U))$ is surjective, there exist $\overline{w_{1,1}}, \ldots, \overline{w_{l,i+1}}\in \Gr_1(\Lambda(U))$ such that
$$\sum_{j=1}^l \overline{w_{j,1}} \cdots \overline{w_{j,i+1}} \cong v \mod \Lambda_i(U)$$
Let $w_{j,i}$ be any representative of $\overline{w_{j,i}}$ in $\Lambda_1(U)$. Then there exists $v'\in \Lambda_i(U)$ such that
$$v=\sum_{j=1}^l w_{j,1}\cdots w_{j,i+1} + v'$$
By induction hypothesis, there exist $v_{1,1},\ldots, v_{m,i} \in \Lambda_1(U)$, such that
$$v'=\sum_{j=1}^m v_{j,1}\cdots v_{j,i}$$
Let $1\in \Lambda_0(U)\subset \Lambda_1(U)$ be the unity of the ring. Then
$$v=\sum_{j=1}^l w_{j,1}\cdots w_{j,i+1} + \sum_{j=1}^m v_{j,1}\cdots v_{j,i}\cdot 1 \in \underbrace{\Lambda_1(U)\cdots \Lambda_1(U)}_{i+1}$$
\end{proof}

\begin{definition}[$\Lambda$-module]
Let $X$ be an $S$-scheme. Let $\Lambda$ be a sheaf of rings of differential operators over $X$. A $\Lambda$-module over $X$ is a sheaf $E$ of left $\Lambda$-modules over $X$ such that $E$ is coherent with respect to the structure of $\SO_X$-modules induced by the morphism $\SO_x\to \Lambda_0$.
\end{definition}

Under the previous definition, a vector bundle with an integrable connection can be alternatively described as a locally free $\Lambda^{\op{DR}}$-module. Similarly, Higgs bundles correspond to locally free $\Lambda^{\op{Higgs}}$-modules and $\lambda$-connections on $X$ correspond to $i_\lambda^*\Lambda^{\op{DR},R}$-modules on $X\times \{\lambda\}\subset X\times \mathbb{A}^1$, where $i_\lambda:\{\lambda\}\hookrightarrow \mathbb{A}^1$.

Now, let $C$ be a smooth complex projective curve. Let $D$ be a finite set of points in $C$. Let $S$ be a complex scheme. Let us consider the complex scheme $X=C \times S$, considered as a relative smooth projective variety over $S$. Let $\SO_X(1)=p^* \SO_C(1)$ be an $S$-very ample invertible sheaf. Let $\bar{D}:=D\times S \subset X$. Then it is an effective Cartier divisor on $X/S$. We are interested in parameterizing certain kinds of geometric objects over $X$ with logarithmic singularities along $\bar{D}$ such as parabolic  connections or parabolic Higgs fields. We generalize these notions by enhancing a $\Lambda$-module over $X$ with an additional parabolic structure over $\bar{D}$.

\begin{definition}[Family of parabolic vector bundles]
A family of parabolic vector bundles over $(C,D)$ parameterized by $S$ is a vector bundle $E$ over $C\times S$ together with a weighted flag on $E|_{\{x\}\times S}$ for each $x\in D$ called parabolic structure, i.e., a filtration
$$E|_{\{x\}\times S}=E_{x,1}\supsetneq E_{x,2}\supsetneq \cdots \supsetneq E_{x,{l_x}} \supsetneq E_{x,l_x+1}=0$$
by sub-vector bundles over $\{x\}\times S$ and a system of real weights $0\le \alpha_{x,1}<\cdots < \alpha_{x,l_x}<1$.

We call parabolic type of $(E,E_\bullet)$ to the system of weights $\alpha=\{\alpha_{x,i}\}$ together with the set of ranks $\overline{r}=\{r_{x,i}\}$, $r_{x,i}=\rk (E|_{\{x\}\times S}/E_{x,i})$. A parabolic structure is said to be full flag if $l_x=\rk(E|_x)$ for every parabolic point.

\end{definition}

Providing such a filtration on the fibers $E|_{\{x\}\times S}$ is equivalent to giving a weighted filtration of $E$ by subsheaves of the form
$$E=E_x^1 \supsetneq E_x^2 \supsetneq \cdots \supsetneq E_x^{l_x} \supsetneq E_x^{l_x+1}=E(-\{x\}\times S)$$
where for every $x\in D$ and every $i=1,\ldots,l_x$, $E_x^i$ is the sheaf fitting in the following short exact sequence
$$0\longrightarrow E_x^i \longrightarrow E \longrightarrow E|_{\{x\}\times S}/E_{x,i}\longrightarrow 0$$

Equivalently \cite[Definition 1.2]{MY92} we can codify the parabolic structure of a parabolic vector bundle over each parabolic point $x\in D$ as a left continuous real decreasing filtration of sub-sheafs $E_{x,\alpha}$ of $E$ such that
\begin{enumerate}
\item For every $x\in D$ and every $\alpha\in \RR$, $E_{x,\alpha}$ is coherent and flat over $S$.
\item $E_{x,0}=E$ 
\item For every $\alpha \in \RR$, $E_{x,\alpha+1}=E_{x,\alpha}(-\{x\}\times S)$
\end{enumerate}

\begin{definition}[Parabolic $\Lambda$-module]
Let $\Lambda$ be a sheaf of rings of differential operators over $X=C\times S$ such that $\Lambda|_{\overline{D}}$ is a locally free $\SO_{\overline{D}}$-module. A parabolic $\Lambda$-module over $X$ is a locally free $\Lambda$-module $E$ over $X$ flat over $S$ together with a weighted flag on $E|_{\{x\}\times S}$ for each $x\in D$ called parabolic structure, i.e., a filtration
$$E|_{\{x\}\times S}=E_{x,1}\supsetneq E_{x,2}\supsetneq \cdots \supsetneq E_{x,{l_x}} \supsetneq E_{x,l_x+1}=0$$
by sub-vector bundles over $\{x\}\times S$ and a system of real weights $0\le \alpha_{x,1}<\cdots < \alpha_{x,l_x}<1$, such that for every $x\in D$ the filtration $E_{x,i}$ is compatible with the $\Lambda$-module structure in the following way. For each $x\in D$ let
$$E=E_x^1 \supsetneq E_x^2 \supsetneq \cdots \supsetneq E_x^{l_x} \supsetneq E_x^{l_x+1}=E(-\{x\}\times S)$$
be the induced filtration of $E$ by subsheaves given by
$$0\longrightarrow E_x^i \longrightarrow E \longrightarrow E|_{\{x\}\times S}/E_{x,i}\longrightarrow 0$$
Then the image of $\Lambda \otimes E_x^i$ under the morphism $\Lambda\otimes E\longrightarrow E$ lies in $E_x^i$ for every $i=1,\ldots,l_x+1$.

If $f:T\to S$ is any $S$-scheme, a family of parabolic $\Lambda$-modules over $X$ parametrized by $T$, is a parabolic $f^*\Lambda$-module $E$ over $C\times T$.
\end{definition}

If $(E,E_\bullet)$ is a parabolic $\Lambda$-module and $F\subseteq E$ is a vector bundle preserved by $\Lambda$, then the parabolic structure $E_\bullet$ induces a structure of parabolic $\Lambda$-module on $F$, taking the filtration
$$F_{x,\bullet}=E_{x,i}\cap F|_{\{x\}\times S}$$
for every $x\in D$. As $E_{x,1}=E|_{\{x\}\times S}$ and $F\subseteq E$, it is clear that $F_\bullet=\{F_{x,i}\}$ defines a parabolic structure on $F$. Moreover, $E_{x,i}$ and $F$ are preserved by $\Lambda$, so $F_\bullet$ is preserved by $\Lambda$ and $(F,F_\bullet)$ is a parabolic sub-$\Lambda$-module.

We will introduce some notation for the basic numerical invariants of a parabolic $\Lambda$-module. Let $E$ be a coherent sheaf over $X$. The Hilbert polynomial of $E$ is $P_E(n)= \chi(E(n))$. By Riemann-Roch theorem, if $X=C\times S$
$$P_E(n)=\deg(E) + \rk(E) (n+1-g)$$.

\begin{definition}
We define the parabolic degree of a parabolic $\Lambda$-module $(E,E_\bullet)$ as the parabolic degree of the underlying parabolic vector bundle
$$\pdeg(E,E_\bullet) := \deg(E) + \sum_{x\in D} \sum_{i=1}^{l_x} \alpha_{x,i} \left (\rk(E_{x,i}) - \rk(E_{x,i+1}) \right)$$
We will call the last summand of the previous expression the parabolic weight of $(E,E_\bullet)$,
$$\owt(E,E_\bullet)=\pdeg(E,E_\bullet)-\deg(E) = \sum_{x\in D} \sum_{i=1}^{l_x} \alpha_{x,i} \left (\rk(E_{x,i}) - \rk(E_{x,i+1}) \right)$$
Moreover, we will write for each $x\in D$
$$\owt_x(E,E_\bullet)=\sum_{i=1}^{l_x} \alpha_{x,i} \left (\rk(E_{x,i}) - \rk(E_{x,i+1}) \right )$$
so $\owt(E,E_\bullet)=\sum_{x\in D} \owt_x(E,E_\bullet)$.
In order to simplify the notation, if the parabolic structure is clear from the context, we may write $\pdeg(E)$, $\owt(E)$ and $\owt_x(E)$ to denote the parabolic degree, weight and weight at a point respectively.
\end{definition}

\begin{definition}
We define the parabolic slope of $(E,E_\bullet)$ as
$$\pmu(E)=\frac{\pdeg(E)}{\rk(E)}=\frac{\deg(E) + \owt(E)}{\rk(E)}$$
in order to simplify the notation in subsequent sections, we will write $\eta(E)=\frac{\owt(E)}{\rk(E)}$.

We also define the parabolic Euler characteristic of $(E,E_\bullet)$ as
$$\pchi(E)=\chi(E)+\owt(E)$$
The polynomial $\parP_E(m):=\pchi(E(m))$ is called the parabolic Hilbert polynomial of $(E,E_\bullet)$. Clearly, we can express the polynomial in terms of the Hilbert polynomial of the underlying sheaf $E$
$$\parP_E(m) = P_E(m)+\owt(E)$$
\end{definition}

\begin{definition}[Slope stability for parabolic $\Lambda$-modules]
A parabolic $\Lambda$-module $E$ over $C$ is said to be (semi-)stable if for every sub-$\Lambda$-module $F$ with the induced parabolic structure and $0<\rk(F)<\rk(E)$
$$\pmu(F) (\le) < \pmu(E)$$
\end{definition}

Let $p,q\in \RR[x]$. By $p (\preceq) \prec q$, we mean that there exists an integer $M$ such that for every $m\ge M$
$$p(m) (\le) < q(m)$$

\begin{lemma}[Gieseker stability for parabolic $\Lambda$-modules]
A parabolic $\Lambda$-module over $C$ is (semi-)stable if and only if for every sub-$\Lambda$-module $F$ with $0<\rk(F)<\rk(E)$ and the induced parabolic structure
$$\frac{\parP_F}{\rk(F)} (\preceq) \prec \frac{\parP_E}{\rk(E)}$$
\end{lemma}

\begin{proof}
By Riemann-Roch Theorem, for every $m$
$$\parP_E(m)=P_E(m)+\owt(E) =\chi(E(m)) + \owt(E) = \deg(E) + \rk(E)(m+1-g) + \owt(E)$$
Therefore $\frac{\parP_F(m)}{\rk(F)} (\le)< \frac{\parP_E(m)}{\rk(E)}$ for every big enough $m$ (and, in fact, for any $m$) if and only if
\begin{multline*}
\pmu(F) + m+1-g = \frac{\deg(F)+\owt(F)}{\rk(F)}+m+1-g\\
(\le) < \frac{\deg(E) + \owt(E)}{\rk(E)}+m+1-g = \pmu(E)+m+1-g
\end{multline*}
and this is equivalent to $\pmu(F) (\le) < \pmu(E)$.
\end{proof}

\begin{lemma}
\label{lemma:maxinduced}
Let $(E,E_\bullet)$ be a parabolic vector bundle and let $(F,F_\bullet)$ be a parabolic subsheaf such that $\pmu(E,E_\bullet)=\pmu(F,F_\bullet)$. Then $F$ has the induced parabolic, structure, i.e., $F_\bullet= E_\bullet \cap F$
\end{lemma}

\begin{proof}
As $(F,F_\bullet)\subseteq (E,E_\bullet)$, then for every $x\in D$ and $i=1,\ldots,l_x$ we have $F_{x,i}\subseteq E_{x,i}\cap F|_x$. We can rewrite the expression of the parabolic weight of $F$ as
\begin{multline*}
\owt_x(F,F_\bullet)=\sum_{i=1}^{l_x} \alpha_{x,i} \left( \dim(F_{x,i})-\dim(F_{x,i+1})\right)\\
=\sum_{i=2}^{l_x} \dim(F_{x,i})\left(\alpha_{x,i}-\alpha_{x,i-1}\right) + \alpha_1 \dim(F_{x,1})=\sum_{i=2}^{l_x} \dim(F_{x,i})\left(\alpha_{x,i}-\alpha_{x,i-1}\right) + \alpha_1 \dim(F|_x)\\
\le \sum_{i=2}^{l_x} \dim(F|_x \cap E_{x,i})\left(\alpha_{x,i}-\alpha_{x,i-1}\right) + \alpha_1 \dim(F|_x) = \owt_x(F,E_\bullet\cap F)
\end{multline*}
Therefore
$$\pmu(F,F_\bullet) = \mu(F)+\sum_{x\in D} \owt_x(F,F_\bullet) \le \mu(F) + \sum_{x\in D}\owt_x(F,E_\bullet\cap F)=\pmu(F,E_\bullet\cap F)$$
As the parabolic weights are strictly increasing, $\alpha_{x,i}-\alpha_{x,i-1}>0$, the previous inequalities only become equalities when $\dim(F_{x,i})=\dim(E_{x,i}\cap F|_x)$ for all $x\in D$ and all $i=1,\ldots,l_x$.
\end{proof}

\begin{lemma}
Let $(E,E_\bullet)$ be a parabolic sheaf and let $T$ be a torsion subsheaf of $E$. Let $(\overline{E},\overline{E}_\bullet)$ be the sheaf $\overline{E}=E/T$ with the induced parabolic structure $\overline{E}_{x,i}=E_{x,i}/(E_{x,i}\cap T|_x)$. Then
$$\pmu(E,E\bullet)\ge \pmu(\overline{E},\overline{E}_\bullet)$$
and for every $m\in \ZZ$
$$\frac{h^0(C,E(m))+\owt(E)}{\rk(E)}\ge \frac{h^0(C,\overline{E}(m))+\owt(\overline{E})}{\rk(\overline{E})}$$
\end{lemma}

\begin{proof} 
We have a short exact sequence
$$0\longrightarrow T \longrightarrow E\longrightarrow \overline{E} \longrightarrow 0$$
so
$$\deg(\overline{E})=\deg(E)-\deg(T)$$
On the other hand, as torsion sheafs on a curve are supported in dimension $0$
$$\deg(T)=h^0(C,T)-h^1(C,T)=h^0(C,T)$$
So $\deg(\overline{E})=\deg(E)-h^0(C,T)$. Moreover, as $T$ is torsion, $\rk(\overline{E})=\rk(E)$.

Now let us consider the parabolic structure. For every $x\in D$ and $i=1,\ldots,l_x$ we have a short exact sequence
$$0\longrightarrow E_{x,i}\cap T|_x \longrightarrow E_{x,i} \longrightarrow \overline{E}_{x,i}=\frac{E_{x,i}}{E_{x,i}\cap T_x}=\frac{E_{x,i}+T|_x}{T|_x}\longrightarrow 0$$
Now, taking quotients yields
$$\dim\left (\frac{\overline{E}|_x}{\overline{E}_{x,i}}\right) = \dim \left( \frac{E|_x/T|_x}{(E_{x,i}+T|_x)/T|_x}\right) = \dim\left(\frac{E|_x}{E_{x,i}+T|_x}\right) \ge \dim \left ( \frac{E|_x}{E_{x,i}}\right)-\dim T|_x$$
As $E$ and $\overline{E}$ have the same rank yields
$$\dim(\overline{E}_{x,i})\le \dim(E_{x,i})+\dim(T|_x)=\dim(E_{x,i})+h^0(x,T|_x)$$
Substituting in the weight formula we obtain
\begin{multline*}
\owt_x(\overline{E})=\sum_{i=2}^{l_x} \dim(\overline{E}_{x,i})(\alpha_{x,i}-\alpha_{x,i-1})+\alpha_{x,1}\dim(\overline{E}_{x,1})\\
\le \sum_{i=2}^{l_x} \dim(E_{x,i})(\alpha_{x,i}-\alpha_{x,i-1})+\alpha_{x,1}\dim(E_{x,1}) + h^0(x,T_x)\left( \alpha_{x,1}+\sum_{i=2}^{l_x}(\alpha_{x,i}-\alpha_{x,i-1}) \right)\\
= \owt_x(E)+h^0(x,T|_x)\alpha_{x,l_x} \le \owt_x(E)+h^0(x,T|_x)
\end{multline*}
and equality is only obtained if $h^0(x,T|_x)=0$. Adding up and taking into account that $h^0(C,T)\ge \sum_{x\in D} h^0(x,T|_x)$ yields
\begin{multline*}
\pmu(\overline{E}) =\frac{\deg(\overline{E})+\sum_{x\in D} \owt_x(\overline{E})}{\rk(\overline{E})} = \frac{\deg(E)-h^0(C,T)+\sum_{x\in D} \owt_x(\overline{E})}{\rk(E)}\\
\le \frac{\deg(E)+\sum_{x\in D} (\owt_x(\overline{E})- h^0(x,T|_x))}{\rk(E)} \le  \frac{\deg(E)+\sum_{x\in D}\owt_x(E)}{\rk(E)}=\pmu(E)
\end{multline*}
With regards to the second part of the lemma, from the short exact sequence
$$0\longrightarrow T(m) \longrightarrow E(m)\longrightarrow \overline{E}(m) \longrightarrow 0$$
we obtain a long exact sequence
$$0\longrightarrow H^0(C,T(m)) \longrightarrow H^0(C,E(m)) \longrightarrow H^0(C,\overline{E}(m)) \longrightarrow H^1(C,T(m))$$
As $T$ is supported in dimension $0$, we have $H^1(C,T(m))=0$ and $h^0(C,T(m))=h^0(C,T)$, so
$$h^0(C,\overline{E}(m))=h^0(C,E(m))-h^0(C,T(m))=h^0(C,E(m))-h^0(C,T)$$
Now we can repeat the previous argument and we obtain the desired inequality.

\end{proof}

\begin{corollary}
\label{cor:parabolicSaturation}
Let $(E,E_\bullet)$ be a parabolic sheaf and let $(F,F_\bullet)\subseteq (E,E_\bullet)$ be a parabolic subsheaf. Let $(F^{\op{sat}},F^{\op{sat}}_\bullet)$ be the saturation of $F$ in $E$ with the induced parabolic structure from $(E,E_\bullet)$. Then
$$\pmu(F,F_\bullet)\le \pmu(F^{\op{sat}},F^{\op{sat}}_\bullet)$$
If moreover if for some $m\in \ZZ$ we have $h^1(C,F(m))=0$ then
$$\frac{h^0(C,F(m))+\owt(F)}{\rk(F)} \le \frac{h^0(C,F^{\op{sat}}(m))+\owt(F^{\op{sat}})}{\rk(F^{\op{sat}})}$$
\end{corollary}

\begin{proof}
By Lemma \ref{lemma:maxinduced} we may assume without loss of generality that $F_\bullet$ is the induced parabolic structure. For the first part of the corollary, let $(Q,Q_\bullet)$ be the sheaf $E/F$ with the induced parabolic structure. Let $(T,T_\bullet)$ be its torsion with the induced parabolic structure and let $(\overline{Q},\overline{Q}_\bullet)$ be the torsion free sheaf $Q=(E/F)/T$ with the induced quotient parabolic structure. Then we have the following commutative diagram of parabolic sheaves
\begin{eqnarray*}
\xymatrix{
& & & 0 \ar[d] & \\
& 0 \ar[d] && (T,T_\bullet) \ar[d] & \\
0 \ar[r] & (F,F_\bullet) \ar[r] \ar[d] & (E,E_\bullet) \ar[r] \ar@{=}[d] & (Q,Q_\bullet) \ar[r] \ar[d]& 0\\
0 \ar[r] & (\overline{F},\overline{F}_\bullet) \ar[r] \ar[d] & (E,E_\bullet) \ar[r] & (\overline{Q},\overline{Q}_\bullet) \ar[r] \ar[d] & 0\\
& (T,T_\bullet) \ar[d] & & 0 &\\
& 0 &&&
}
\end{eqnarray*}
Where the two rows and columns are exact, so we have
\begin{eqnarray*}
\pdeg(F) = \pdeg(E)-\pdeg(Q)\\
\pdeg(\overline{F})=\pdeg(E)-\pdeg(\overline{Q})
\end{eqnarray*}
On the other hand, by the previous lemma, we know that $\pdeg(Q)\ge \pdeg(\overline{Q})$. Substituting yields
$$\pdeg(\overline{F})=\pdeg(F)+\pdeg(Q)-\pdeg(\overline{Q}) \ge \pdeg(F)$$
As $\rk(F)=\rk(\overline{F})$ we obtain $\pmu(F)\le \pmu(\overline{F})$.

For the second part of the lemma, observe that from the short exact sequence
$$0\longrightarrow F(m) \longrightarrow \overline{F}(m) \longrightarrow T(m)\longrightarrow 0$$
we obtain the long exact sequence
$$0\longrightarrow H^0(C,F(m))\longrightarrow H^0(C,\overline{F}(m)) \longrightarrow H^0(C,T(m)) \longrightarrow H^1(C,F(m))=0$$
Therefore
$$h^0(C,\overline{F}(m))=h^0(C,F(m))+h^0(C,T(m))=h^0(C,F(m))+h^0(C,T)$$
On the other hand, consider the following commutative diagram of sheaves, where the rows and columns are exact
\begin{eqnarray*}
\xymatrix{
& 0\ar[d] & 0 \ar[d] & 0 \ar[d] & \\
0 \ar[r] & F_x^i\ar[r] \ar[d] & \overline{F}_x^i \ar[r] \ar[d] & T_x^i \ar[r] \ar[d] & 0\\
0 \ar[r] & F \ar[r] \ar[d] & \overline{F} \ar[r] \ar[d] & T \ar[r] \ar[d] & 0\\
& F|_x/F_{x,i} \ar[d] & \overline{F}|_x/\overline{F}_{x,i} \ar[d] & T|_x/T_{x,i} \ar[d] &\\
& 0 & 0 & 0 &
}
\end{eqnarray*}
Then by the snake lemma we obtain
$$0\longrightarrow F|_x/F_{x,i} \longrightarrow \overline{F}|_x/\overline{F}_{x,i} \longrightarrow  T|_x/T_{x,i} \longrightarrow 0$$
As $\dim(\overline{F}|_x)=\dim(F|_x)$, yields 
$$\dim (\overline{F}_{x,i}) = \dim(F_{x,i})-\dim(T|_x)+\dim(T_{x,i})\ge \dim(F_{x,i})-h^0(x,T|_x)$$
Now we can proceed as in the second part of the previous Lemma and the desired inequality follows.
\end{proof}

We provide some insight on the structure of the sub-sheafs of a parabolic $\Lambda$-module. First of all, the following Lemma allows us to construct saturated parabolic subsheafs of a parabolic $\Lambda$-module which are preserved by $\Lambda$ from any subsheaf.

\begin{lemma}
\label{lemma:submoduleSaturation}
Let $(E,E_\bullet)$ be a parabolic $\Lambda$-module of rank $r$ on $X$. Suppose that $F\subset E$ is a subsheaf. Then the subbundle $\im(\Lambda_r\otimes F\to E)^{\op{sat}}$ with the induced parabolic structure is a parabolic sub-$\Lambda$-module.
\end{lemma}

\begin{proof}
By \cite[Lemma 3.2]{Si2}, $G:=\im(\Lambda_r\otimes F \to E)^{\op{sat}}$ is a subbundle of $E$ preserved by $\Lambda$. As $(E,E_\bullet)$ is a parabolic $\Lambda$-module, for every parabolic point $x\in D$, the filtration $E_{x\times S,i}$ is preserved by $\Lambda$. As $G$ is preserved by $\Lambda$, the induced filtration $G_{x,i} = G|_{x \times S} \cap E_{x,i}$ is preserved by $\Lambda$, so $(G,G_\bullet)$ is a parabolic sub-$\Lambda$-module of $(E,E_\bullet)$.
\end{proof}

\begin{theorem}[Harder-Narasimhan filtration]
\label{thm:HNfiiltration}
Suppose that $(E,E_\bullet)$ is a parabolic $\Lambda$-module on $C$. There is a unique filtration by parabolic sub-$\Lambda$-modules called the Harder-Narasimhan filtration
$$0=(E_0,E_{0,\bullet})\subsetneq (E_1,E_{1,\bullet}) \subsetneq \ldots \subsetneq (E_l,E_{l,\bullet}) = (E,E_\bullet)$$
such that the parabolic quotients $(E_i/E_{i-1}, E_{i,\bullet}/E_{i-1,\bullet})$ are semistable $\Lambda$-modules with strictly decreasing parabolic slopes.
\end{theorem}

\begin{proof}
The set of possible slopes of a subsheaf of $E$ is bounded from above. As the set of possible values of the weight of a parabolic sub-sheaf is finite, the set of possible parabolic slopes of parabolic sub-$\Lambda$-modules $(F,F_\bullet)$ is bounded from above. Let $\pmu_{\max}^\Lambda(E)$ be the maximum parabolic slope of a sub-$\Lambda$-module of $(E,E_\bullet)$. Let $(F,F_\bullet)$ be a sub-$\Lambda$-module such that $\pmu(F)=\pmu_{\max}^\Lambda(E)$. Repeating the argument in Lemma \ref{lemma:maxinduced} yields that as $F$ attains the maximum parabolic slope then $F$ must have the induced parabolic structure. Moreover, its saturation $(F^{\op{sat}},F^{\op{sat}}_\bullet)$ is preserved by $\Lambda$ and has a greater parabolic slope, so $F$ must be saturated. The rank of $F$ is bounded, so we can choose a saturated parabolic sub-$\Lambda$-module $(E_1,E_{1,\bullet})$ with $\pmu(E_1)=\pmu_{\max}^\Lambda(E)$ and maximum rank among those satisfying that condition.

We take $(E_1,E_{1,\bullet})$ as the first step of the Harder-Narasimhan filtration and build the rest of it inductively by applying the previous method to $(E/E_1,E_\bullet/E_{1,\bullet})$.
\end{proof}

A completely analogous proof to the previous one gives us the following theorem.
\begin{theorem}[Jordan-H\"older filtration]
\label{thm:JHfiltration}
Let $(E,E_\bullet)$ be a semistable parabolic $\Lambda$-module on $C$ over $\CC$. There is a unique filtration by sub-$\Lambda$-modules called the Jordan-H\"older filtration
$$0=(E_0,E_{0,\bullet})\subsetneq (E_1,E_{1,\bullet}) \subsetneq \ldots \subsetneq (E_l,E_{l,\bullet}) = (E,E_\bullet)$$
such that the parabolic quotients $(E_i/E_{i-1}, E_{i,\bullet}/E_{i-1,\bullet})$ are stable $\Lambda$-modules with strictly decreasing parabolic slopes.
\end{theorem}

We say that two semistable parabolic $\Lambda$-modules $(E,E_\bullet)$ and $(E',E'_\bullet)$ are S-equivalent if $\Gr(E,E_\bullet)\cong \Gr(E',E'_\bullet)$, i.e., if they have isomorphic Jordan-H\"older filtrations.

Let $S$ be a complex scheme and let $T$ be a scheme over $S$. We denote by $X_T=X\times_S T$ the base change of $X$ to $T$ and by $\Lambda_T$ the base change of $\Lambda$ to $T$. By \cite[Lemma 2.6]{Si2} it is a sheaf of rings of differential operators on $X_T$. In particular, if $\Spec(\CC) \cong s\to S$ is any geometric point, we denote by $X_s$ the fiber of $X$ over $s$ and by $\Lambda_s$ the base change of $\Lambda$ to $s$, which it is a sheaf of rings of differential operators on $X_s$.

\begin{definition}
A parabolic $\Lambda$-module $(E,E_\bullet)$ on $X=C\times S$ is (semi-)stable if the restrictions $(E|_{X_s},E_\bullet|_{X_s})$ to the geometric fibers $X_s$ are (semi-)stable  parabolic $\Lambda_s$-modules for every geometric point $s$ of $S$, all of them with the same Hilbert polynomial and parabolic type.
\end{definition}

Finally, we recall the following notion of ``almost'' stability due to Maruyama \cite{MaruyamaBound}

\begin{definition}
A coherent sheaf $E$ on $X$ is said to be of type $b$, for some $b\in \RR$ if for every subsheaf $F\subsetneq E$,
$$\mu(F)\le \mu(E)+b$$
\end{definition}

\section{Boundedness theorems}
\label{section:boundnessTheorems}

The main result proven in this section is the boundedness of the family of semistable parabolic $\Lambda$-modules over $C\times S$ with fixed Hilbert polynomial and parabolic type. In order to do so, we prove that every semistable parabolic $\Lambda$-module is of type $b$ for a certain uniform $b$. Then we use Simpson's theorems on Mumford-Castelnuovo regularity for bounded families of sheafs to provide uniform bounds for the regularity of semistable parabolic $\Lambda$-modules. Additionally, we find numerical bounds for the number of sections of twists of subsheafs of semistable parabolic $\Lambda$-modules. Finally, we obtain some sharper inequalities for the Hilbert polynomial of certain subsheafs of a semistable parabolic $\Lambda$-module.

Before introducing the main boundedness theorem, we shall prove two previous technical lemmas. 

\begin{lemma}
\label{lemma:minsolope}
Let $(E_1,E_{1,\bullet})$ and $(E_2,E_{2,\bullet})$ be parabolic vector bundles over $X$. For every parabolic vector bundle $(E,E_\bullet)$, let $\pmu_{\min}(E,E_\bullet)$ denote the minimum parabolic slope of a parabolic quotient of $E$. Then
$$\pmu_{\min}(E_1\oplus E_2,E_{1,\bullet}\oplus E_{2,\bullet}) = \min(\pmu_{\min}(E_1,E_{1,\bullet}), \pmu_{\min}(E_2,E_{2,\bullet}))$$
\end{lemma}

\begin{proof}
Let $\pi_i$ be the canonical projection of $E_1\oplus E_2$ to $E_i$. As every quotient of $E_i$ is a quotient of $E_1\oplus E_2$, if $(F_i,F_{i,\bullet})$ is a parabolic quotient of $(E_i,E_{i,\bullet})$ such that $\pmu(F_i)=\pmu_{\min}(E_i)$, then
$$\pmu_{\min}(E_1\oplus E_2,E_{1,\bullet}\oplus E_{2,\bullet}) \le \pmu(F_i,F_{i,\bullet}) =\pmu_{\min}(E_i,E_{i,\bullet})$$
Therefore 
$$\pmu_{\min}(E_1\oplus E_2,E_{1,\bullet}\oplus E_{2,\bullet}) \le \min(\pmu_{\min}(E_1,E_{1,\bullet}), \pmu_{\min}(E_2,E_{2,\bullet}))$$
Let us prove that the opposite inequality holds. Let $f:E_1\oplus E_2 \twoheadrightarrow F$ be a parabolic quotient such that $\pmu(F)=\pmu_{\min}(E_1\oplus E_2)$. By Lemma \ref{lemma:maxinduced} $F$ has the induced parabolic structure $F_\bullet=f(E_{1,\bullet}\oplus E_{2,\bullet})$. Consider the following exact commutative diagram of parabolic sheafs with the induced parabolic structures.
\begin{eqnarray*}
\xymatrixcolsep{3pc}
\xymatrix{
0 \ar[r] & E_1 \ar[r]^{i_1} \ar@{->>}[d]^{f} & E_1\oplus E_2 \ar[r]^{\pi_2} \ar@{->>}[d]^{f} & E_2 \ar[r] & 0\\
0 \ar[r] & f(E_1) \ar[r]^{i}  & F \ar[r] & F/f(E_1) \ar[r] & 0
}
\end{eqnarray*}
We have
$$F/f(E_1)=\frac{f(E_1\oplus E_2)}{f(E_1)}=\frac{f(E_1)+f(E_2)}{f(E_1)}\cong \frac{f(E_2)}{f(E_1)\cap f(E_2)}$$
and for every $x\in D$ and every $i=1,\ldots,l_x$
$$\frac{F_{x,i}}{f(E_1|_x)\cap F_{x,i}}=\frac{f(E_{1,x,i}\oplus E_{2,x,i})}{f(E_1|_x)\cap F_{x,i}} = \frac{f(E_{1,x,i})+f(E_{2,x,i})}{f(E_1|_x)\cap F_{x,i}}=\frac{f(E_{2,x,i})}{f(E_1|_x)\cap f(E_{2,x,i})}$$
Therefore, $(F/f(E_1),F_\bullet/f(E_1))$ is a parabolic quotient of $(E_2,E_{2,\bullet})$. On the other hand for every $x\in D$ and $i=1,\ldots,l_x$
$$F_{x,i}\cap f(E_1)|_x = \left ( f(E_{1,x,i})+f(E_{2,x,i}) \right) \cap f(E_1)|_x = f(E_{1,x,i})+ f(E_{2,x,i})\cap f(E_1)|_x \supseteq f(E_{1,x,i})$$
so $f(E_1)$ with the induced parabolic structure by $(F,F_\bullet)$ is a parabolic quotient of $(E_1,E_{1,\bullet})$. Finally, the second row is exact, so
\begin{multline*}
\pmu(F,F_\bullet)\ge \min \left(\pmu(f(E_1),F_\bullet\cap f(E_1)),\pmu(F/f(E_1),F_\bullet/f(E_1))\right)\\
 \ge \min (\pmu_{\min}(E_1,E_{1,\bullet}) , \pmu_{\min}(E_2,E_{2,\bullet}))
 \end{multline*}
\end{proof}

\begin{corollary}
\label{cor:minslope}
If $E$ is a parabolic vector bundle over $X$ then for every finite dimensional complex vector space $V$
$$\pmu_{\min}(E)= \pmu_{\min} (V\otimes_{\CC} E)$$
\end{corollary}

\begin{proof}
Inductively apply the previous lemma taking $E_1=E^{\oplus n}$ and $E_2=E$ for $1\le n\le dim V-1$.
\end{proof}

\begin{lemma}
Let $(E,E_\bullet)$ be a semistable parabolic $\Lambda$-module, and let $(F,F_\bullet)$ be a parabolic sub-bundle of $(E,E_\bullet)$ with the induced parabolic structure. Let $(G_i,G_i,\bullet)$ denote the image of the morphism of parabolic sheaves $\Lambda_i \otimes F \to E$. Observe that as $F_\bullet=E_\bullet \cap F$, then
$$G_{i,\bullet}=\Lambda_i\cdot F_\bullet = \Lambda_i \cdot \left(E_\bullet\cap F\right) = \left( \Lambda_i \cdot E_\bullet \right) \cap G_i = E_\bullet \cap G_i$$
so $G$ has the induced parabolic structure. For $i=i,\ldots,r$, consider the quotient parabolic sheaf $Q_i=G_i/G_{i-i}$ with the induced parabolic structure. Then for $i=1,\ldots,r$ there exists a surjective morphism of parabolic sheaves
$$\varphi_i:\Gr_1(\Lambda) \otimes_{\SO_X} \left( Q_i, Q_{i,\bullet} \right) \twoheadrightarrow \left( Q_{i+1}, Q_{i+1,\bullet}\right)$$
and a surjective morphism of parabolic sheaves
$$\varphi_0:\Gr_1(\Lambda) \otimes_{\SO_X} (F,F_\bullet) \twoheadrightarrow \left ( Q_1, Q_{1,\bullet}\right)$$
\end{lemma}

\begin{proof}
By Lemma \ref{lemma:sdoSurj}, $\Lambda_1\cdot \Lambda_i=\Lambda_{i+1}$ for all $i$, so
$$\Lambda_1\cdot G_i = \Lambda_1 \cdot \Lambda_i \cdot F = \Lambda_{i+1} \cdot F = G_{i+1}$$
As the previous equation also holds for the corresponding parabolic filtrations, we obtain a surjective morphism of parabolic sheaves
$$\Lambda_1 \otimes (G_i,G_{i,\bullet}) \twoheadrightarrow (G_{i+1},G_{i+1,\bullet}) \twoheadrightarrow \left (Q_{i+1},Q_{i+1,\bullet}\right )$$
To prove the lemma it is enough to show that the previous morphism factors through the corresponding quotients. Let $U\subseteq X$ be any open subset. Let $\lambda_1,\lambda_2\in \Lambda_1(U)$ such that $\lambda_2-\lambda_1=x$, for some $x\in \Lambda_0(U)$. Let $v_1,v_2\in G_i(U)$ such that $v_2-v_1=w$ for some $w\in G_{i-1}(U)$. Then yields
$$\varphi_i([\lambda_2],[v_2])=[\lambda_2\cdot v_2] =\left [ (\lambda_1+x)\cdot (v_1+w) \right] = [\lambda_1\cdot v_1] + [x\cdot v_2 + \lambda_1 \cdot w]$$
By hypothesis
$$x\cdot v_2 \in \Lambda_0(U)\cdot G_i(U) = \Lambda_0(U)\cdot \Lambda_i(U) \cdot F(U) = \Lambda_i(U) \cdot F(U) = G_i(U)$$
$$\lambda_1 \cdot w \in \Lambda_1(U) \cdot G_{i-1}(U) = G_i(U)$$

Therefore, $x\cdot v_2+\lambda_1\cdot w \cong 0 \mod G_i(U)$ and we get that $\varphi_i([\lambda_2],[v_2])=\varphi_i([\lambda_1],[v_1])$. The given argument also holds for any of the steps of the parabolic filtration, so we get a morphism of parabolic sheafs.
The proof of the second part is immediate from the previous computation taking into account that $G_0=\Lambda_0\cdot F = F$.
\end{proof}

\begin{lemma}
\label{lemma:LambdaModBoundC}
The set of semistable parabolic $\Lambda$-modules over $C$ with a fixed Hilbert polynomial $P$ and fixed parabolic type is bounded.
\end{lemma}

\begin{proof}
Let $(E,E_\bullet)$ be a semistable parabolic $\Lambda$-module, and let $(F,F_\bullet)$ be a parabolic subsheaf of maximum parabolic slope. By Lemma \ref{lemma:maxinduced}, $F_\bullet$ is the induced parabolic structure. Let $r$ be the rank of $E$. As in the previous Lemma, let $(G_i,G_{i,\bullet})$ be the image of $\Lambda_i\otimes F \to E$, for $i=1,\ldots,r$. Let us denote by $(G,G_\bullet)$ the saturation of $(G_r,G_{r,\bullet})$. By Lemma \ref{lemma:submoduleSaturation}, $(G,G_\bullet)$ is a parabolic sub-$\Lambda$-module, so
$$\pmu(G_r)\le \pmu(G) \le \pmu(E)$$
By the previous Lemma, for every $i=1,\ldots,r-1$ there exists a surjection of parabolic sheaves
$$\Gr_1(\Lambda) \otimes_{\SO_X} \left( Q_i, Q_{i,\bullet} \right) \twoheadrightarrow \left( Q_{i+1}, Q_{i+1,\bullet} \right)$$
By Serre's vanishing Lemma, there exists an $m\in \ZZ$ such that $\Gr_1(\Lambda)\otimes \SO_X(m)$ is generated by global sections. Let $V=H^0(\Gr_1(\Lambda)\otimes \SO_X(m))$. Then we have a surjection
\begin{equation}
\label{eq:LambdaModBoundC1}
V\otimes_{\CC}\SO_{X}(-m) \otimes_{\SO_X} \left( G_i / G_{i-1} \right) \twoheadrightarrow \left( G_{i+1} / G_i \right)
\end{equation}

Let $(R_i,R_{i,\bullet})$ be a parabolic quotient of $(G_i,G_{i,\bullet})$ such that $\pmu(R_i)=\pmu_{\min}(G_i)$. Then it has the induced parabolic structure and $(R_i,R_{i,\bullet})$ is a semistable parabolic sheaf, because a destabilizing sheaf for $(R_i,R_{i,\bullet})$ would lead to a parabolic quotient $(R_i,R_{i,\bullet}) \twoheadrightarrow (R_i',R_{i,\bullet}')$ with less parabolic slope. As any parabolic quotient of $(R_i,R_{i,\bullet})$ is a parabolic quotient of $(G_i,G_{i,\bullet})$, $(R_i',R_{i,\bullet'})$ would be a quotient of $(G_i,G_{i,\bullet})$ with less parabolic slope than $(R_i,R_{i,\bullet})$, contradicting the minimality assumption.

For each $0\le i<r$, if $(R_{i+1},R_{i+1,\bullet})$ has a nontrivial parabolic subsheaf $(H,H_\bullet)$ which is a parabolic quotient of $(G_i,G_{i,\bullet})$, then by semi-stability of $(R_{i+1},R_{i+1,\bullet})$, $\pmu(H)\le \pmu(R_{i+1})=\pmu_{\min}(G_{i+1})$. On the other hand, as $(H,H_\bullet)$ is a parabolic quotient of $(G_i,G_{i,\bullet})$,
$$\pmu_{\min}(G_i) \le \pmu(H) \le \pmu_{\min}(G_{i+1})$$
Otherwise, let $H=\im(G_i \hookrightarrow G_{i+1}  \twoheadrightarrow R_{i+1})$ with the induced parabolic structure. It is a parabolic subsheaf of $(R_{i+1},R_{i+1,\bullet})$ which is a quotient of $(G_i,G_{i,\bullet})$, so $H=0$. Then, $(R_{i+1},R_{i+1,\bullet})$ is a parabolic quotient of $(Q_{i+1},Q_{i+1,\bullet})$. If $i>0$, surjection \eqref{eq:LambdaModBoundC1} implies that $(R_{i+1},R_{i+1,\bullet})$ is a parabolic quotient of $V\otimes_{\CC}\SO_{X}(-m) \otimes_{\SO_X} \left( Q_i, Q_{i,\bullet} \right)$. Therefore, we get a parabolic quotient
$$V\otimes_{\CC} G_i \twoheadrightarrow V\otimes_\CC \left (Q_i, Q_{i,\bullet} \right) \twoheadrightarrow (R_{i+1}(m),R_{i+1,\bullet}(m))$$
By Corollary \ref{cor:minslope},
\begin{multline*}
\pmu_{\min}(G_i) = \pmu_{\min}(V\otimes_{\CC}G_i) \le \pmu(R_{i+1}(m))=\\
\pmu(R_{i+1})+m = \pmu_{\min}(G_{i+1})+m
\end{multline*}
For $i=0$, from the Lemma we get a surjection
$$V\otimes_{\CC} \SO_X(m) \otimes_{\SO_X} (F,F_\bullet) \twoheadrightarrow \left (Q_1,Q_{1,\bullet}\right)$$
so by the same argument $\pmu_{\min}(F)\le \pmu_{\min}(G_1)+m$.
Combining all the previous inequalities for $i=0,\ldots,r-1$, we conclude that $\pmu_{\min}(F)\le \pmu_{\min}(G_r)+rm \le \pmu(E)+rm$. As $(F,F_\bullet)$ is the parabolic subsheaf with maximum parabolic slope, every parabolic quotient of $(F,F_\bullet)$ must have bigger or equal parabolic slope, so $\pmu(F)\le \pmu_{\min}(F) \le \pmu(E)+rm$. Therefore, for every parabolic subsheaf $(F',F'_\bullet)\subseteq (E,E_\bullet)$, $\mu(F')+\owt(F')\le\pmu(F)\le \pmu(E)+rm=\mu(E)+\owt(E)+rm$. As $\owt(F')\ge 0$ for all parabolic sheafs, yields
$$\mu(F')\le \mu(E) + rm+\owt(E)-\owt(F') \le \mu(E)+rm+\owt(E)$$
Every subsheaf $F'\subseteq E$ can be given the induced parabolic structure, so the previous inequality proves that there exists a number $b\in \RR$ such that for every semistable parabolic $\Lambda$-module $(E,E_\bullet)$ over $X$ flat over $S$ of rank $r$ and the given parabolic type and every subsheaf $F'\subseteq E$, $\mu(F')\le \mu(E)+b$.
By \cite[Theorem 2.6]{MaruyamaBound}, the set of sheaves underlying a semistable parabolic $\Lambda$-module with Hilbert polynomial $P$ and the given parabolic structure is bounded. Given one such sheaf, the parabolic $\Lambda$-module structure is uniquely determined by a suitable element of $Fl(E|_{\{x\}\times S})$ for each $x\in D$, and a morphism $Hom(\Lambda_1 \otimes_{\SO_X} E,E)$, so the set of semistable parabolic $\Lambda$-modules is bounded.
\end{proof}

We can extend the previous lemma to the relative case.
\begin{lemma}
\label{lemma:LambdaModBound}
The set of semistable parabolic $\Lambda$-modules over $X=C\times S$ with a fixed Hilbert polynomial $P$ and fixed parabolic type is bounded.
\end{lemma}

\begin{proof}
By \cite[Proposition 3.5]{Si2}, the number $m\in \ZZ$ fixed in the previous proof can be chosen so that it works uniformly over all geometric points $s\in S$. Therefore, the upper bound on $\mu(E)-\mu(F)$ in the previous Lemma holds over every $s\in S$. Then the boundedness is a consequence of \cite[Theorem 1.1]{Si2}.
\end{proof}

Given a sheaf $F$ on $X$ flat over $S$, if $\pi:X\to S$ then we write
$$H^i(X/S,F)=R^i\pi_*F$$

\begin{corollary}
\label{cor:acyclicBound}
Let $X=C\times S$. There exists an integer $N$ depending only on $X$, $P$ and the parabolic weights such that for every $S$-scheme $S'$, every $n\ge N$ and every semistable parabolic $\Lambda$-module $(E,E_\bullet)$ over $X':=C\times S'$
\begin{enumerate}
\item For all $i>0$ $H^i(X'/S', E(n))=0$
\item The morphism
$$H^0(X'/S',E(n)) \otimes_{\SO_{S'}} \SO_{X'}(-n) \to E$$
is surjective.
\item $H^p(X'/S',E(n))$ is locally free over $S'$ and commutes with base change, in the sense that if $f:S''\to S'$ is an $S$-morphism, then
$$f^*H^0(X'/S',E(n)) \cong H^0(C\times S''/S'', f^*E(n))$$
\end{enumerate}
\end{corollary}

\begin{proof}
It holds as a consequence of the previous boundedness Lemma and \cite[Lemma 1.9]{Si2}.
\end{proof}

Now we will introduce some Lemmas providing bounds on the cohomology of semistable parabolic $\Lambda$-modules and its subsheafs.

\begin{lemma}
\label{lemma:sectionsBound}
There exists a number $B\in \RR$ depending on $e$, $r$ and $X$ such that if $E$ is a torsion free sheaf of type $e$ and rank $r$ on a fiber $X_s$, then for all $k$
$$h^0(X_s,E(k))\le r \left [(\mu(E)+k+B) \right]^+$$
where $[a]^+ =\max(a,0)$ for $a\in \RR$. 
\end{lemma}

\begin{proof}
Let $0=F_0\subsetneq F_1 \subsetneq \ldots \subsetneq F_l=E$ be the Harder-Narasimhan filtration of $E$ as a sheaf over $X_s$. As $E$ is of type $e$, then for every $i=1,\ldots,l$,
$$\mu \left( F_i/F_{i-1} \right) \le \mu(E) +e$$
On the other hand, as $F_i/F_{i-1}$ are semistable torsion free sheafs on $X_s$, by \cite[Lemma 1.7]{Si2}, there exists a number $B_{r_i}$ depending only on $r_i:=\rk(F_i/F_{i-1})$, such that
$$h^0\left(X_s, \left(F_i/F_{i-1}\right)(k) \right)\le rk\left(F_i/F_{i-1}\right)  \left [\mu \left(F_i/F_{i-1}\right) +k + B_{r_i} \right]^+$$
As the set of possible ranks for $F_i/F_{i-1}$ is bounded, taking $B'=\max_{k=1,\ldots,r}(B_k)$, yields
\begin{multline*}
h^0\left(X_s, \left(F_i/F_{i-1}\right)(k) \right)\le rk\left(F_i/F_{i-1}\right) \left [ \mu \left(F_i/F_{i-1}\right) +k + B' \right]^+\\
\le rk\left(F_i/F_{i-1}\right) \left[ \mu(E)+e+k+B'\right]^+
\end{multline*}
for every $i=1,\ldots,l$. Therefore
$$h^0(X_s,E(k))\le \sum_{i=1}^l h^0\left(X_s, \left(F_i/F_{i-1}\right)(k) \right) \le \sum_{i=1}^l rk\left(F_i/F_{i-1}\right) \left[ \mu(E)+e+k+B'\right]^+$$
As $\sum_{i=1}^l rk\left(F_i/F_{i-1}\right) =r$, taking $B=e+B'$ we obtain the desired bound.
\end{proof}

\begin{corollary}
\label{cor:sectionsBound}
There exists a number $B\in \RR$ depending on $\Lambda$, $r$, the parabolic type and $X$ such that if $E$ is a semistable parabolic $\Lambda$-module of rank $r$ and the given parabolic type on a geometric fiber $X_s$, then for all $k$
$$h^0(X_s,E(k))\le r \left [(\mu(E)+k+B) \right]^+$$
where $[a]^+ =\max(a,0)$ for $a\in \RR$. 
\end{corollary}

\begin{proof}
By Lemma \ref{lemma:LambdaModBound}, the set of coherent sheaves underlying a semistable parabolic $\Lambda$-module of the given parabolic type over a geometric fiber $X_s$ is contained in the set of coherent sheaves of type $b$, for some number $b$ depending only on $\Lambda$, $r$, the parabolic type and $X$. Then, the results yields as a consequence of the previous Lemma.
\end{proof}

Now, we provide an extension on Corollary \ref{cor:acyclicBound} allowing us find uniform bounds for the Serre vanishing theorem on destabilizing subsheafs of any given parabolic $\Lambda$-module.

\begin{lemma}
\label{lemma:uniformBound}
Let $b,e\in \RR$. There exists an integer $N$ depending only on $b$, $e$, $r$ and $X$ such that if $E$ is a torsion free sheaf of type $e$ and slope $\mu(E)\ge b$ on a geometric fiber $X_s$ then for every $n\ge N$
\begin{enumerate}
\item $h^1(X_s,E(n))=0$
\item $E(n)$ is generated by global sections.
\end{enumerate}
\end{lemma}

\begin{proof}
Let $E$ be a torsion free sheaf of type $e$ and slope $\mu(E)\ge b$ on $X_s$. Let us prove that $K\otimes E^\lor$ is of type $e'$ for some $e'$ depending only on $r$ and the genus $g$ of $X_s$, where $K$ is the canonical bundle on $X_s$. Let $F$ be a subsheaf of $K\otimes E^\lor$ of maximum slope. In particular, $F$ must be saturated. As both sheaves are torsion free, taking duals, $K^\lor \otimes F^\lor$ is a quotient of $E$. Let $G$ be the kernel of the quotient. 
\begin{eqnarray*}
\xymatrixcolsep{3pc}
\xymatrixrowsep{1pc}
\xymatrix{
0 \ar[r] & F \ar[r]  & K\otimes E^\lor & & \\
0 \ar[r] & G \ar[r] & E \ar[r] & K^\lor \otimes F^\lor \ar[r] & 0
}
\end{eqnarray*}
$E$ is of type $e$, so
$$\deg(G)\le \rk(G)\mu(E)+\rk(G)e$$
On the other hand, by additivity of the degree
$$\deg(K^\lor\otimes F^\lor)=\deg(E)-\deg(G)\ge \deg(E)-\rk(G)\mu(E)-\rk(G) e$$
Dividing by the rank and taking into account that $\rk(K^\lor\otimes F^\lor)=\rk(F)=\rk(E)-\rk(G)$ yields,
$$\mu(K^\lor \otimes F^\lor) \ge \frac{\rk(E)}{\rk(F)}\mu(E)-\frac{\rk(G)}{\rk(F)}\mu(E) - \frac{\rk(G)}{\rk(F)} e=\mu(E) - \frac{\rk(G)}{\rk(F)} e \ge \mu(E) - \rk(E)e$$
Computing the slope of the left hand side results in
$$\mu(E)-\rk(E) e\le \mu(K^\lor \otimes F^\lor) = \mu(F^\lor) - 2(g-1)=-\mu(F)-2(g-1)$$
Therefore
\begin{multline*}
\mu(F)\le -\mu(E) +\rk(E) e -2(g-1)\\
=\mu(K\otimes E^\lor)+\rk(E) e+2(g-1)-2(g-1)= \mu(K\otimes E^\lor)+\rk(E) e
\end{multline*}
Taking $e'=\rk(E) e$ we get that $K\otimes E^\lor$ is of type $e'$.

By Lemma \ref{lemma:sectionsBound} there exists a number $B$ depending only on $e'$, $r$ and $X$ such that for every $n$, $h^0(X_s,K\otimes E^\lor(-n))=0$ if $0\ge \mu(K\otimes E^\lor)-k+B=2\rk(E)(g-1)-\mu(E)-n+B$. As $\mu(E)\ge b$, if we fix an $N$ such that $0\ge 2\rk(E)(g-1)-b-N+B$, then for every $n\ge N$
$$0\ge 2\rk(E)(g-1)-b-N+B\ge 2\rk(E)(g-1)-\mu(E)-n+B$$
so for every $n\ge N$, $h^1(X_s,E(n))=h^0(X_s,K\otimes E^\lor(-n))=0$. This yields the desired bound for the first part of the lemma. As the dimension of $X_s$ is 1, (i) is equivalent to $E$ being $(N+1)$-regular in the sense of Mumford-Castelnuovo. By \cite[Lemma 1.7.2]{huybrechts}, $E(n)$ is generated by global sections for every $n\ge N+1$.
\end{proof}

\begin{corollary}
\label{cor:acyclicDestab}
There exists an integer $N$ depending on $\Lambda$, $P$, the parabolic type and $X$ such that for every $n\ge N$, and every parabolic $\Lambda$-module $(E,E_\bullet)$ over a geometric fiber $X_s$ with Hilbert polynomial $P$ and fixed parabolic type, if $(F,F_\bullet)$ is a parabolic sub-$\Lambda$-module of $(E,E_\bullet)$ with maximum slope then
\begin{enumerate}
\item $h^1(X_s,F(n))=0$
\item $F(n)$ is generated by global sections. In particular
$$H^0(X_s,F(n))\otimes \SO_{X_s}(-n) \longrightarrow F$$
is surjective.
\end{enumerate}
\end{corollary}

\begin{proof}
Let $r$ denote the rank of any (and therefore all) of the considered parabolic $\Lambda$-modules $E$. Let $(F,F_\bullet)$ a parabolic sub-$\Lambda$-module with maximum slope. Then it is semistable as a parabolic $\Lambda$-module. In particular, it is a torsion free sheaf of type $b$ for the constant $b$ given by Lemma \ref{lemma:LambdaModBound} such that $\pmu(F)\ge \pmu(E)$. The set of possible values of $\owt(F)$ is bounded, as
$$\owt(F) = \frac{\sum_{x\in D} \sum_{i\in \alpha_{F,x}} \alpha_{x,i}}{\rk(F)} \le \sum_{x\in D} \sum_{i=1}^{l_x} \alpha_{x,i}$$
where $\alpha_{F,x}$ is the set of indexes $i\in\{1,\ldots,l_x\}$ such that $E_{x,i}\cap F|_{\{x\}\times S} \ne E_{x,i+1} \cap F|_{\{x\} \times S}$. Calling $\owt_{\max}$ to the right hand side of the inequality, yields
$$\mu(F) \ge \pmu(E)-\owt(F) \ge \pmu(E)-\owt_{\max}$$
By the previous Lemma, there exists an integer $N_{\rk(F)}$ depending only on $b$, $\pmu(E)-\owt_{\max}$, $\rk(F)$ and $X$ such that for $n\ge N_{\rk(F)}$, $F(n)$ is acyclic and it is generated by global sections. As $0< \rk(F) \le r$,  it is enough to take $N=\max_{i=1}^r (N_i)$.
\end{proof}

Now we will provide some sharp inequalities for the Hilbert polynomials of subsheafs of a semistable parabolic $\Lambda$-module. Let $(E,E_\bullet)$ be any $\Lambda$-module over a geometric fiber $X_s$ and $(F,F_\bullet)$ be a subsheaf. If $(E,E_\bullet)$ is semistable and $(F,F_\bullet)$ is preserved by $\Lambda$, semi-stability condition implies that for every $n$
$$\frac{h^0(X_s,F(n))-h^1(X_s,F(n))}{\rk(F)}+\eta(F) =\frac{P_F(n)}{\rk(F)}+\eta(F) \le \frac{P_E(n)}{\rk(E)}+\eta(E)$$
We will prove that for big enough $n$ the previous inequality can be sharpened by removing the term $h^1(X_s,F(n))$. The equality case for the sharpened inequality will be analyzed.

Finally, we will prove an inequality implying that if $E(n)$ is generated by global sections and the previous inequality is strict then there exists a uniform lower bound for the difference between its right hand side and its left hand side.

\begin{lemma}
\label{lemma:sharpStabInequality}
There exists an integer $N$ such that for every $n\ge N$ if $(E,E_\bullet)$ is a semistable parabolic $\Lambda$-module with Hilbert polynomial $P$ and fixed parabolic type, then for every parabolic subsheaf $(F,F_\bullet)\subsetneq (E,E_\bullet)$ such that its saturation is a parabolic sub-$\Lambda$-module and every $n\ge N$
$$\frac{h^0(X_s,F(n))}{\rk(F)}+\eta(F) \le \frac{P_E(n)}{\rk(E)}+\eta(E)$$
Moreover, if equality holds for some $n\ge N$ then $F$ is saturated and we have $h^1(X_s,F(n))=0$.
\end{lemma}

\begin{proof}
Let $(F,F_\bullet)\subsetneq(E,E_\bullet)$ be a parabolic subsheaf. Without loss of generality we can assume that $F$ has the induced parabolic structure. Let $0=G_0\subsetneq G_1\subsetneq \ldots \subsetneq G_l=F$ be the Harder-Narasimhan filtration of $F$ as a subsheaf of $E$. For every $i=1,\ldots,l$, $G_i/G_{i-1}$ is a semistable sheaf, so, as before, there exist integers $B_{r_i}$ depending only on $r_i:=\rk\left(G_i/G_{i-1}\right)$ and $X$ such that for every $n$ and every $i=1,\ldots,l$
$$h^0\left(X_s,\left(G_i/G_{i-1}\right)(n)\right)\le \rk\left(G_i/G_{i-1}\right) \left[ \mu \left(G_i/G_{i-1}\right) + n + B_{r_i} \right ]^+$$
Let $B=\min_{k=1,\ldots,r} B_k$. Then for every $i$ yields
\begin{equation}
\label{eq:sharpStabInequality1}
h^0(X_s,F(n)) \le \sum_{i=1}^l h^0\left(X_s,\left(G_i/G_{i-1}\right)(n)\right)\le \sum_{i=1}^l r_i \left[ \mu \left(G_i/G_{i-1}\right) + n + B\right]^+
\end{equation}
As $(E,E_\bullet)$ is a semistable parabolic $\Lambda$-module, by Lemma \ref{lemma:LambdaModBound} there exists a number $b$ depending only on $\Lambda$, $r$, the parabolic type and $X$, such that $\mu\left(G_i/G_{i-1}\right) \le \mu(E)+b$. On the other hand, let $\nu(F)=\min_{i=1,\ldots,l}\left (\mu \left(G_i/G_{i-1} \right) \right) = \mu (G_1)$. Then, substituting the bounds in equation \eqref{eq:sharpStabInequality1} and taking into account that $r_1\ge 1$ and $\sum_{i=1}^l r_i=\rk(F)$ yields
$$h^0(X_s,F(n)) \le (\rk(F)-1) \left[ \mu(E)+b+n+B\right]^+ + \left[ \nu(F) + n + B \right]^+$$
Now, suppose that $\nu(F) \le \mu(E)-C$ for some $C\ge 0$. Then for $n\ge C-\mu(E)-B=N_1(C)$ we have
\begin{multline*}
h^0(X_s,F(n))\le (\rk(F)-1) (\mu(E)+b+n+B) + \mu(E)-C+ n + B =\\
 \rk(F)(n+\mu(E)+B) + (\rk(F)-1)b -C
\end{multline*}
Both $\rk(F)$ and $\eta(F)$ are bounded uniformly over each choice of $E$ and $F$. Therefore, there exists $C$ big enough so that for $n\ge C-\mu(E)-B=N_1(C)$
$$h^0(X_s,F(n))\le \rk(F)(n+\mu(E)+B) + (\rk(F)-1)b -C <\rk(F)(n+1-g+\pmu(E)-\eta(F))$$
Then for $n\ge N_1(C)$
$$\frac{h^0(X_s,F(n))}{\rk(F)} +\eta(F) < n+1-g+\pmu(E) = \frac{P_E(n)}{\rk(E)}+\eta(E)$$
Therefore, there exist positive numbers $C$ and $N_1=N_1(C)$ depending only on $\Lambda$, $r$, the parabolic type, $P$ and $X$ such that the Lemma holds for the given $N_1$ for every subsheaf $F$ such that $\nu(F)\le \mu(E)-C$. Let us suppose that $(F,F_\bullet)$ is a parabolic subsheaf whose saturation is a parabolic sub-$\Lambda$-module and such that $\nu(F) \ge \mu(E)-C$. As the slope is invariant under Jordan equivalence
$$\mu(F)=\mu\left(\bigoplus_{i=1}^r \left(G_i/G_{i-1}\right)\right) \ge \nu(F)\ge \mu(E)-C$$
Moreover, for every subsheaf $G\subsetneq F\subsetneq E$
$$\mu(G)\le \mu(E)+b\le \mu(F)+C+b$$
Therefore, $F$ is a torsion free sheaf of type $b+C$ with $\mu(F)\ge \mu(E)-C$. By Lemma \ref{lemma:uniformBound}, there exists an integer $N_2$ depending on $b+C$, $\mu(E)-C$ and $X$, i.e., on $\Lambda$, $P$, $X$ and the parabolic type, such that for every $n\ge N_2$, $h^1(X_s,F(n))=0$. By Corollary \ref{cor:parabolicSaturation} for every such sheaf $F$
$$\frac{h^0(X_s,F(n))}{\rk(F)} +\eta(F)\le \frac{h^0(X_s,F^{\op{sat}}(n))}{\rk(F^{\op{sat}})} +\eta(F^{\op{sat}})$$
so we may assume without loss of generality that $F$ is a saturated sub-$\Lambda$-module. As $(F,F_\bullet)$ is a parabolic sub-$\Lambda$-module of $(E,E_\bullet)$, by semi-stability, for $n\ge N_2$
$$\frac{h^0(F(n))}{\rk(F)} +\eta(F) = \frac{P_F(n)-h^1(F(n))}{\rk(F)}+\eta(F) = \frac{P_F(n)}{\rk(F)}+\eta(F)\le \frac{P_E(n)}{\rk(E)}+\eta(E)$$
Then, it is enough to pick $N=\max(N_1,N_2)$ to get the first part of the result for every parabolic sub-$\Lambda$-module. Now suppose that for some $n\ge N$
$$\frac{h^0(X_s,F(n))}{\rk(F)} +\eta(F)= \frac{h^0(X_s,F^{\op{sat}}(n))}{\rk(F^{\op{sat}})} +\eta(F^{\op{sat}})=\frac{P_E(n)}{\rk(E)}+\eta(E)$$
Then $h^0(X_s,F(n))=h^0(X_s,F^{\op{sat}}(n))$. By the choice of $C$ through the proof, the only option for equality to hold is that $\mu(F) \ge \mu(E)-C$, so $F^{\op{sat}}(n)$ is generated by global sections. Therefore, $F(n)$ is also generated by global sections and, hence, $F=F^{\op{sat}}$. By the choice of $N$, we know that
$$h^1(X_s,F(n))=h^1(X_s,F^{\op{sat}}(n))=0$$
\end{proof}

\begin{lemma}
\label{lemma:slopeDiffbound}
Let us fix a certain parabolic type. Then there exists a real number $\delta>0$ such that for every parabolic sheaf $(E,E_\bullet)$ of Hilbert polynomial $P$ of the given parabolic type, if $(F,F_\bullet)\subsetneq (E,E_\bullet)$ is a subsheaf such that
$$\frac{h^0(X_s,F(n))}{\rk(F)}+\eta(F) < \frac{h^0(X_s,E(n))}{\rk(E)}+\eta(E)$$
then
$$\frac{h^0(X_s,F(n))}{\rk(F)}+\eta(F)+ \delta \le  \frac{h^0(X_s,E(n))}{\rk(E)}+\eta(E)$$
\end{lemma}

\begin{proof}
The left and right side of the inequality are sums of rational numbers, so its difference is a positive rational number $p/q$, with $p,q>0$ coprime, whose denominator is at most the least common multiple of the denominators of all the summands appearing in the expression. Therefore
\begin{multline*}
\frac{h^0(X_s,E(n))}{\rk(E)}+\eta(E)-\frac{h^0(X_s,F(n))}{\rk(F)}-\eta(F)\\
\ge \frac{1}{\op{LCM}\{\rk(F)q_{x,i},\rk(E)q_{x,i}\}}\ge \frac{1}{r! \prod_{x\in D} \prod_{i=1}^{l_x} q_{x,i}}
\end{multline*}
where $\alpha_{x,i}=\frac{p_{x,i}}{q_{x,i}}$ with $p_{x,i}\in \ZZ^{\ge 0}$ and $q_{x,i}\in \ZZ^{>0}$.
\end{proof}

\section{Parametrizing scheme for parabolic $\Lambda$-modules}
\label{section:parameterSpace}

Given a scheme $X$ over $S$, a coherent sheaf $F$ over $X$ and a polynomial $P$, let $\Quot_{X/S}(F,P)$ denote the Quot scheme of quotients of $F$ over $X$ flat over $S$ with Hilbert polynomial $P$. It is a projective scheme representing the moduli functor $\mathfrak{Quot}_{X/S}(F,P):(\Sch_S)\longrightarrow (\Sets)$ that assigns each $f:T\to S$ the set of quotients $f^*F \twoheadrightarrow Q$ over $X\times_S T$ flat over $T$ with Hilbert polynomial $P$. Let $\mathfrak{Quot}_{X/S}^{\op{LF}}(F,P): (\Sch_S) \longrightarrow (\Sets)$ be the subfunctor of families of locally free quotients, and let $\Quot_{X/S}^{\op{LF}}(F,P)\subseteq \Quot_{X/S}(F,P)$ be the open subscheme representing such subfunctor.

\begin{lemma}
\label{lemma:pullbackQuot}
Let $E$ and $F$ be coherent sheafs over $X$ flat over $S$ such that there exists a surjective morphism $p:E\twoheadrightarrow F$. Then $p^*:\Quot_{X/S}(F,P)\to \Quot_{X/S}(E,P)$ is a closed embedding.
\end{lemma}

\begin{proof}

Let $K=\Ker(p)$. Let $f:T\to S$ and let $(G,\psi_G)\in \mathfrak{Quot}_{X/S}(E,P)(T)$, where $\psi_G:f^*E \twoheadrightarrow G$. As $F$ is flat over $S$, $\Ker(f^*E\to f^*F)=f^*K$. Then $\psi_G$ factors through the pullback of $F$ if and only if the image of $f^*K$ by the quotient $\psi_G:f^*E\twoheadrightarrow G$ is zero.
\begin{eqnarray*}
\xymatrixcolsep{1pc}
\xymatrixrowsep{1pc}
\xymatrix{
0 \ar[d] &&\\
f^*K \ar[d] \ar@{-->}[dr] &&\\
f^*E \ar[d] \ar[r] & G \ar[r]  & 0\\
f^*F \ar[d] \ar@{-->}[ur] & &\\
0 &&
}
\end{eqnarray*}

Let $(G_E,\psi_E)$ be the universal quotient of $\Quot_{X/S}(E,P)$ and $g:T\to \Quot_{X/S}(E,P)$ be the morphism corresponding to $(G,\psi_G)$. Then $(G,\psi_G)$ belongs to the image of $\Quot_{X/S}(F,P)$ if and only if the pullback of $\pi_X^*K \to G_E$ by $g$ is zero, where $\pi_X:X\times_S \Quot_{X/S}(E,P)\to X$ is the projection to the first factor. By \cite[Lemma 4.3]{Yokogawa93}, there is a closed subscheme $Z$ of $\Quot_{X/S}(E,P)$ such that the pullback is zero if and only if $g$ factors through $Z$. Therefore, the image of $\Quot_{X/S}(F,P)$ is closed.

Let $f':T'\to S$ be another $S$-scheme and let $\varphi:T'\to T$ be a morphism of $S$-schemes. Let us prove that the following diagram is commutative.
\begin{eqnarray*}
\xymatrixcolsep{3pc}
\xymatrixrowsep{2pc}
\xymatrix{
\mathfrak{Quot}_{X/S}(F,P)(T) \ar[d]^{p^*} \ar[r]^{\varphi^*} & \mathfrak{Quot}_{X/S}(F,P)(T') \ar[d]^{p^*}\\
\mathfrak{Quot}_{X/S}(E,P)(T)\ar[r]^{\varphi^*} & \mathfrak{Quot}_{X/S}(E,P)(T')
}
\end{eqnarray*}
An element of $\mathfrak{Quot}_{X/S}(F,P)(T)$ is given as a quotient $f^*F\twoheadrightarrow G$ over $X\times_S T$, which is a quotient
$$f^*E\twoheadrightarrow f^*F \twoheadrightarrow G$$
As the pullback is right exact, its image by $\varphi^*$ is a quotient
\begin{eqnarray*}
\xymatrixcolsep{2pc}
\xymatrixrowsep{1pc}
\xymatrix{
\varphi^*f^*E \ar@{->>}[r] \ar@{=}[d] & \varphi^*f^*F \ar@{->>}[r] \ar@{=}[d] & \varphi^*G\\
f'^*E  \ar@{->>}[r] & f'^*F &
}
\end{eqnarray*}
The pullback by $\varphi$ of the composition $f^*E\to f^*F \to G$ is the composition of the pullbacks $\varphi^*f^*E\to \varphi^*f^*F\to \varphi^*G$, so we get that $(p^*\circ \varphi^*)(G)=(\varphi^*\circ p^*)(G)$. Therefore, $p^*$ induces a natural transformation $\mathfrak{Quot}_{X/S}(F,P) \to \mathfrak{Quot}_{X/S}(E,P)$ and $p^*$ is a closed embedding. 
\end{proof}

Given schemes $X_i\to S$ for $i=1,\ldots, n$, let
$$\fprod{S}{i=1}{n}X_i$$
denote the fiber product of the $X_i$ over $S$.

\begin{lemma}
Let $P$ be a fixed Hilbert polynomial, with leading coefficient $r$, and let $\overline{r}=\{r_{x,i}\}$ for $x\in D$, $1<i\le l_x$ be integers. Let $F$ be a coherent sheaf over $X=C\times S$ flat over $S$ such that $F|_{\overline{D}}$ is locally free. Let $\mathfrak{FQuot}^{\op{LF}}_{X/S}(F,P,\overline{r})$ be the functor that associates each $S$-scheme $T$ the set of isomorphism classes of pairs $(E,E_\bullet)$ consisting on a locally free quotient sheaf $E$ of the pullback of $F$ over $X_T=C\times T$ flat over $T$ with Hilbert polynomial $P_E=P$ and a filtration by sub-bundles over $T$
$$E|_{\{x\}\times T} = E_{x,1} \supsetneq E_{x,2} \supsetneq \ldots \supsetneq E_{x,l_x}$$
for each $x\in D$ such that for each $1<i\le l_x$, $\rk(E|_{x\times T}/E_{x,i})=r_{x,i}$.
Then there is a closed subscheme $\op{FQuot}_{X/S}^{\op{LF}}(F,P,\overline{r})$ of
$$\Quot_{X/S}^{\op{LF}}(F,P)\times_S \fprod{S}{x\in D}{} \Grass(F|_{\{x\}\times S},r) \times_S \fprod{S}{x\in D}{} \fprod{S}{i=2}{l_x} \Grass(F|_{\{x\}\times S},r_{x,i})$$
over $\Quot_{X/S}^{\op{LF}}(F,P)$ representing $\mathfrak{FQuot}_{X/S}^{\op{LF}}(F,P,\overline{r})$.
\end{lemma}

\begin{proof}
Let $p:F\to E$ be the universal quotient of $Q:=\Quot_{X/S}^{\op{LF}}(F,P)$, and let $\pi_Q:Q\to S$. Suppose that $D$ consists on a single closed point $x\in C$. As $E$ is locally free and $F|_{\{x\}\times S}$ is locally free, $E|_{\{x\}\times Q}$ is a locally free quotient of $\pi_Q^*F|_{\{x\}\times Q}$ of rank $r$, so it represents a $Q$-point $e:Q\to \Grass(F|_{\{x\}\times S},r)$. Therefore, the graph of $e$ is a closed subscheme $\tilde{Q}$ of $\Quot_{X/S}^{\op{LF}}(F,P)\times_S \Grass(F|_{\{x\}\times S},r)$ over $Q$ corresponding to the family of pairs of a quotient sheaf $E$ and its restriction to $\{x\}\times S$.

Now, we will prove the claim in the case $D=\{x\}$ by induction on $l_x$.  We have proven the result for $l_x=1$. Suppose that it is true for filtrations of length $l_x-1$. Let $\overline{r}'=\{r_{x,3},\ldots,r_{x,l_x}\}$. Then there exists a closed subscheme
$$\op{FQuot}_{X/S}^{\op{LF}}(F,P,\overline{r}') \subset Q\times_S \Grass(F|_{\{x\}\times S},r) \times_S \fprod{S}{i=3}{l_x} \Grass(F|_{\{x\}\times S},r_{x,i})$$
over $Q$ representing $\mathfrak{FQuot}_{X/S}^{\op{LF}}(F,P,\overline{r}')$. Let $(E,\{E_{x,1},E_{x,3},E_{x,4},\ldots,E_{x,l_x}\})$ be the universal filtered quotient of $FQ'=\op{FQuot}_{X/S}^{\op{LF}}(F,P,\overline{r}')$.
 
Clearly, parameterizing filtrations $E|_{\{x\}\times T} = E_{x,1}\supsetneq \ldots \supsetneq E_{x,l_x}$ is the same as parameterizing the corresponding subsequent quotients
$$E|_{\{x\}\times T} \twoheadrightarrow E|_{\{x\}\times T}/E|_{x,l_x} \twoheadrightarrow E|_{\{x\}\times T}/E_{x,l_x-1} \twoheadrightarrow \cdots \twoheadrightarrow E|_{\{x\}\times T}/E_{x,2}$$
Therefore, to give a filtration $E|_{\{x\}\times T}=E_{x,1} \supsetneq \ldots \supsetneq E_{x,l_x}$ is the same as giving a filtration $E|_{\{x\}\times T}=E_{x,1} \supsetneq E_{x,3} \supsetneq \ldots \supsetneq E_{x,l_x}$ and a quotient $E|_{\{x\}\times T} / E|_{x,3}=E_{x,1}/E_{x,3} \twoheadrightarrow E|_{\{x\}\times T}/E_{x,2}=E_{x,1}/E_{x,2}$

Thus, the functor $\mathfrak{FQuot}_{X/S}^{\op{LF}}(F,P,\overline{r})$ is represented by
$$FQ=FQ'\times_{FQ'} \Grass(E_{x,1}/E_{x,3},r_{x,2})=\Grass(E_{x,1}/E_{x,3},r_{x,2})$$

Let us prove that this product embeds into the desired product of Grassmanians. Let $\pi_{FQ'}:FQ'\to S$. $E_{x,1}/E_{x,3}$ is a quotient of $\pi_{FQ'}^* F_{\{x\}\times S}$ over $FQ'$, so by previous lemma, $FQ$ is a closed subscheme of $\Grass(\pi_{FQ'}^*F|_{\{x\}\times S},r_{x,2})$ over $FQ'$. By definition of the Grassmanian functor and the base change formula yields
$$FQ \hookrightarrow \Grass(\pi_{FQ'}^*F|_{\{x\}\times S},r_{x,2}) \cong FQ' \times_S \Grass(F|_{\{x\}\times S},r_{x,2})$$
By induction hypothesis, there is a closed embedding over $Q$
$$FQ'\hookrightarrow Q\times_S \Grass(F|_{\{x\}\times S},r) \times_S \fprod{S}{i=3}{l_x} \Grass(F|_{\{x\}\times S},r_{x,i})$$
so there exists a closed embedding over $Q$
$$FQ\hookrightarrow Q\times_S \Grass(F|_{\{x\}\times S},r) \times_S \fprod{S}{i=2}{l_x} \Grass(F|_{\{x\}\times S},r_{x,i})$$
Finally, let $D=\{x_1,\ldots,x_M\}$ be any finite set of points in $C$. It is clear that
$$\op{FQuot}_{X/S}^{\op{LF}}(F,P,\{r_{x_j,i}\}) = \op{FQuot}_{X/S}^{\op{LF}}(F,P,\{r_{x_1,i}\}) \times_{Q} \cdots \times_{Q} \op{FQuot}_{X/S}^{\op{LF}}(F,P,\{r_{x_M,i}\})$$
is a closed subspace of
$$Q\times_S \fprod{S}{x\in D}{}\Grass(F|_{\{x\}\times S},r) \times_S \fprod{S}{x\in D}{}\fprod{S}{i=2}{l_x} \Grass(F|_{\{x\}\times S},r_{x,i})$$
over $Q$ which represents the functor $\mathfrak{FQuot}_{X/S}^{\op{LF}}(F,P,\{r_{x_j,i}\})$.
\end{proof}

\begin{corollary}
\label{cor:filteredQuot}
Let $F$ be a coherent sheaf over $X=C\times S$ such that $F|_{\overline{D}}$ is locally free and let $\overline{r}=\{r_{x,i}\}$ for $x\in D$, $1<i\le l_x$ be integers. Let $Q\to \Quot_{X/S}^{\op{LF}}(F,P)$ be any family of isomorphism classes of locally free quotient sheafs of $F$ on $X$ flat over $S$. Let $\mathfrak{FQ}(\overline{r})$ be the functor that associates each $S$-scheme $T$ the set of isomorphism classes of pairs $(E,E_\bullet)$ consisting on a quotient $\pi^*F\twoheadrightarrow E$ in $Q(T)$ and a filtration by sub-bundles over $T$
$$E|_{\{x\}\times T} = E_{x,1} \supsetneq E_{x,2} \supsetneq \ldots \supsetneq E_{x,l_x}$$
for each $x\in D$ such that for each $1<i\le l_x$, $\rk(E|_{\{x\}\times T})=r_{x,i}$. Then there is a closed subscheme $FQ(\overline{r})$ of
$$Q \times_S \fprod{S}{x\in D}{} \Grass(F|_{\{x\}\times S},r) \times_S \fprod{S}{x\in D}{} \fprod{S}{i=2}{l_x} \Grass(F|_{\{x\}\times S},r_{x,i})$$
over $Q$ representing $\mathfrak{FQ}(\overline{r})$.
\end{corollary}

\begin{proof}
The proof is completely analogous to the previous one changing $\Quot_{X/S}^{\op{LF}}(F,P)$ to the given family $Q$ and the universal quotient by its pullback to $Q$. 
\end{proof}

Grothendiek \cite{GrothQuot} proved that the quot scheme $\Quot_{X/S}(F,P)$ is a projective scheme over $S$ by constructing an explicit embedding into a certain Grassmanian. More precisely, he stated that there exists an integer $M$ such that for every $m\ge M$ there exists an embedding
$$\psi_m : \Quot_{X/S}(F,P) \hookrightarrow \Grass(H^0(X/S,F(m)),P(m))$$
defined in the following way. By Serre's vanishing theorem, there exists an $M$ such that for every $m\ge M$, $F(m)$ is generated by global sections and $H^0(X/S,F(m))$ is compatible with base change. Grothendieck proved that moreover $M$ can be chosen in a way that for any quotient
$$0 \to K_G \to f^*F \to  G \to 0$$
on $C\times T$, for any $T$-point of $\Quot_{X/S}(F,P)$ and any $f:T\to S$, $H^0(C\times T / T, G(m))$ is locally free of rank $P(m)$ and $H^1(C\times T/T,K_G(m))=0$. Then, tensoring the previous sequence by $\SO_{C\times T}(m)$ and taking the corresponding long exact sequence yields
$$H^0(C\times T / T, f^*F(m)) \to H^0(C\times T/T, G(m)) \to H^1(C\times T/T,K_G(m)) \cong 0$$
so we get a $T$-point of the Grassmanian $\Grass(H^0(X/S,F(m)),P(m))$.

Composing $\psi_m$ with the plucker embedding
$$\Grass(H^0(X/S,F(m)),P(m)) \hookrightarrow \PP \left ( \bigwedge^{P(m)}H^0(X/S,F(m)) \right)$$
yields the desired embedding of the quot scheme into a projective bundle over $S$. We will denote by $\SL_m$ the pullback of the corresponding canonical ample line bundle on the Grassmanian by $\psi_m$.

\begin{theorem}
\label{thm:LambdaModParam}
Let $\Lambda$ be a sheaf of rings of differential operators on $X$ over $S$ such that $\Lambda|_{\overline{D}}$ is locally free. Let us fix a parabolic type. Let $P$ be a polynomial and let $\{r_{x,i}\}$ be integers for $x\in D$ and $1<i\le l_x$. There exists an integer $N$ such that the functor $\SR^{s}:(\Sch_S)\to (\Sets)$ (respectively $\SR^{ss}$) that associates each $S$-scheme $T$ the set of isomorphism classes of pairs consisting on a (semi-)stable parabolic $\Lambda$-module $(E,E_\bullet)$ over $X \times_S T$ with Hilbert polynomial $P_E=P$ such that for each $x\in D$, $\rk(E|_{x\times T}/E_{x,i})=r_{x,i}$ and an isomorphism
$$\alpha: \SO_T \otimes_\CC \CC^{P(N)} \to H^0(X\times_S T / T, E(N))$$
is representable by a quasi-projective scheme $R^{s}$ ($R^{ss}$) over $S$.
\end{theorem}

\begin{proof}
Let $r$ be the rank of $E$. Let $N$ be $|D|+1$ plus the maximum of the bounds given by Corollary \ref{cor:acyclicBound}, Corollary \ref{cor:acyclicDestab} and Lemma \ref{lemma:sharpStabInequality}. Let $Q_5$ the subscheme of $\Quot_{X/S}(\Lambda_r \otimes_{\SO_X} \SO_X(-N)\otimes_\CC \CC^{P(N)},P)$ described in \cite[Theorem 3.8]{Si2}, parameterizing triples $(E,\varphi,\alpha)$ consisting on a $\Lambda$-module $E$ over $X$ with $\varphi:\Lambda\otimes E\to E$ and an isomorphism $\alpha: \SO_S\otimes_{\CC} \CC^{P(N)} \to H^0(X / S, E(N))$.

By construction, every sheaf in the family $Q_5$ is a quotient of $\Lambda_r \otimes \SO_X(-N) \otimes_\CC \CC^{P(N)}$. Let $Q_5^{\op{LF}}$ be the open subset of triples $(E,\varphi,\alpha)\in Q_5$ such that $E$ is locally free. By the previous corollary, there exists a locally closed subscheme
\begin{multline*}
FQ_5^{\op{LF}} \hookrightarrow Q_5^{\op{LF}}\times_S \fprod{S}{x\in D}{} \Grass(\Lambda_r|_{\{x\}\times S} \otimes \SO_S(-N)\otimes_\CC \CC^{P(N)},r)\\
\times_S \fprod{S}{x\in D}{}\fprod{S}{i=2}{l_x} \Grass(\Lambda_r|_{\{x\}\times S} \otimes \SO_S(-N) \otimes_\CC \CC^{P(N)},r_{x,i})
\end{multline*}
over $Q_5^{\op{LF}}$ whose $T$-points parameterize tuples $(E,E_\bullet,\varphi,\alpha)$ consisting on a rigidified locally free $\Lambda$-module $(E,\varphi,\alpha)$ in $Q_5(T)$ and a filtration by sub-bundles over $T$
$$E|_{\{x\}\times T} = E_{x,1} \supsetneq E_{x,2} \supsetneq \ldots \supsetneq E_{x,l_x}$$
for every $x\in D$ such that for each $1<i\le l_x$, $\rk(E|_{\{x\}\times T}/E_{x,i})=r_{x,i}$.

Let $f:T\to S$, and let $(E,E_\bullet,\varphi,\alpha)$ be a $T$-point in $FQ_5^{\op{LF}}$. We say that $(E,E_\bullet)$ satisfies condition $R_j$ (belongs to $R_j(T)$) if for every $x\in X$, $1<i\le l_x$ the image of
$$f^*(\Lambda_j) \otimes E_x^i \hookrightarrow f^*(\Lambda_j) \otimes E \twoheadrightarrow  E$$
lies in $E_{x}^i$. Let $Q_{x,i}=E/E_{x}^i$. Then the previous condition is equivalent to requiring that the morphism $f^*(\Lambda_j) \otimes E_{x}^i\to Q_{x,i}$ given by the composition
$$f^*(\Lambda_j)\otimes E_{x,i} \hookrightarrow f^*(\Lambda_j) \otimes E \twoheadrightarrow E \twoheadrightarrow Q_{x,i}$$
is zero.

Let $(\SE,\SE_\bullet,\Phi,A)$ be the universal pair for $FQ_5^{\op{LF}}$. For each $x\in D$ and each $i=1,\ldots,l_x$, let $\SE_x^i$ be the vector bundle fitting in the short exact sequence
$$0\longrightarrow \SE_x^i \longrightarrow \SE \longrightarrow \SE|_{\{x\}\times FQ_5^{\op{LF}}}/\SE_{x,i}\longrightarrow 0$$
Moreover, take $\SE_x^{l_x+1}=\SE(-\{x\}\times S)$. Let $\SQ_{x,i}=\SE/\SE_{x}^i$. Let $f:T\to S$ and let $(E,E_\bullet,\varphi,\alpha)$ be a $T$-point in $FQ_5^{\op{LF}}$. It is given by the pullback of $(\SE,\SE_\bullet,\Phi,A)$ by a morphism $e:T\to FQ_5^{\op{LF}}$. By flatness of $\SQ_{x,i}$, $Q_{x,i}=e^*\SQ_{x,i}$. Therefore, $(E,E_\bullet,\varphi,\alpha)$ satisfies condition $R_j$ if and only if the pullback by $e$ of the morphisms given by the compositions
$$\pi^*(\Lambda_j) \otimes \SE_x^i \hookrightarrow \pi^*(\Lambda_j) \otimes \SE\twoheadrightarrow \SE \twoheadrightarrow \SQ_{x,i}$$
are all zero. By \cite[Lemma 4.3]{Yokogawa93}, this condition is represented by a closed subscheme $R_j$ of $FQ_5^{\op{LF}}$ over $S$.

Let $R=\bigcap_{j=1}^\infty R_j\subseteq FQ_5^{\op{LF}}$. As $Q_5$ is noetherian, $Q_5^{\op{LF}}$ is noetherian. On the other hand, $FQ_5^{\op{LF}}$ is quasiprojective over $Q_5^{\op{LF}}$, so it is also noetherian. Therefore, $R$ is a closed subscheme of $FQ_5^{\op{LF}}$ and a point of $FQ_5^{\op{LF}}(T)$ belongs to $R(T)$ if and only if it satisfies the conditions $R_j$ for all $j\ge 1$. 

Let $R^s$ (respectively $R^{ss}$) be the sub-scheme of $R$ parameterizing points of $R$ whose underlying parabolic $\Lambda$-module is (semi-)stable. In the next section (Lemma \ref{lemma:GITLambdaModules}) we will prove that (semi-)stability condition on $R$ is equivalent to GIT-(semi-)stability for a certain group action. Therefore, $R^s$ and $R^{ss}$ are locally closed subschemes of $R$. Let us prove that $R^s$ and $R^{ss}$ represent the functors $\SR^s$ and $\SR^{ss}$ respectively.

Let $\Lambda_{R^{ss}}$ be the base change of $\Lambda$ to $R^{ss}$ via $R^{ss}\to S$, and let $(\SE^{R^{ss}},\SE_\bullet^{R^{ss}},\varphi^{R^{ss}},\alpha^{R^{ss}})$ be the universal rigidified parabolic $\Lambda_{R^{ss}}$-module on $R^{ss}$. As $(\SE^{R^{ss}},\varphi^{R^{ss}},\alpha^{R^{ss}})$ is a $R^{ss}$-point of $Q_5^{\op{LF}}$, we have a natural morphism
$$\alpha^{R^{ss}}:\SO_{R^{ss}}\otimes_\CC \CC^{P(N)} \to H^0(C\times R^{ss}/R^{ss},\SE^{R^{ss}}(N))$$

As $N$ was chosen so that the conclusion of Corollary \ref{cor:acyclicBound} holds and the restriction of $(\SE^{R^{ss}},\SE_\bullet^{R^{ss}},\varphi^{R^{ss}})$ to any closed point is semistable, $H^0(C\times R^{ss}/R^{ss},\SE^{R^{ss}}(N))$ is locally free of rank $P(N)$ and compatible with base change. On the other hand, condition $Q_2$ in Simpson's construction of scheme $Q_5$ \cite[Theorem 3.8]{Si2} imply that $\alpha^{R^{ss}}$ is injective on the fibers over closed points, so it is an isomorphism. Therefore, we obtain a universal object $(\SE^{R^{ss}},\SE_\bullet^{R^{ss}},\varphi^{R^{ss}},\alpha^{R^{ss}})$ over $R^{ss}$, inducing the quadruples described by the functor $\SR^{ss}$ for every base change $e:T\to R^{ss}$. Let us verify that $R^{ss}$ represents the functor $\SR^{ss}$.

Let $f:T\to S$ be an $S$-scheme of finite type, and let $(E,E_\bullet,\varphi,\alpha)$ be a pair of a semistable parabolic $\Lambda$-module $(E,E_\bullet,\varphi)$ over $C\times T$ with Hilbert polynomial $P$ and the given fixed parabolic structure and an isomorphism
$$\alpha:\SO_{T} \otimes_\CC \CC^{P(N)} \to H^0(C\times T/T,E(N))$$
By Corollary \ref{cor:acyclicBound}, $E(N)$ is generated by global sections and we have a surjection
$$\SO_{C\times T}(-N)  \otimes_{\SO_T}  H^0(C\times T/T, E(N)) \to E \to 0$$
Therefore, $\alpha$ induces a morphism
\begin{multline*}
f^*(\Lambda_r) \otimes \SO_{C\times T}(-N) \otimes_\CC \CC^{P(N)} \cong f^*(\Lambda_r) \otimes \SO_{C\times T}(-N)\otimes  H^0(C\times T/T, E(N))\to E \to 0
\end{multline*}
which defines a point on $Q_5^{\op{LF}}$ clearly.  Restricting the previous morphism to $\{x\}\times T$, for $x\in D$, we get quotients
$$f^*(\Lambda_r)|_{\{x\}\times T} \otimes \SO_T(-N) \otimes_\CC \CC^{P(N)}\to E|_{\{x\}\times T} = E_{x,1} \to 0$$
The rest of the parabolic structure $E_\bullet$ induces a set of quotients
$$f^*(\Lambda_r)|_{\{x\}\times T} \otimes \SO_T(-N) \otimes_\CC \CC^{P(N)}\to E|_{\{x\}\times T} = E_{x,1} \to E_{x,1}/E_{x,i} \to 0$$
for every $x\in D$ and $1<i\le l_x$, so $(E,E_\bullet,\varphi,\alpha)$ defines clearly a point in $FQ_5(T)$. As the filtration is preserved by $f^*\Lambda$, it lies in $R(T)$. The parabolic $\Lambda$-module is semistable, so it represents a point $e\in R^{ss}(T)$. By universality of $(\SE^{R^{ss}},\SE_\bullet^{R^{ss}},\varphi^{R^{ss}},\alpha^{R^{ss}})$ yields $e^*(\SE^{R^{ss}},\SE_\bullet^{R^{ss}},\varphi^{R^{ss}},\alpha^{R^{ss}})\cong (E,E_\bullet,\varphi,\alpha)$.

The previous construction is clearly compatible with base change, so it defines a natural transformation $\SR^{ss} \to Hom(\cdot,R^{ss})$. Taking the pullback of the universal object defines an inverse natural transformation $Hom(\cdot,R^{ss})\to \SR^{ss}$, so $R^{ss}$ represents $\SR^{ss}$.

By definition, the points in $R^{s}$ represent points $(E,E_\bullet,\varphi,\alpha)$ in $R^{ss}$ such that $(E,E_\bullet,\varphi)$ is stable. Then, the restriction of the natural transformation $Hom(\cdot,R^{ss}) \to \SR^{ss}$ to $Hom(\cdot,R^{s})\subseteq Hom(\cdot, R^{ss})$ lies in the subfunctor $\SR^{s} \subseteq \SR^{ss}$. As the natural transformation is an isomorphism, its restriction to $Hom(\cdot,R^{s})$ is an isomorphism onto its image, so we get an isomorphism of functors $Hom(\cdot,R^s)\to \SR^s$. Then $R^s$ represents $\SR^s$ and by \cite[Lemma 1.11]{Si2}, $R^s$ is an open subscheme of $R^{ss}$.
\end{proof}

\section{Geometric invariant theory}
\label{section:GIT}

Two different $T$-points of the previous scheme $R^s$ ($R^{ss}$) with the same underlying parabolic $\Lambda$-module $(E,E_\bullet,\varphi)$ differ only in the choice of the isomorphism
$$\alpha: \SO_T \otimes_\CC \CC^{P(N)}\to H^0(X\times_S T / T, E(N))$$
Therefore, they are related by an automorphism of $\SO_T \otimes_\CC \CC^{P(N)}$, which is equivalent to a morphism $T\to \GL_{P(N)}(\CC)$. As dilatations $T\to \CC^*$ preserve the isomorphism class of $(E,E_\bullet,\varphi)$ up to tensoring by a line bundle over the parameter space $L\to T$, two isomorphism classes of $T$-families of parabolic $\Lambda$-modules differ effectively by a morphism $T\to \op{SL}_{P(N)}(\CC)$. 

Therefore, the moduli functor of (semi-)stable parabolic $\Lambda$-modules is clearly a categorical quotient of the functor described in the previous theorem by the action of $\op{SL}_{P(N)}(\CC)$ on $\CC^{P(N)}$.  Subsequently, we can obtain a coarse moduli space for the desired moduli functor by finding a good categorical quotient of the scheme $R^s$ ($R^{ss}$) described in the previous theorem by the action of $\op{SL}_{P(N)}(\CC)$. We will use Geometric Invariant Theory to describe this quotient.

First of all, we will briefly review Mumford's notation on GIT quotients and the main GIT-stability theorem. Let $X$ be a proper complex algebraic scheme and let $G$ be an algebraic group acting on $X$. Let $\lambda$ be a one parameter subgroup (1-PS) of $G$. For every closed point $x\in X$, composing with the action of $G$ on $X$ yields a morphism of $\GG_m$ to $X$. By properness of $X$, it extends to a morphism $f_{x,\lambda}:\mathbb{A}^1 \to X$ such that $f_{x,\lambda}(0)$ is a fixed point for the $\GG_m$ action.  Let $L$ be a $G$-linearized line bundle over $X$. By \cite[\S 1.3]{MumfordGIT}, the induced $\GG_m$ linearization of $L$ restricted to $f_{x,\lambda}(0)$ is given by a character of $\GG_m$, $\xi(\alpha)=\alpha^r$ for $\alpha\in \GG_m$. We define
$$\mu_G^L(x,\lambda)=-r$$
We will be interested in the following two functorial properties of $\mu$
\begin{enumerate}
\item For fixed $x$ and $\lambda$, $\mu_G^{(\bullet)}(x,\lambda)$ defines a homomorphism from the group of $G$-linearized line bundles on $X$ to $\ZZ$.
\item If $X\to Y$ is a $G$-linear morphism of schemes on which $G$ acts and $L$ is a $G$-linearized line bundle over $Y$ then $f^*L$ is a $G$-linearized line bundle over $X$ and
$$\mu_G^{f^*L}(x,\lambda)=\mu_G^L(f(x),\lambda)$$
\end{enumerate}

\begin{lemma}
\label{lemma:productMu}
Let $\varepsilon_1,\ldots,\varepsilon_M$ be positive rational numbers. Let $G$ be an algebraic group and let $X_1,\ldots,X_M$ be complex projective schemes such that for each $i=1,\ldots, M$, $G$ acts on $X_i$. For each $i$ let $\SO_{X_i}(1)$ be an ample $G$-linearized line bundle over $X_i$. Then for each closed point $x=(x_1,\ldots,x_M)$ in $X_1\times \ldots \times X_M$ and each 1-PS $\lambda$ of $G$
$$\mu_G^{\bigotimes_{i=1}^M \SO_{X_i}(\varepsilon_i)}(x,\lambda) = \sum_{i=1}^M \varepsilon_i \mu_G^{\SO_{X_i}(1)}(x_i,\lambda)$$
\end{lemma}

\begin{proof}
See \cite[Chapter 3]{MumfordGIT}.
\end{proof}

\begin{lemma}
\label{lemma:subgroupMu}
Let $G$ be an algebraic group acting on a complex proper scheme $X$, and let $H$ be a subgroup of $G$. Let $L$ be any $G$-linearized line bundle over $X$. Then $H$ acts on $X$, the $G$-linearization of $L$ induces an $H$-linearization of $L$. Let $\lambda$ be a 1-PS of $H$, and let $\overline{\lambda}$ be the 1-PS of $G$ obtained composing $\lambda$ with the inclusion of $H$ in $G$. Then for every geometric point $x$
$$\mu_H^L(x,\lambda)=\mu_G^L(x,\overline{\lambda})$$
\end{lemma}

\begin{proof}
For every closed point $x\in X$, the composition of the action of $H$ on $X$ with $\lambda$ coincides with the composition of the action of $G$ with $\overline{\lambda}$. Therefore, both actions induce the same morphism $f_{x,\lambda}:\mathbb{A}^1\to X$ and the same $\GG_m$-linearization of $L$. Thus, the linearization of $L$ is given by the same character and the equality holds by definition of $\mu$.
\end{proof}

\begin{theorem}[{\cite[Theorem 2.1]{MumfordGIT}}]
\label{thm:GIT}
Let $G$ be a reductive group acting on a proper complex scheme $X$. Let $L$ be a $G$-invariant ample line bundle over $X$. Then for every geometric point of $X$
\begin{enumerate}
\item $x$ is GIT-semistable if and only if $\mu^L(x,\lambda)\ge 0$ for all 1-PS $\lambda$.
\item $x$ is GIT-stable if and only if $\mu^L(x,\lambda)> 0$ for all 1-PS $\lambda$.
\end{enumerate}
\end{theorem}

Now, we will apply the previous Theorem to compute the GIT-stability condition for the linear action on a product of Grassmanians.

Let $V$ and $W_i$ be complex vector spaces for $i=1\ldots,M$. Let $p_i$ be an integer $0\le p_i\le dim W_i$ for $i=1,\ldots,M$. For every $i$, let $\Grass(p_i, W_i \otimes V)$ and $\Grass(W_i\otimes V, p_i)$ denote the Grassmanians of subspaces and quotients respectively. There is a canonical isomorphism
$$\Grass(W_i \otimes V ,p_i) \cong \Grass(\dim(W_i)\dim(V)-p_i,W_i \otimes V)$$

Let us consider the canonical action of $\op{SL}(V)$ on $\Grass(W_i \otimes V,p_i)$ extended from the action on $V$. For each $i$, $\Grass(W_i \otimes V,p_i)$ gets embedded into $\PP\left(\bigwedge^{p_i} (W_i\otimes V) \right)$ by Plucker embedding. For each $i$, let $\SO_i(1)$ denote the pullback of $\SO_{\PP\left(\bigwedge^{p_i} (W_i \otimes V) \right)}(1)$.

\begin{lemma}
Let $\varepsilon_1,\ldots,\varepsilon_M$ be positive rational numbers. Let $x=(L_1, \ldots, L_M)$ be a geometric point of $\prod_{i=1}^M \Grass(p_i,W_i \otimes V)$, i.e., let $L_i$ be a subspace of $W_i \otimes V$ of dimension $p_i$ for each $i=1,\ldots,M$. Then $x$ is GIT-(semi-)stable with respect to the action of $\op{SL}(V)$, linearized by $\Theta=\bigotimes_{i=1}^m \SO_i(\varepsilon_i)$, if and only if for all linear subspaces $L\subseteq V$
\begin{equation}
\label{eq:GITGrassCondition}
\frac{\sum_{i=1}^M \varepsilon_i \dim (L_i\cap (W_i \otimes L))}{\dim L} (\le) < \frac{\sum_{i=1}^M \varepsilon_i p_i}{\dim V}
\end{equation}
\end{lemma}

\begin{proof}
Let $n=\dim(V)$ and $m_i=\dim(W_i)$. Let $\lambda$ be any 1-PS of $\op{SL}(V)$. Let $\{e_1,\ldots,e_n\}$ be a basis of $V$ such that for every $t\in\CC$, the matrix of $\lambda(t)$ in the basis $\{e_1,\ldots,e_n\}$ is given by
$$\lambda(t)=\left(t^{r_i}\delta_{ij} \right)_{i,j=1}^n$$
Where $r_1\ge r_2\ge \ldots \ge r_n$ and $\sum_{i=1}^n r_i=0$. Then, fixed any basis $\{w_{i,1},\ldots, w_{i,m_i}\}$ of $W_i$, $\{w_{i,k} \otimes e_j\}$ is a basis of $W_i \otimes V$. Let $\overline{\lambda}$ be the composition of $\lambda$ with the canonical inclusion of $\op{SL}(V)$ in $\op{SL}(W_i \otimes V)$. Under, the given adapted basis, $\overline{\lambda}(t)$ has the form
$$\overline{\lambda}(t) = \left (\begin{array}{ccc}
\overbrace{\begin{array}{ccc|}
t^{r_1} & & \\
& \ddots & \\
& & t^{r_1}\\
\hline
\end{array}}^{m_i} & &\\
& \ddots &\\
&& \overbrace{\begin{array}{|ccc}
\hline
t^{r_n} & & \\
& \ddots & \\
& & t^{r_n}
\end{array}}^{m_i}
\end{array} \right)$$

Let us call $\overline{r}_l$ to the exponent of $t$ for the $l$-th entry on the diagonal of $\overline{\lambda}(t)$, i.e., for $1\le j\le n$ and $1\le k\le m_i$
$$\overline{r}_{(j-1)m_i+k}=r_j= r_{\lceil l/m_i \rceil}$$
On the other hand, for each subset $I\subseteq \{1,\ldots,n\} \times \{1,\ldots, m_i\}$, let $\overline{L}_I$ denote the linear subspace of $V\times W_i$  defined as
$$\overline{L}_I=\left \{\sum_{j=1}^n \sum_{k=1}^{\dim(W_i)} x_{j,k}w_{i,k} \otimes  e_j \middle | x_{j,k}=0 \, \forall (j,k)\in I \right\}$$
Similarly, for each subset $I\subseteq \{1,\ldots,n\}$, let $L_I$ be the subspace of $V$ defined as
$$L_I=\left \{\sum_{j=1}^n x_j e_j \middle| x_j=0 \, \forall j\in I \right \}$$
Finally, if we denote $I_l=\{(j,k)\in\{1,\ldots,n\} \times \{1,\ldots, m_i\} | (j-1)m_i+k\ge l\}$, for $1\le l\le nm_i$, combining Lemma \ref{lemma:subgroupMu} and \cite[Proposition 4.3]{MumfordGIT} yields
$$\mu_{\op{SL}(V)}^{\SO_i(1)}(L_i,\lambda)=\mu_{\op{SL}(W_i\otimes V)}^{\SO_i(1)}(L_i,\overline{\lambda})= - p_i \overline{r}_{nm_i} + \sum_{l=1}^{nm_i-1} \dim \left (L_i\cap \overline{L}_{I_{l+1}} \right )(\overline{r}_{l+1}-\overline{r}_l)$$
On the other hand, for each $l$, $\overline{r}_{l+1}-\overline{r}_l=r_{\lceil (l+1)/m_i \rceil}-r_{\lceil l/m_i \rceil}$ is only nonzero if $l$ is a multiple of $m_i$. Moreover, for each $1\le j<n$,
\begin{multline*}
L_{I_{jm_i+1}}=\SL ( \{e_{j'}\otimes w_{i,k} | 1\le j'\le j,1\le k\le m_i\}) =\\
W_i\otimes  \SL(\{e_{j'} | 1\le j'\le j\}) = W_i\otimes L_{\{j+1,\ldots,n\}}
\end{multline*}
Therefore
\begin{multline*}
\mu_{\op{SL}(V)}^{\SO_i(1)}(L_i,\lambda)=- p_i r_n + \sum_{j=1}^{n-1} \dim \left (L_i\cap \overline{L}_{I_{jm+1}} \right )(r_{j+1}-r_j)=\\ - p_i r_n +
\sum_{j=1}^{n-1} \dim \left (L_i\cap \left(W_i\otimes  L_{\{j+1,\ldots,n\}}\right) \right )(r_{j+1}-r_j)
\end{multline*}
Now, by Lemma \ref{lemma:productMu},
$$\mu_{\op{SL}(V)}^{\Theta}(x,\lambda)=-\sum_{i=1}^M \varepsilon_i p_i r_n + \sum_{i=1}^M \varepsilon_i \sum_{j=1}^{n-1}  \dim \left (L_i\cap \left(W_i\otimes  L_{\{j+1,\ldots,n\}}\right) \right )(r_{j+1}-r_j)$$
By Theorem \ref{thm:GIT}, $x$ is GIT-(semi-)stable if and only if $\mu_{\op{SL}(V)}^\Theta(x,\lambda)$ is positive (non-negative) for all 1-PS $\lambda$. $\mu_{\op{SL}(V)}^\Theta(x,\lambda)$ is a linear function of the $r_j$. Thus, its value is positive (non-negative) for all $r_j$ such that $r_1\ge \ldots \ge r_n$ and $\sum_{j=1}^n r_j =0$ if and only if it is positive (non-negative) for the extreme sets of $r_j$
$$\left \{ \begin{array}{c}
r_1=r_2=\ldots=r_l=n-l\\
r_{l+1}=\ldots=r_n=-l
\end{array} \right.$$
for every $1\le l <n$. For such $\{r_j\}$, $r_{j+1}-r_j$ is nonzero just for $j=l$, so $\mu_{\op{SL}(V)}^{\Theta}(x,\lambda) (\ge) > 0$ for such $r_j$ if and only if
\begin{equation}
\label{eq:GrassGIT1}
l\sum_{i=1}^M \varepsilon_i p_i - \sum_{i=1}^M \varepsilon_i \dim \left(L_i \cap \left (W_i\otimes L_{\{l+1,\ldots,n\}}\right) \right) n (\ge) > 0
\end{equation}
Every 1-PS of $\op{SL}(V)$ is conjugate to a 1-PS which is diagonalized in the basis $\{e_1,\ldots,e_n\}$ and for which $r_1\ge r_2\ge \ldots \ge r_n$. Therefore, $x$ is GIT-semistable if and only if condition \eqref{eq:GrassGIT1} holds for every basis choice $\{e_1,\ldots,e_n\}$. Under base change, $L_{\{l+1,\ldots,n\}}$ ranges over the set of linear subspaces of $V$ of dimension $l$, so $x$ is GIT-(semi-)stable if and only if for every subspace $L\subsetneq V$
$$\dim(L)\sum_{i=1}^M \varepsilon_i p_i - n\sum_{i=1}^M \varepsilon_i \dim \left(L_i \cap (W_i\otimes L) \right) (\ge)> 0$$
\end{proof}

\begin{corollary}
\label{cor:GITlemma}
Let $\varepsilon_1,\ldots,\varepsilon_M$ be rational numbers. Let $x=(\varphi_1,\ldots,\varphi_M)$ be a geometric point of $\prod_{i=1}^M \Grass(W_i\otimes  V,p_i)$, i.e., let $\varphi_i: W_i\otimes V \twoheadrightarrow L_i$ be a quotient of dimension $p_i$ for each $i=1,\ldots,M$. Then $x$ is GIT-(semi-)stable with respect to $\bigotimes_{i=1}^m \SO_i(\varepsilon_i)$ if and only if for all linear subspaces $L\subseteq V$
$$\frac{\sum_{i=1}^M \varepsilon_i \dim(\varphi_i(W_i\otimes L))}{\dim L} (\ge)> \frac{\sum_{i=1}^M \varepsilon_i p_i}{\dim V}$$
\end{corollary}

\begin{proof}
Let $p_i'=\dim(V) \dim (W_i)-p_i$. For each $\varphi_i: W_i \otimes V \twoheadrightarrow L_i$, let $L_i'=\Ker(\varphi_i)$ be the corresponding image of $\varphi_i$ in $\Grass(p_i',W_i \otimes V)$, and let $x'=(L_1',\ldots,L_M')$. Clearly, $x$ is GIT-semistable if and only if $x'$ is GIT-semistable, i.e., if and only if inequality \eqref{eq:GITGrassCondition} holds for every linear subspace $L\subseteq V$. Let $L\subseteq V$ be any linear subspace. Then, for each $i$
\begin{multline*}
\dim(L_i'\cap (W_i\otimes L))=\dim (\Ker (\varphi_i) \cap (W_i \otimes L))= \dim ( \Ker (\varphi_i |_{W_i \otimes L})) =\\
\dim(W_i\otimes L) - \dim ( \im (\varphi_i|_{W_i \otimes L})) = \dim (W_i)\dim (L) - \dim (\varphi_i(W_i \otimes L))
\end{multline*}
Therefore, by the previous Lemma, $x'$ is (semi-)stable if and only if
\begin{multline*}
\frac{\sum_{i=1}^M \varepsilon_i \dim(\varphi_i(W_i \otimes L))}{\dim L} = \frac{\sum_{i=1}^M \varepsilon_i \left(\dim(W_i)\dim(L) -  \dim(L_i\cap (W_i \otimes L))\right)}{\dim L} = \\
\sum_{i=1}^M \varepsilon_i \dim(W_i) -\frac{\sum_{i=1}^M \varepsilon_i \dim(L_i\cap (W_i \otimes L))}{\dim L} (\ge) > \sum_{i=1}^M \varepsilon_i \dim(W_i) - \frac{\sum_{i=1}^M \varepsilon_i p_i'}{\dim(V)} =\\
 \frac{\sum_{i=1}^M \varepsilon_i ( \dim(W_i)\dim(V)-p_i')}{\dim(V)}= \frac{\sum_{i=1}^M \varepsilon_i p_i}{\dim(V)}
\end{multline*}

\end{proof}

We can extend the previous GIT-stability conditions for Grassmanians of locally free vector bundles by means of the following Lemma due to Simpson

\begin{lemma}[{\cite[Lemma 1.13]{Si2}}]
Suppose that $Z\to S$ is a projective scheme. Let $G$ be a reductive algebraic group acting on $Z$, such that the action is trivial on $S$ and preserves the morphism $Z\to S$. Let $\SL$ be a relatively very ample linearizable invertible sheaf for the action of $G$. If $t\to S$ is a geometric point then the (semi-)stable points of the fibrer $Z_s$ are those which are (semi-)stable in the total space, i.e. $Z^s_t=(Z^s)_t$ ($Z^{ss}_t = (Z^{ss})_t$ respectively).
\end{lemma}

Let $\SW_i$ be a locally free vector bundle on $S$ for $i=1,\ldots,l$. Let $V$ be a finite dimensional vector space. Let $r_i$ be integers and let $\varepsilon_i$ be rational numbers for $i=1,\ldots,l$. Let $G_i=\Grass(\SW_i\otimes_\CC V,r_i) \to S$ and let us consider the canonical action of $\op{SL}(V)$ on $G_i$ for each $i$. Let $\SO_i(1)$ denote the canonical $G$-linearizable line bundle on $S$ corresponding to the Plucker embedding
$$\Grass(\SW_i\otimes_\CC V,r_i) \to \PP \left( \bigwedge^{r_i} \SW_i\otimes_\CC V \right)$$
Then $\op{SL}(V)$ acts on $G=\prod_{i=1}^l G_i$ and $\Theta=\bigotimes_{i=1}^l \SO_i(\varepsilon_i)$ is a relatively very ample $G$-linearizable invertible sheaf, flat over $S$. By the previous Lemma, a geometric point in $G$, standing over a fiber $t\to S$ is (semi-)stable if and only if the corresponding point in $G_s = \prod_{i=1}^l \Grass(\SW_i|_s \otimes V,r_i)$ satisfies the condition of Corollary \ref{cor:GITlemma}.

\begin{lemma}
\label{lemma:GITLambdaModules}
Let $R$ be the quasi-projective scheme described by Theorem \ref{thm:LambdaModParam}. Let $\SO_{x,i}(1)$ denote the canonical ample line bundle on $\Grass(\Lambda_r|_{\{x\}\times S} \otimes \SO_S(-N) \otimes_\CC \CC^{P(N)},r_{x,i})$ under the plucker embedding for $x\in D$ and $1\le i\le l_x$, where we are setting $r_{x,1}=r$ for each $x\in D$. There exists an integer $\overline{m}$ such that for every $m\ge \overline{m}$ there exist positive rational numbers $\varepsilon_0$, $\varepsilon_{x,i}$, for $x\in D$, $i=1,\ldots,l_x$ such that a geometric point of $R$ is GIT-(semi-)stable for the linearization induced by the restriction of
$$\Theta = \SL_{m}(\varepsilon_0)\otimes \bigotimes_{x\in D} \bigotimes_{i=1}^{l_x} \SO_{x,i}(\varepsilon_i)$$
to $R$ if and only if it corresponds to a (semi-)stable parabolic $\Lambda$-module.
\end{lemma}

\begin{proof}
Let $M=|D|$. For each $x\in D$, let $\varepsilon_{x,i}=\alpha_{x,i}-\alpha_{x,i-1}$ for $1<i\le l_x$ and $\varepsilon_{x,1}=1-\alpha_{x,l_x}$. As $\alpha_{x,i}$ are strictly crescent and less than $1$, it is clear that $\varepsilon_{x,i}>0$ for all $x\in D$ and $i=1,\ldots,l_x$. For any $m>N$ let
$$\varepsilon_0 = \frac{\pmu(E) + N-M+1-g}{m-N}$$
The choice of $N$ in Theorem \ref{thm:LambdaModParam} ensures that $\varepsilon_0>0$. By definition a $\Lambda$-module over $C\times S$ is (semi-)stable if and only if its restriction to $X_s$ for every $s\in S$ is (semi-)stable. On the other hand, by the previous lemma an $S$-point in $R$ is (semi-)stable, if and only if its specification to every $s\in S$ is (semi-)stable. Therefore, we can restrict ourselves to closed points of $R$, i.e., $\Lambda$-modules over a certain geometric fiber $X_s$ of $C\times S$.
First, let us prove that all GIT-semistable points of $R$ are semistable. Let $(E,E_\bullet,\varphi)$ be a parabolic $\Lambda$-module over $X_s$, for some $s\in S$ underlying a GIT-semistable point of $R$. Suppose that $(E,E_\bullet,\varphi)$ is unstable. Let $(F,F_\bullet,\varphi)$ be the maximum destabilizing sub-$\Lambda$-module. By maximality of the parabolic slope, it is a semistable parabolic $\Lambda$-module and we can assume without loss of generality that $F$ has the induced parabolic structure from $(E,E_\bullet)$.

By Corollary \ref{cor:acyclicDestab}, $F(N)$ is generated by global sections. Let $L=H^0(X_s,F(N))$. By Corollary \ref{cor:GITlemma}, GIT-semi-stability of $(E,E_\bullet,\varphi,\alpha)$ for the linearization $\Theta$ implies that
\begin{multline}
\label{eq:GITLambdaModules1}
\frac{\varepsilon_0 \dim \im(H^0(X_s,\Lambda_{X_s,r} \otimes \SO_{X_s}(m-N) ) \otimes L\to H^0(X_s,E(m)))}{\dim L}\\ 
+ \frac{\sum_{x\in X} \varepsilon_{x,1} \dim \im(\Lambda_r|_{\{(x,s)\}}\otimes L \to E_{x,1})}{\dim L} \\ 
+ \frac{\sum_{x\in D}\sum_{i=2}^{l_x} \varepsilon_{x,i} \dim \im (\Lambda_r|_{\{(x,s)\}}\otimes \SO_S(-N)|_{s}\otimes L \to E_{x,1}/E_{x,i})}{\dim L}\\
 \ge \frac{\varepsilon_0 P_E(m) +\sum_{x\in D}\left(\varepsilon_{x,1} r+\sum_{i=2}^{l_x} \varepsilon_{x,i} r_{x,i} \right)}{P_E(N)}
\end{multline}

On the other hand, by $N$-regularity of $F$, yields
\begin{eqnarray*}
\xymatrixcolsep{3pc}
\xymatrixrowsep{2pc}
\xymatrix{
H^0(X_s,\Lambda_{X_s,r} \otimes \SO_{X_s}(m-N) ) \otimes H^0(X_s,E(N)) \ar@{->>}[r] & H^0(X_s,E(m))\\
H^0(X_s,\Lambda_{X_s,r} \otimes \SO_{X_s}(m-N) ) \otimes H^0(X_s,F(N)) \ar@{->>}[r] \ar@{^{(}->}[u] & H^0(X_s,F(m)) \ar@{^{(}->}[u]
}
\end{eqnarray*}
Thus
$$\dim \im(H^0(X_s,\Lambda_{X_s,r} \otimes \SO_{X_s}(m-N) ) \otimes L\to H^0(X_s,E(m)))=P_F(m)$$
As $\Lambda_{X_s,r} \otimes \SO_{X_s}(-N) \otimes L$ generates $F$, $\Lambda_r|_{\{(x,s)\}}\otimes \SO_S(-N)|_s\otimes L$ generates $F|_{\{(x,s)\}}=F_{x,1}$ in $E_{x,1}$ for each $x\in D$. Therefore, for each $x\in D$ and each $1<i\le l_x$
\begin{multline*}
\dim \im \left(\Lambda_r|_{\{(x,s)\}}\otimes \SO_S(-N)|_{s} \otimes L \to \left(E_{x,1}/E_{x,i}\right)\right)=\dim \left ( \frac{F|_{\{(x,s)\}}}{E_{x,i}}\right) \\
=\rk \left ( \frac{F_{x,1}}{E_{x,i}\cap F_{x,1}} \right)= \rk \left( \frac{ F_{x,1}}{F_{x,i}}\right)= \rk(F_{x,1})-\rk(F_{x,i})=\rk(F)-\rk(F_{x,i})
\end{multline*}
Therefore, substituting the previous computations in equation \eqref{eq:GITLambdaModules1} and taking into account that by Corollary \ref{cor:acyclicDestab}, $\dim L= h^0(X_s,F(N))=P_F(N)$, yields
\begin{multline}
\label{eq:GITLambdaModules2}
\frac{\varepsilon_0 P_F(m) + \sum_{x\in D}\left(\varepsilon_{x,1} \rk(F) + \sum_{i=2}^{l_x} \varepsilon_{x,i}(\rk(F)-\rk(F_{x,i})) \right)}{P_F(N)}\\
\ge \frac{\varepsilon_0 P_E(m) +\sum_{x\in D}\left(\varepsilon_{x,1} r+\sum_{i=2}^{l_x} \varepsilon_{x,i} r_{x,i} \right)}{P_E(N)}
\end{multline}
For each $x\in D$ we have
\begin{multline}
\label{eq:GITLambdaModules3}
\varepsilon_{x,1} \rk(F) + \sum_{i=2}^{l_x} \varepsilon_{x,i}(\rk(F)-\rk(F_{x,i})) = \sum_{i=1}^{l_x} \varepsilon_{x,i} \rk(F) - \sum_{i=2}^{l_x} (\alpha_{x,i}-\alpha_{x,i-1})\rk(F_{x,i})\\
 = (1-\alpha_1)\rk(F) + \alpha_1 \rk(F_{x,2}) - \sum_{i=2}^{l_x-1} \alpha_{x,i}(\rk(F_{x,i})-\rk(F_{x,i+1}))- \alpha_{x,l_x} \rk(F_{x,l_x})\\
 = \rk(F) - \sum_{i=1}^{l_x-1} \alpha_{x,i}(\rk(F_{x,i})-\rk(F_{x,i+1}))- \alpha_{x,l_x} \rk(F_{x,l_x}) =\rk(F) - \owt_x(F)
\end{multline}
Adding up over $x$ and substituting in both sides of equation \eqref{eq:GITLambdaModules2} yields
$$\frac{\varepsilon_0 P_F(m) + M\rk(F) - \owt(F)}{P_F(N)} \ge \frac{\varepsilon_0 P_E(m)+M\rk(E)-\owt(E)}{P_E(N)}$$
By Riemann-Roch formula, $P_F(k)=\rk(F)(k+1-g)+\deg(F)$. Substituting in the previous equation and dividing both numerators and denominators by the corresponding ranks yields
$$\frac{\varepsilon_0 (m+1-g)+ M + \varepsilon_0 \mu(F) - \eta(F)}{N+1-g+\mu(F)} \ge \frac{\varepsilon_0 (m+1-g) +M + \varepsilon_0 \mu(E) - \eta(E)}{N+1-g+\mu(E)}$$
Subtracting $\varepsilon_0$ to both sides of the inequality gives
$$\frac{\varepsilon_0 (m-N)+M - \eta(F)}{N+1-g+\mu(F)} \ge \frac{\varepsilon_0 (m-N)+M- \eta(E)}{N+1-g+\mu(E)}$$
By the choice of $N$, both denominators are positive, so by multiplying and grouping one obtains
\begin{multline*}
\mu(F) \left( \varepsilon_0(m-N)+M - \eta(E) \right) + \eta(F) \left( N+1-g+\eta(E) \right) \\
\le \mu(E) \left( \varepsilon_0(m-N)+M \right) + \eta(E) \left( N+1-g \right)
\end{multline*}
Adding and substracting $\mu(E)\eta(E)$ to the right hand side of the inequality yields
\begin{multline*}
\mu(F) \left( \varepsilon_0(m-N)+M - \eta(E) \right) + \eta(F) \left( N+1-g+\mu(E) \right) \\
\le \mu(E) \left( \varepsilon_0(m-N)+M -\eta(E) \right) + \eta(E) \left( N+1-g  +\mu(E)\right)
\end{multline*}
By the choice of $\varepsilon_0$, $\varepsilon_0(m-N)+M - \eta(E) = N+1-g+\mu(E)>0$. Thus
$$\mu(F)+\eta(F) \le \mu(E) + \eta(E)$$
contradicting that $F$ destabilizes $E$.

Moreover $\pmu(F)=\pmu(E)$ if and only if \eqref{eq:GITLambdaModules2} is an equality. Therefore, if $(E,E_\bullet,\varphi)$ is a strictly semistable parabolic $\Lambda$-module and $(F,F_\bullet,\varphi)$ is a parabolic sub-$\Lambda$-module   such that $\pmu(F)=\pmu(E)$, then taking $L=H^0(X_s,F(N))$ we obtain equality in \eqref{eq:GITLambdaModules2} and, therefore, $(E,E_\bullet,\varphi,\alpha)$ is strictly GIT semistable.

Now, we will prove that semi-stability implies GIT-stability for big enough $m$. Let $(E,E_\bullet)$ be a semistable parabolic $\Lambda_s$-module over $X_s$. By Corollary \ref{cor:acyclicDestab}, the image of $(E,E_\bullet)$ under the embedding $\psi_m$ is GIT-semistable if and only if condition \eqref{eq:GITLambdaModules1} holds for every $L\subseteq \CC^{P(N)}$.

Let $L\subseteq \CC^{P(N)}$ be any vector subspace. Let $F$ be the subsheaf of $E$ obtained as the image of $L$ under the quotient $\Lambda_{X_s,r} \otimes \SO_{X_s}(-N) \otimes_\CC \CC^{P(N)} \twoheadrightarrow E$ and let $K_{L,E}$ be the kernel of the quotient
$$0\to K_{L,E} \to \Lambda_{X_s,r} \otimes \SO_{X_s}(-N) \otimes_\CC L \to F \to 0$$
Tensoring the previous short exact sequence by $\SO_{X}(m)$ yields
$$0\to K_{L,E}(m) \to \Lambda_{X_s,r} \otimes \SO_{X_s}(m-N) \otimes_\CC L \to F(m) \to 0$$
The set of possible subshefs $F$ generated this way and the set of possible kernels $K_{L,E}$ both form bounded families of sheaves on $X$ flat over $S$, so by \cite[Lemma 1.9]{Si2}, there exists an $m_0$ such that for every $m\ge m_0$, $H^1(X_s,K_{L,E}(m))=0$ and $H^1(X_s,F(m))=0$. Therefore, the corresponding long exact sequence in cohomology reduces to
\begin{multline*}
0\to H^0(X_s,K_{L,E}(m)) \to H^0(X_s,\Lambda_{X_s,r} \otimes \SO_{X_s}(m-N)) \otimes_\CC L \\
\to H^0(X_s,F(m)) \to H^1(X_s,K_{L,E}(m))=0
\end{multline*}
Thus,
$$\dim \im(H^0(X_s,\Lambda_{X_s,r} \otimes \SO_{X_s}(m-N) ) \otimes L \to H^0(X_s,E(m)))= h^0(X_s,F(m))=P_F(m)$$
On the other hand, similarly to the first part of the proof, $L\otimes \Lambda_r|_{\{(x,s)\}}\otimes \SO_S(-N)|_s$ generates $F_{x,1}$ in $E_{x,1}$ for each $x\in D$. Therefore, for each $x\in D$ and each $1<i\le l_x$
\begin{multline*}
\dim \im \left(\Lambda_r|_{\{(x,s)\}}\otimes \SO_S(-N)|_{s} \otimes L \to \left(E_{x,1}/E_{x,i}\right)\right)=\dim \left ( \frac{F|_{\{(x,s)\}}}{E_{x,i}}\right) \\
=\rk \left ( \frac{F_{x,1}}{E_{x,i}\cap F_{x,1}} \right)= \rk \left( \frac{ F_{x,1}}{F_{x,i}}\right)= \rk(F_{x,1})-\rk(F_{x,i})=\rk(F)-\rk(F_{x,i})
\end{multline*}
Substituting the previous computation in condition \eqref{eq:GITLambdaModules1} and taking into account \eqref{eq:GITLambdaModules3} implies that condition \eqref{eq:GITLambdaModules1} is equivalent to
\begin{equation}
\label{eq:GITLambdaModules4}
\frac{\varepsilon_0 P_F(m) + M\rk(F) - \owt(F)}{\dim L} \ge \frac{\varepsilon_0 P_E(m)+M\rk(E)-\owt(E)}{P_E(N)}
\end{equation}
Under the isomorphism $\alpha_E:H^0(X_s,E(N))\to \CC^{P(N)}$, $L$ corresponds to a subspace of sections of $F(N)$. Therefore, $\dim L \le \dim H^0(X_s,F(N))=P_F(N)+h^1(X_s,F(N))$ and in order to prove that the previous condition holds it is sufficient to demonstrate that
$$\frac{\varepsilon_0 P_F(m) + M\rk(F) - \owt(F)}{P_F(N)+h^1(X_s,F(N))} \ge \frac{\varepsilon_0 P_E(m)+M\rk(E)-\owt(E)}{P_E(N)}$$
holds. Let us denote $\tau_F(N)=h^1(X_s,F(N))/\rk(F)$. Applying Riemann-Roch theorem and dividing by the rank yields that the condition is equivalent to
$$\frac{\varepsilon_0 (m+1-g)+M + \varepsilon_0 \mu(F) - \eta(F)}{N+1-g+\mu(F)+\tau_F(N)} \ge \frac{\varepsilon_0 (m+1-g)+M + \varepsilon_0 \mu(E) - \eta(E)}{N+1-g+\mu(E)}$$
Subtracting $\varepsilon_0$ to both sides of the inequality gives
$$\frac{\varepsilon_0 (m-N)+M - \eta(F)-\varepsilon_0 \tau_F(N)}{N+1-g+\mu(F)+\tau_F(N)} \ge \frac{\varepsilon_0 (m-N)+M- \eta(E)}{N+1-g+\mu(E)}$$
By the choice of $N$, both denominators are positive, so by multiplying and grouping one obtains
\begin{multline*}
\mu(F) \left( \varepsilon_0(m-N)+M - \eta(E) \right) + \eta(F) \left( N+1-g+\eta(E) \right) \le \mu(E) \left( \varepsilon_0(m-N)+M \right) \\
+ \eta(E) \left( N+1-g \right)-\tau_F(N) \left(\varepsilon_0(N+1-g+\mu(E))+\varepsilon_0(m-N)+M-\eta(E)\right)
\end{multline*}
Adding and substracting $\mu(E)\eta(E)$ to the right hand side of the inequality yields
\begin{multline*}
\mu(F) \left( \varepsilon_0(m-N)+M - \eta(E) \right) + \eta(F) \left( N+1-g+\mu(E) \right) \\
\le \mu(E) \left( \varepsilon_0(m-N)+M -\eta(E) \right) + \eta(E) \left( N+1-g  +\mu(E)\right)\\
-\tau_F(N) \left(\varepsilon_0(N+1-g+\mu(E))+\varepsilon_0(m-N)+M-\eta(E)\right)
\end{multline*}
Again, by the choice of $\varepsilon_0$, $\varepsilon_0(m-N)+M -\eta(E)=N+1-g  +\mu(E)>0$ so we have to prove that for big enough $m$
$$\mu(F)+\eta(F)+(\varepsilon_0+1) \tau_F(N) \le \mu(E)+\eta(E)$$
Adding $N+1-g$ to both sides of the inequality and applying Riemann-Roch theorem, this is equivalent to proving that
\begin{multline*}
\frac{h^0(X_s,F(N))}{\rk(F)} + \eta(F) + \varepsilon_0 \tau_F(N) = \frac{P_F(N)+h^1(X_s,F(N))}{\rk(F)} + \eta(F)+\varepsilon_0\tau_F(N)\\
 = N+1-g + \mu(F)+\eta(F)+(\varepsilon_0+1) \tau_F(N) \le N+1-g +\mu(E)+\eta(E)= \frac{P_E(N)}{\rk(E)}+\eta(E)
\end{multline*}
As $E$ is an element of $Q_5(s)$, it is a quotient $\CC^{P(N)}\otimes \SO_{X_s}(-N) \twoheadrightarrow E$. Let $G$ be the image of $L\otimes \SO_{X_s}(-N)$ under such quotient. By construction of $Q_5$, it is clear that $F$ is the image of $G$ under the action of $\Lambda_{X_s,r}$
\begin{eqnarray*}
\xymatrixcolsep{3pc}
\xymatrixrowsep{2pc}
\xymatrix{
\Lambda_{X_s,r} \otimes \SO_{X_s}(-N)\otimes_\CC L \ar@{->>}[r] \ar@{->>}[d] & F \ar@{^{(}->}[d]\\
\Lambda_{X_s,r} \otimes G \ar[r] &E
}
\end{eqnarray*}
By Lemma \ref{lemma:submoduleSaturation}, $F^{\op{sat}}=\im(\Lambda_{X_s,r}\otimes G \to E)^{\op{sat}}$ with the induced parabolic structure is a parabolic sub-$\Lambda$-module of $E$. Therefore, by Lemma \ref{lemma:sharpStabInequality}
$$\frac{h^0(X_s,F(N))}{\rk(F)} + \eta(F) \le \frac{P_E(N)}{\rk(E)}+\eta(E)$$
Let us distinguish two cases. If the inequality is an equality, then by Lemma \ref{lemma:sharpStabInequality} we know that $h^1(X_s,F(N))=0$, so $\tau_F(N)=0$ and we are done.
Otherwise, by Lemma \ref{lemma:slopeDiffbound}, there exists a number $\delta>0$ such that
$$\frac{h^0(X_s,F(N))}{\rk(F)} + \eta(F)+\delta \le \frac{P_E(N)}{\rk(E)}+\eta(E)=\frac{h^0(X_s,E(N))}{\rk(E)}+\eta(E)$$
As we have seen, the set of possible built subsheaves $F$ is a bounded family over a projective curve. Therefore, by \cite[Theorem 3.13]{Kleiman71}, the set of possible values for $\mu(F(N))$ is bounded from below. As we have
\begin{multline*}
\tau_F(N)=\frac{h^0(X_s,F(N))}{\rk(F)} - (N+1-g)-\mu(F)\\
\le \frac{P_E(N)}{\rk(E)}-(N+1-g) + \eta(E)-\eta(F) -\mu(F) \le \frac{P_E(N)}{\rk(E)}+\eta(E)-(N+1-g)-\min_F\{\mu(F)\}
\end{multline*}
the set of possible values of $\tau_F(N)$ is bounded from above. Let $\tau$ be the maximum of such values. As $N$ is fixed, there exists an $m_1\ge m_0$ such that for $m\ge m_1$
$$\varepsilon_0 \tau_F(N)\le \varepsilon_0 \tau < \delta$$
Therefore, for $m\ge m_1$, if $(E,E_\bullet,\varphi)$ is semistable then
\begin{multline*}
\frac{h^0(X_s,F(N))}{\rk(F)} + \eta(F) + \varepsilon_0 \tau_F(N) < \frac{h^0(X_s,F(N))}{\rk(F)} + \eta(F) + \delta \\
 \le \frac{P_E(N)}{\rk(E)}+\eta(E)
\end{multline*}
Therefore, condition \eqref{eq:GITLambdaModules4} holds for every $L\subseteq \CC^{P(N)}$ and $(E,E_\bullet)$ is GIT-semistable under embedding $\psi_m$ for every $m\ge m_1$.

Suppose that $(E,E_\bullet,\varphi,\alpha)$ is strictly GIT semistable. Let $L\subseteq \CC^{P(N)}$ such that we have an equality in \eqref{eq:GITLambdaModules4}. Take $F=\im\left( \Lambda_{X_s,r}\otimes \SO_{X_s}(-N)\otimes_{\CC} L \to E \right)$ as before. As $L\subseteq H^0(X_s, F(N))$, we may assume without loss of generality that $L=H^0(X_s,F(N))$. Then
$$\frac{h^0(X_s,F(N))}{\rk(F)}+\eta(F)+\epsilon_0 \tau_F(N)=\frac{P_E(N)}{\rk(E)}+\eta(E)$$
If we had $\tau_F(N)\ne 0$ then, in particular,
$$\frac{h^0(X_s,F(N))}{\rk(F)}+\eta(F)<\frac{P_E(N)}{\rk(E)}+\eta(E)$$
So, using Lemma \ref{lemma:slopeDiffbound}  as $0<\epsilon_0\tau_F(N)<\delta$
\begin{multline*}
\frac{h^0(X_s,F(N))}{\rk(F)} + \eta(F) + \varepsilon_0 \tau_F(N) < \frac{h^0(X_s,F(N))}{\rk(F)} + \eta(F) + \delta \\
 \le \frac{P_E(N)}{\rk(E)}+\eta(E)
\end{multline*}
contradicting the equality assumption in \eqref{eq:GITLambdaModules4}. Therefore, $\tau_F(N)=0$ and we obtain
$$\frac{h^0(X_s,F(N))}{\rk(F)}+\eta(F) = \frac{P_E(N)}{\rk(E)}+\eta(E)$$
By Lemma \ref{lemma:sharpStabInequality}, $F$ must be saturated and, therefore, $(F,F_\bullet,\varphi)$ is a parabolic sub-$\Lambda$-module with $\pmu(F)=\pmu(E)$, so $(E,E_\bullet,\varphi)$ is strictly semistable.

\end{proof}

\begin{theorem}
\label{thm:LambdaModModuli}
Let $\SM(\Lambda,P,\alpha,\overline{r}): (\Sch_S) \to (\Sets)$ denote the functor that associates each $S$-scheme $f:T\to S$ the set of isomorphism classes of semistable parabolic $\Lambda$-modules over $C\times T$ with Hilbert polynomial $P$ and the given parabolic type modulo S-equivalence and tensoring by a line bundle over $T$. Let $\SM^s(\Lambda,P,\alpha,\overline{r})$ be the subfunctor corresponding to isomorphism classes of stable parabolic $\Lambda$-modules. There exists a quasi-projective variety $M(\Lambda,P,\alpha,\overline{r})$ such that
\begin{enumerate}
\item $M(\Lambda,P,\alpha,\overline{r})$ is a coarse moduli space for $\SM(\Lambda,P,\alpha,\overline{r})$.
\item There is an open subscheme $M^s(\Lambda,P,\alpha,\overline{r}) \subseteq M(\Lambda,P,\alpha,\overline{r})$ which is a coarse moduli space for the functor $\SM^s(\Lambda,P,\alpha,\overline{r})$. Moreover, it admits a locally universal family in the \'etale topology.
\end{enumerate}
\end{theorem}

\begin{proof}
Let $\Theta$ be any very ample line bundle over $R$ for which Lemma \ref{lemma:GITLambdaModules} holds. Then Lemma \ref{lemma:GITLambdaModules} implies that the subscheme of GIT-(semi-)stable points of $R$ is $R^s$ (respectively $R^{ss}$). Take the GIT quotient $M(\Lambda,P,\alpha,\overline{r}):=R^{ss}\gitq \op{SL}_{P(N)}(\CC)$. By the discussion at the start of this section, $\SM(\Lambda,P,\alpha,\overline{r})$ is the quotient of the functor $\SR^{ss}$ described in Theorem \ref{thm:LambdaModParam} by $\op{SL}_{P(N)}(\CC)$. By \cite[Proposition 1.11]{Si2}, $M(\Lambda,P,\alpha,\overline{r})$ is a universal categorical quotient of $R^{ss}$ by $\op{SL}_{P(N)}(\CC)$, the projection map $R^{ss}\to M(\Lambda,P,\alpha,\overline{r})$ is affine and $M(\Lambda,P,\alpha,\overline{r})$ is quasi-projective. Moreover, $M^s(\Lambda,P,\alpha,\overline{r}) \subseteq M(\Lambda,P,\alpha,\overline{r})$ is an open sub-scheme whose preimage under the quotient morphism is $R^s$ such that $R^{s}\to M^s(\Lambda,P,\alpha,\overline{r})$ is a universal geometric quotient. This proves that $M(\Lambda,P,\alpha,\overline{r})$ universally corepresents
$$\SR^{ss} / \op{SL}_{P(N)}(\CC) \cong \SM(\Lambda,P,\alpha,\overline{r})$$
and $M^s(\Lambda,P,\alpha,\overline{r})$ universally corepresents
$$\SR^{s} / \op{SL}_{P(N)}(\CC) \cong \SM^s(\Lambda,P,\alpha,\overline{r})$$
Therefore, in order to prove the Theorem it is enough to prove that the geometric points of $M(\Lambda,P,\alpha,\overline{r})$ coincide with the S-equivalence classes of parabolic $\Lambda$-modules. By \cite[Lemma 1.10]{Si2}, the closed points in $M(\Lambda,P,\alpha,\overline{r})$ are in one to one correspondence with the closed orbits of $\op{SL}_{P(N)}(\CC)$ in $R^{ss}$. Therefore, it is enough to prove that the orbit of any semistable parabolic $\Lambda$-module $(E,E_\bullet)$ contains its graduate $\Gr(E)$ by the Jordan-H\"older filtration and that the orbit of a graduate object $\Gr(E)$ is closed. The proof is then completely analogous to the one given for \cite[Theorem 1.21.3]{Si2} (used also in \cite[Theorem 4.7.3]{Si2} to prove the same property for non-parabolic $\Lambda$-modules).
\end{proof}

\section{Residual parabolic $\Lambda$-modules}
\label{section:residualLambdaModules}

One of the principal uses of enhancing a vector bundle with a parabolic structure is to control the behavior of some ``\textit{field}'' such as a logarithmic connection or a Higgs field near a puncture in a smooth Riemann surface. This is usually done through the control of some kind of \textit{residue} of the structure around the parabolic points of the surface.

Following the definitions of residue of a parabolic Higgs field or residue of a parabolic connection (parabolic $\SD_C$-module) given by Simpson \cite{SimpsonNonCompact}, the residue at a parabolic point is an endomorphism of the fiber of the underlying bundle over the point induced by the action of the field or connection respectively.

While studying the moduli spaces of these geometric structures and the correspondences between them, some restrictions over the residues of both the Higgs field and the connection appear naturally. For example, in the case of parabolic Higgs bundles, for every stable parabolic vector bundle $(E,E_\bullet)$, the cotangent space at $(E,E_\bullet)$ to the moduli space of parabolic vector bundles can be canonically identified through Serre duality with $H^0(\SParEnd(E,E_\bullet)\otimes K(D))$. Therefore, every element of the cotangent bundle corresponds to a stable parabolic Higgs bundle. All the parabolic Higgs fields $\varphi$ obtained this way share the property of being strongly parabolic, i.e., for every $x\in D$ and for every $i=1,\ldots,l_x$
\begin{equation}
\label{eq:HiggsResCondition}
\Res(\varphi,x)(E_{x,i})\subseteq E_{x,i+1}
\end{equation}
In fact, if we restrict ourselves to studying strongly parabolic Higgs fields, then the cotangent bundle of the moduli space of stable parabolic vector bundles fits as an open dense subset of the moduli space of stable strongly parabolic Higgs bundles. We can also understand the strongly parabolic condition as a control of the eigenvalues of the Higgs field at the parabolic points. In this sense, it is equivalent to imposing that the residue has a null spectrum.

On the other hand, through Simpson's correspondence we know that there exists an equivalence of categories between stable parabolic Higgs bundles and stable filtered $\SD_C$-modules \cite{SimpsonNonCompact}. If $(E,E_\bullet,\nabla)$ is a parabolic connection corresponding to an stable strongly parabolic Higgs bundle, then the action of the residue $\Res(\nabla,x)$ at the fiber $E|_x$ must satisfy the following two conditions
\begin{enumerate}
\item The eigenvalues of $\Res(\nabla,x)$ must coincide with the parabolic weights and,
\item $\Res(\nabla,x)$ must act on $E_{x,i}/E_{x,i+1}$ as multiplication by $\alpha_i$.
\end{enumerate}
These conditions can be reformulated as imposing that for every $x\in D$ and every $i$
\begin{equation}
\label{eq:ConnectionResCondition}
(\Res(\nabla,x)-\alpha_{x,i} \id_{E|_x})(E_{x,i})\subseteq E_{x,i+1} 
\end{equation}
As another example, in order to define parabolic $\lambda$-connections we must impose the following condition on the residue, serving as a sort of interpolation between the other two
$$(\Res(\nabla,x)-\lambda\alpha_{x,i} \id_{E|_x})(E_{x,i})\subseteq E_{x,i+1}$$
This will be explored in more detail in the last part of this paper (section \ref{section:lambdaConnections}).

Through this section we aim to give a suitable definition for the residue of a parabolic $\Lambda$-module that generalizes the ones given by Simpson \cite{SimpsonNonCompact} in the previous scenarios. Using this definition, we will present a ``\textit{residual}'' condition on parabolic $\Lambda$-modules that unifies the previous examples and we will show that the moduli of semistable residual parabolic $\Lambda$-modules exists as a closed subscheme of the one constructed in the previous section.

Let $\Lambda$ be a sheaf of rings of differential operators such that $\Lambda|_{\overline{D}}$ is locally free. Let $(E,E_\bullet,\varphi)$ be a parabolic $\Lambda$-module. By definition, for every parabolic point $x\in D$, the image of $\Lambda \otimes E(-\{x\}\times S)$ by the morphism $\varphi:\Lambda\otimes E \longrightarrow E$ lies in $E(-\{x\}\times S)$. Therefore, if $i_x:S \cong \{x\}\times S \hookrightarrow C\times S$ is the canonical inclusion we have a commutative diagram of sheaves of $(\SO_X,\SO_X)$-modules
\begin{eqnarray*}
\xymatrix{
& \Lambda \otimes_{\SO_X} E(-\{x\} \times S) \ar[r] \ar[d] & \Lambda \otimes_{\SO_X} E \ar[r] \ar[d]^{\varphi} & \Lambda \otimes_{\SO_X} (i_x)_*E|_{\{x\}\times S} \ar@{-->}[d] \ar[r] & 0\\
0 \ar[r] & E(-\{x\}\times S) \ar[r] & E \ar[r]^{\op{ev}} & (i_x)_*E|_{\{x\}\times S} \ar[r] & 0
}
\end{eqnarray*}
and we obtain a morphism
$$\varphi|_{\{x\}\times D}: \Lambda \otimes_{\SO_X} (i_x)_*E|_{\{x\}\times S}\longrightarrow (i_x)_*E|_{\{x\}\times S}$$
Taking the pullback by $i_x:S\hookrightarrow C\times S$  we obtain an induced morphism
$$\op{Res}(\varphi,x): \Lambda|_{\{x\}\times S} \otimes_{\SO_S} E|_{\{x\}\times S} \to E|_{\{x\}\times S}$$

\begin{definition}[Total residue of a $\Lambda$-module]
We call
$$\Res(\varphi,x):\Lambda|_{\{x\}\times S} \otimes_{\SO_S} E|_{\{x\}\times S} \longrightarrow E_{\{x\}\times S}$$
the total residue of the parabolic $\Lambda$-module $(E,E_\bullet,\varphi)$ at the point $x\in D$.
\end{definition}

As $\Lambda\otimes E\to E$ preserves the parabolic structure, then for every $x\in D$, $\Res(\varphi,x)$ preserves the parabolic filtration in the sense that for every $i=1,\ldots,l_x$
$$\Res(\varphi,x)\left(\Lambda|_{\{x\}\times S} \otimes_{\SO_S} E_{x,i}\right)\subseteq E_{x,i}$$

As an explicit example of this construction, let $S$ be a point and let us consider a parabolic connection $\nabla:E\longrightarrow E\otimes K_C(D)$. Let $x\in D$ be a parabolic point and let $z$ be a local coordinate in a neighborhood of $x$. Then there is a matrix $A_x$ such that $\nabla$ can be locally written as
$$\nabla(v)=A_xv \frac{dz}{z}+dv$$
where $v$ is a local section of $E$ around $x$. Now let us consider the operator $\nabla':T_C(-D)\otimes E \longrightarrow E$ obtained from $\nabla$ by contraction. It corresponds to the action of $\Lambda^{\op{DR},\log D}_1$ on $E$. If $\SX$ is a local section of $T_C(-D)$ and $v$ is a local section of $E$ around $x\in D$, then
$$\nabla'(\SX,v)=A_xv \SX(v)/z + \SX(v)$$
In particular, as $\SX$ is locally written as $\SX=tz\frac{\partial}{\partial z}$ for some local regular function $t$, then
$$\nabla'(\SX,v)=\nabla'\left(tz\frac{\partial}{\partial z},v\right) = tA_xv + tz\frac{\partial v}{\partial z}$$
Observe that the second summand of the right hand side is always a local section of $E(-x)$. Moreover, if $v\in E(-x)$ then the first factor also belongs to $E(-x)$, so $\nabla'$ sends $T_C(-D)\otimes E(-D)$ to $E(-D)$. As the evaluation of the second summand at $z=0$ is always $0$, the evaluation
$$\nabla'(\SX,v)|_{z=0}=\left(tA_x v \right)|_{z=0} = t(0)A_x v(0)$$
only depends on $\SX|_{z=0}$ and $v|_{z=0}$, so we obtain a morphism
\begin{eqnarray*}
\xymatrixcolsep{0.1pc}
\xymatrixrowsep{0.05pc}
\xymatrix{
{\Res(\nabla',x)|_{T_C(-D)|_x}}&:&T_C(-D)|_x \otimes E|_x & \cong &\SO_x \otimes E|_x \ar[rrrrrr] &&&&&& E|_x\\
&&tz\left.\frac{\partial}{\partial z} \right|_{z=0} \otimes v \ar@{|->}[rr] && t \otimes v \ar@{|->}[rrrrrr]&&&&&& tA_xv
}
\end{eqnarray*}
Notice that through the canonical isomorphism $\SO_x\otimes E|_x \cong E|_x$, the morphism $\Res(\nabla',x)|_{T_C(-D)|_x}\in \End(E_x)$ coincides with the usual notion of residue of $\nabla$, in the sense of the order $-1$ coefficient in the Laurent series of $\nabla$ around the point $x$.

Now let us consider the complete action of $\Lambda^{\op{DR},\log D}_1=\SO_C\oplus T_C(-D)$ on $E$. It is locally given by
$$\nabla'(f+\SX,v)=fv+\nabla'(\SX,v)$$
Where $f$ is a locally regular function around $x\in D$. Once more the restriction of $\nabla'(f+\SX,v)$ to $z=0$ only depends on $f(0)$, $\SX|_{z=0}$ and $v|_{z=0}$, so we obtain a morphism
\begin{eqnarray*}
\xymatrixcolsep{0.1pc}
\xymatrixrowsep{0.05pc}
\xymatrix{
{\Res(\nabla',x)|_{\Lambda_1^{\op{DR},\log D}|_x}}&:&{\Lambda_1^{\op{DR},\log D}|_x\otimes E|_x }& \cong  &(\SO_x \oplus \SO_x)\otimes E|_x \ar[rrrrrr] &&&&&& E|_x\\
&&\left(f+tz\left. \frac{\partial}{\partial z} \right|_{z=0}\right) \otimes v \ar@{|->}[rr] && (f,t) \otimes v \ar@{|->}[rrrrrr]&&&&&& fv+tA_xv
}
\end{eqnarray*}

Similarly, if $\varphi:E\longrightarrow E\otimes K_C(D)$ is a Higgs field, then around $x\in D$ it is locally given as
$$\varphi(v)=A_xv\frac{dz}{z}$$
for some matrix $A_x$. If we consider the induced morphism $\varphi':(\SO_C\oplus T_C(D))\otimes E\to E$ given by
$$\varphi'(f+\SX,v) = fv+A_xv \SX(z)/z$$
then the evaluation of this expression at $z=0$ clearly only depends on $f(0)$, $\SX|_{z=0}$ and $v|_{z=0}$ so it induces a morphism
\begin{eqnarray*}
\xymatrixcolsep{0.1pc}
\xymatrixrowsep{0.05pc}
\xymatrix{
{\Res(\varphi',x)|_{\Lambda_1^{\op{Higgs},\log D}|_x}}&:&{\Lambda_1^{\op{Higgs},\log D}|_x\otimes E|_x }& \cong  &(\SO_x \oplus \SO_x)\otimes E|_x \ar[rrrrrr] &&&&&& E|_x\\
&&\left(f+tz\left. \frac{\partial}{\partial z}\right |_{z=0}\right) \otimes v \ar@{|->}[rr] && (f,t) \otimes v \ar@{|->}[rrrrrr]&&&&&& fv+tA_xv
}
\end{eqnarray*}

Observe that if $(E,E_\bullet,\varphi)$ is a parabolic $\Lambda$-module and $R\in H^0(S,\Lambda|_{\{x\}\times S})$, then composing $R$ with $\Res(\varphi,x)$ yields an endomorphism of the fiber $E|_{\{x\}\times S}$ that preserves the parabolic filtration.
$$\Res_R(\varphi,x):E|_{\{x\}\times S}=\SO_X|_{\{x\}\times S} \otimes E|_{\{x\}\times S} \stackrel{f^*R}{\longrightarrow}f^*\Lambda|_{\{x\}\times S} \otimes E|_{\{x\}\times S} \stackrel{\Res(\varphi,x)}{\longrightarrow} E|_{\{x\}\times S}$$

\begin{definition}
Let $\Lambda$ be a sheaf of rings of differential operators over $C\times S$ flat over $S$ such that $\Lambda|_{\overline{D}}$ is locally free. A residual condition for $\Lambda$ over $D$ is a set of sections $\overline{R}=\{R_{x,i}\}$ with $R_{x,i}\in H^0(S,\Lambda|_{\{x\}\times S})$ for each $x\in D$ and each $i=1,\ldots,l_x$. Given a residual condition $\overline{R}$, for each $f:T\to S$ and each family of parabolic $\Lambda$-modules over $T$, $(E,E_\bullet,\varphi)$ we obtain morphisms
$$\Res_{R_{x,i}}(\varphi,x):E|_{\{x\}\times T}=\SO_X|_{\{x\}\times T} \otimes E|_{\{x\}\times T} \stackrel{f^*(R_{x,i})}{\longrightarrow}f^*\Lambda|_{\{x\}\times T} \otimes E|_{\{x\}\times T} \stackrel{\Res(\varphi,x)}{\longrightarrow} E|_{\{x\}\times T}$$
So we obtain an endomorphism $\Res_{R_{x,i}}(\varphi,x)\in H^0(T, \End_{\SO_T}(E|_{\{x\}\times T}))$. Moreover, as $\Lambda$ preserves the parabolic filtration, $\Res_{R_{x,i}}(\varphi,x)$ preserves the parabolic filtration over $x$ for every $i=1,\ldots,l_x$.
If $\overline{R}$ is a residual condition for $\Lambda$ over $D$, we say that a parabolic $\Lambda$-module $(E,E_\bullet)$ is $\overline{R}$-residual if for every $x\in D$ and every $i=1,\ldots,l_x$
$$\Res_{R_{x,i}}(\varphi,x)(E_{x,i}) \subseteq E_{x,i+1}$$
\end{definition}

For example, if we take
$$R_{x,i}^{\op{DR}}=\left(-\lambda_{x,i}, z \left.\frac{\partial}{\partial z} \right|_{z=0}\right) \in H^0(x,\Lambda_1^{\op{DR},\log D}|_x)\subsetneq H^0(x,\Lambda^{\op{DR},\log D}|_x)$$
Then for the connection $\nabla:E\to E\otimes K(D)$ in the previous example
$$\Res_{R_{x,i}^{\op{DR}}}(\nabla,x)=A_x-\lambda_i \id \in \End(E|_x)$$
Therefore a connection  is $\overline{R}^{\op{DR}}$-residual if
\begin{enumerate}
\item The eigenvalues of the residue at each parabolic point $x\in D$ are $\{\lambda_{x,i}\}$ respectively and
\item the residue acts on $E_{x,i}/E_{x,i+1}$ by multiplication by $\lambda_{x,i}$
\end{enumerate}
Therefore, we can control the eigenvalues of a connection through an $\overline{R}^{\op{DR}}$-residual condition. For example, taking $\lambda_{x,i}=\alpha_{x,i}$ we recover the residue condition of a parabolic connection described at the start of the section. This can be also achieved in the context of Higgs bundles. For example, if we take
$$R_{x,i}^{\op{Higgs}}=z \left. \frac{\partial}{\partial z}\right|_{z=0} \in H^0(x,\Lambda_1^{\op{Higgs},\log D}|_x)\subsetneq H^0(x,\Lambda^{\op{Higgs},\log D}|_x)$$
Then $\Res_{R_{x,i}^{\op{Higgs}}}(\varphi,x) = \varphi|_x$ so a Higgs bundle $\varphi:E\longrightarrow E\otimes K(D)$ is $\overline{R}^{\op{Higgs}}$-residual if and only if it is strongly parabolic.

\begin{theorem} 
\label{thm:ResLambdaModModuli}
Let $\Lambda$ be a sheaf of rings of differential operators over $C\times S$ flat over $S$ such that $\Lambda|_{\overline{D}}$ is locally free.  Let $\overline{R}$, be a residual condition for $\Lambda$ over $D$. Let $\SM(\Lambda, \overline{R},P,\alpha,\overline{r}): (\Sch_S) \to (\Sets)$ denote the functor that associates each $S$-scheme $f:T\to S$ the set of isomorphism classes of semistable $\overline{R}$-residual parabolic $\Lambda$-modules over $C\times T$ with Hilbert polynomial $P$ and the given parabolic type modulo S-equivalence and tensoring by a line bundle over $T$. Let $\SM^s(\Lambda,\overline{R},P,\alpha,\overline{r})$ be the subfunctor corresponding to classes of stable parabolic $\Lambda$-modules. There exists a quasi-projective variety $M(\Lambda,\overline{R},P,\alpha,\overline{r})$ such that
\begin{enumerate}
\item $M(\Lambda,\overline{R},P,\alpha,\overline{r})$ is a coarse moduli space for the functor $\SM(\Lambda,\overline{R},P,\alpha,\overline{r})\subset \SM(\Lambda,P,\alpha,\overline{r})$.
\item There is an open subscheme $M^s(\Lambda,\overline{R},P,\alpha,\overline{r}) \subseteq \SM(\Lambda,\overline{R},P,\alpha,\overline{r})$ which is a coarse moduli space for the functor $\SM^s(\Lambda,\overline{R},P,\alpha,\overline{r})$. Moreover, it admits a locally universal family in the \'etale topology. 
\end{enumerate}
\end{theorem}

\begin{proof}
Let $R$ be the quasi-projective scheme described by Theorem \ref{thm:LambdaModParam}. Let $(\SE,\SE_\bullet)$ be the $\Lambda$-module underlying the corresponding universal object and $\varphi^{\op{univ}}:\pi^*\Lambda \otimes \SE \to \SE$ be the universal action of $\Lambda$ on $\SE$. Let $f:T\to S$. A family of rigidified parabolic $\Lambda$-modules $(E,E_\bullet,\varphi,\alpha)$ corresponding to a $T$ point $e:T\to R$ is $\overline{R}$-residual if for every $x\in D$ and $1\le i\le l_x$ the following composition of morphisms
\begin{eqnarray*}
\xymatrixcolsep{1.4pc}
\xymatrix{
e^*\SE_{x,i} \ar@{^{(}->}[r] & e^*\SE|_{x,1} \ar^-{f^*(R_{x,i})}[rr] && f^*(\Lambda|_{\{x\}\times S})\otimes e^*\SE_{x,1} \ar^-{e^*\Res(\varphi^{\op{univ}},x)}[rrr] &&& e^*\SE_{x,1} \ar@{->>}[r] & e^*(\SE_{x,1}/\SE_{x,i+1})
}
\end{eqnarray*}
is zero. By \cite[Lemma 4.3]{Yokogawa93}, there is a closed subscheme $R_{\Res}\subseteq R$ parameterizing $\overline{R}$-residual $\Lambda$-modules of $R$. Let $R_{\Res}^{ss}:= R^{ss}\cap R_{\Res}$ and $R_{\Res}^s:=R^s\cap R_{\Res}$. It is clear that $R_{\Res}$ is invariant under the action of $\op{SL}_{P(N)}(\CC)$ on $R$. Therefore, the quotient of $R^{ss}$ by $\op{SL}_{P(N)}(\CC)$ restricts to a quotient of $R_{\Res}^{ss}$.

Theorem \ref{thm:LambdaModModuli} applies and we get that the closed subscheme
$$M(\Lambda,\overline{R},P,\alpha,\overline{r})=R_{\Res}^{ss}\gitq \op{SL}_{P(N)}(\CC)\hookrightarrow R^{ss} \gitq \op{SL}_{P(N)}(\CC)=M(\Lambda,P,\alpha,\overline{r})$$
is a coarse quasi-projective moduli space for the given moduli functor $\SM(\Lambda,\overline{R},P,\alpha,\overline{r})$. Similarly $M^s(\Lambda,\overline{R},P,\alpha,\overline{r})\subseteq M^s(\Lambda,P,\alpha,\overline{r})$ is a quasi-projective coarse moduli space for $\SM^s(\Lambda,\overline{R},P,\alpha,\overline{r})$

The existence of a local universal family in the \'etale topology on $M^s(\Lambda,\overline{R},P,\alpha,\overline{r})$ is obtained by restriction of local universal families over $M^s (\Lambda,P,\alpha,\overline{r})$.
\end{proof}

\section{Existence of a universal family}
\label{section:fineModuli}

In this section we prove that, under certain generic hypothesis,  the moduli spaces of stable parabolic $\Lambda$-modules previously constructed are fine moduli spaces, i.e., there exists a universal family and the corresponding scheme represents the moduli functor. This will be achieved generalizing the proof in \cite[Section 3]{BY}.

\begin{lemma}
Let $(E,E_\bullet)$ be a parabolic vector bundle on $(C,D)$ such that the underlying vector bundle $E$ is of type $e$, degree $d$ and rank $r$. Let $(H,H_\bullet)$ be a parabolic line bundle of degree $h$. Let $M=|D|$. If
$$d>rh+r(2g-2+M)+r(r-1)e$$
then $h^1(\ParHom((H,H_\bullet),(E,E_\bullet)))=0$.
\end{lemma}

\begin{proof}
By parabolic Serre duality we obtain
\begin{multline*}
h^1\left(\ParHom((H,H_\bullet),(E,E_\bullet)) \right) = h^0\left(\SParHom((E,E_\bullet),(H,H_\bullet))\otimes K(D)\right)\\
\le h^0\left(\ParHom ((E,E_\bullet),(H,H_\bullet))\otimes K(D)\right)
\end{multline*}
Where $\ParHom$ and $\SParHom$ denote the sheaves of parabolic morphisms and strongly parabolic morphisms respectively.

Let $\varphi:(E,E_\bullet)\to (H,H_\bullet)\otimes K(D)$ be a nonzero parabolic morphism and let $(K,K_\bullet)$ be the kernel of the morphism $\varphi$ endowed with the induced parabolic structure from $(E,E_\bullet)$. Then yields
$$\deg(K)\ge \deg(E) - \deg\left (H\otimes K(D)\right) = d-h-(2g-2+M)$$
On the other hand, as $(E,E_\bullet)$ is of type $e$ we have
$$\mu(K)=\frac{\deg(K)}{r-1}\le \mu(E)+e = \frac{d}{r}+ e$$
Solving for $\deg(K)$ in the second inequality and substituting in the first one yields
$$(r-1)d+r(r-1)e\ge r\deg(K) \ge rd-rh-r(2g-2+M)$$
Therefore
$$d\le rh+r(2g-2+M)+r(r-1)e$$
\end{proof}

Now let $(\SE,\SE_\bullet,\Phi,A)$ be the universal rigidified $\Lambda$-module over $R^s$ defined in Theorem \ref{thm:LambdaModParam}. It is clear that the action of $\op{SL}_{P(N)}(\CC)$ on $(\SE,\SE_\bullet,\Phi)$ is trivial, but the center of $\GL_{P(N)}(\CC)$ acts on $(\SE,\SE_\bullet,\Phi)$ by dilatations on the fibers. Observe that the action of $\CC*$ on the action $\Phi$ is also trivial. Therefore, if there exists a line bundle $L$ over $R^s$ with a natural lift of the $\GL_{P(N)}(\CC)$ action such that the center $\CC^*$ acts by multiplication on $L$, then the rigidified $\Lambda$-module
$$\left(\SE\otimes \pi_{R^s}^*L^{-1},\SE_\bullet\otimes \pi_{R^s}^*L^{-1},\Phi\otimes \id,\alpha\otimes \id \right)$$
is $\GL_{P(N)}(\CC)$-equivariant and, therefore, it projects to a universal family of stable parabolic $\Lambda$-modules over $\SM^s(\Lambda,P,\alpha,\overline{r})$.

\begin{lemma}
\label{lemma:universalBundle}
Given a parabolic type $\overline{r}=\{r_{x,i}\}$, let
$$m_{x,i}=r_{x,i+1}-r_{x,i}$$
for $i=1,\ldots,l_x$. If the great common divisor of the numbers $\{d,m_{x,i} | x\in D, 1<i\le l_x\}$ is one, then there exists a line bundle $L$ over $R^s$ with natural action of $\GL_{P(N)}(\CC)$ such that the elements $\gamma \cdot \id\in \GL_{P(N)}(\CC)$ act by multiplication by $\gamma$.
\end{lemma}

\begin{proof}
From Lemma \ref{lemma:LambdaModBoundC}, we know that there is some $e\in \RR$ such that every stable parabolic $\Lambda$-module underlying a point in $R^s$ is of type $e$. Take a line bundle $H$ with
$$\deg(H) =h < \frac{d}{r}-2g+2-M-(r-1)e$$
For every choice $\kappa:D\to \ZZ$ such that $1\le \kappa(x)\le l_x+1$, set
$$\beta_x\left(=\beta_{x,1}\right)=\left \{\begin{array}{ll}
\alpha_{x,\kappa(x)} & \text{if }\kappa(x)\le l_x\\
\frac{1+\alpha_{x,l_x}}{2} & \text{if }\kappa(x)=l_x+1
\end{array}\right.$$
and endow $H$ with the trivial parabolic structure for the system of weights $\beta$. Let us denote it by $(H,H_\bullet^\kappa)$. Moreover, let
$$\chi(\kappa,h)=d+r(1-g-h)-\sum_{x\in D} \sum_{i=1}^{\kappa(x)-1} m_{x,i}$$
By the previous lemma, for every stable parabolic $\Lambda$-module $(E,E_\bullet,\varphi)$ we have $h^1\left(\ParHom((H,H^\kappa_\bullet),(E,E_\bullet)) \right) =0$. Applying Riemann-Roch theorem yields
$$h^0\left(\ParHom((H,H^\kappa_\bullet),(E,E_\bullet)) \right) = \chi(\kappa,h)$$ 
so $\SE'=H^0\left(C\times R^s/R^s,\ParHom(\pi_C^*(H,H^\kappa_\bullet),(\SE,\SE_\bullet)) \right)$ is a locally free sheaf of rank $\chi(\kappa,h)$ over $R^s$. Let $L(\kappa,h)$ be its determinant. By construction $\GL_{P(N)}(\CC)$ acts on $\SE'$ and there is an induced action on $L(\kappa,h)$ such that $\gamma\cdot \id$ acts as multiplication by $\gamma^{\chi(\kappa,h)}$.

Given $a_1,\ldots,a_M\in \ZZ$, $\kappa_1,\ldots,\kappa_M:D\to\ZZ$ and $h_1,\ldots,h_M\in \ZZ$, consider the bundle
$$\bigotimes_{i=1}^M L(\kappa_i,h_i)^{a_i}$$
with the induced $\GL_{P(N)}(\CC)$ action. Then $\gamma\cdot \id$ acts as multiplication by $\gamma^{\sum_{i=1}^M a_i\chi(\kappa_i,h_i)}$. Therefore it is enough to prove that there exist $a_i$, $\kappa_i$ and $h_i$ such that $\sum_{i=1}^M a_i\chi(\kappa_i,h_i)=1$. Let $\kappa_{x,i},\kappa_{x,i}^+:D\to \ZZ$ be given by
$$\begin{array}{l}
\kappa_{x,i}(y)=\left \{ \begin{array}{ll}
i-1& x=y\\
0 & x\ne y
\end{array} \right. \\
\kappa_{x,i}^+(y)=\left \{ \begin{array}{ll}
i& x=y\\
0 & x\ne y
\end{array} \right.
\end{array}$$
Then for every $h$ yields
$$\begin{array}{l}
\chi(\kappa_{x,i}^+,h)-\chi(\kappa_{x,i},h)=m_{x,i}\\
\chi(\kappa,h)-\chi(\kappa,h-1)=r\\
\chi(0,h)-(1-g-h)\left(\chi(0,h)-\chi(0,h-1)\right) = d
\end{array}$$
As $\op{GCD}(\{m_{x,i},r,d\})=\op{GCD}(\{m_{x,i},d\})=1$ the lemma follows.
\end{proof}

The previous discussion leads to the following theorem.

\begin{theorem}
\label{thm:fineModuli}
For every $x\in D$ and $i=1,\ldots,l_x$ let $m_{x,i}=r_{x,i+1}-r_{x,i}=\dim(E_{x,i})-\dim(E_{x,i+1})$. Let $d$ be the degree of the underlying vector bundle, so that $P(m)=r(m+1-g)+d$. If the great common divisor of the numbers $\{d,m_{x,i} | x\in D, 1<i\le l_x\}$ is one then
\begin{enumerate}
\item The moduli space of stable parabolic $\Lambda$-modules is a fine moduli space, i.e., there exists a universal family $(\SE,\SE_\bullet,\Phi)$ over $\SM^s(\Lambda,P,\alpha,\overline{r})$.
\item For every residual condition $\overline{R}$, the moduli space of stable residual parabolic $\Lambda$-modules is a fine moduli space, i.e., there exists a universal family $(\SE,\SE_\bullet,\Phi)$ over $\SM^s(\Lambda,\overline{R},P,\alpha,\overline{r})$.
\end{enumerate}
\end{theorem}

\begin{proof}
If the coprimality condition holds then the discussion at the start of the section combined with Lemma \ref{lemma:universalBundle} proves that $\SM^s(\Lambda,P,\alpha,\overline{r})$ admits a universal parabolic $\Lambda$-module. Restricting it to the closed subscheme $\SM^s(\Lambda,\overline{R},P,\alpha,\overline{r})\subseteq \SM^s(\Lambda,\overline{R},P,\alpha,\overline{r})$ we obtain the second desired universal family.
\end{proof}

Observe that, in particular, if take $\Lambda$ to be trivial then we recover precisely the numerical condition on $\overline{r}$ given by Boden and Yokogawa \cite[Proposition 3.2]{BY} for the existence of a universal family on the moduli space of parabolic vector bundles. Moreover, if the parabolic type is trivial and $\alpha=0$ the moduli space coincides with the moduli space of stable $\Lambda$-modules and the numerical condition reduces to asking for the rank and degree to be coprime. For trivial $\Lambda$ or for $\Lambda=\Lambda^{\op{Higgs}}$ this is known to be a necessary and sufficient condition \cite{RamananCoprimoFino}.

Notice that, while for $\Lambda$-modules on the compact case this coprimality condition implies that there does not exist any strictly semistable object, for nontrivial parabolic structures there exist non-generic systems of weights such that there exist strictly semistable $\Lambda$-modules in $M(\Lambda,P,\alpha,\overline{r})$ and simultaneously the subscheme $M^s(\Lambda,P,\alpha,\overline{r})\subsetneq M(\Lambda,P,\alpha,\overline{r})$ admits a universal family. In particular, we have

\begin{corollary}
\label{cor:fineModuli}
If $\overline{r}$ is a full flag parabolic type then
\begin{enumerate}
\item $\SM^s(\Lambda,P,\alpha,\overline{r})$ is a fine moduli space.
\item $\SM^s(\Lambda,\overline{R},P,\alpha,\overline{r})$ is a fine moduli space.
\end{enumerate}
\end{corollary}

\section{Moduli space of parabolic $\lambda$-connections}
\label{section:lambdaConnections}

Let $\xi$ be a line bundle over $C$, let $\alpha$ be a fixed system of weights over $D$ and $\overline{r}=\{r_{x,i}\}$ a parabolic type. Let us suppose that $\deg(\xi)=- \sum_{x\in D}\sum_{i=1}^{l_x} \alpha_{x,i}$. Fixing a line bundle and a system of weights $\alpha$ over $C$ allows us to describe canonically a parabolic line bundle over $C$, $(\xi,\xi_\beta)$, taking the underlying vector bundle as $\xi$ and defining trivial filtrations over each $x\in D$ with parabolic weight
$$\beta_x:=\beta_{x,1}=\sum_{i=1}^{l_x} \alpha_x,i (r_{x,i+1}-r_{x,i})$$
As $\xi$ has rank one, any parabolic structure on $\xi$ consists of trivial filtrations. It is possible that for some $x\in D$, $\beta_x\ge 1$. Taking into account the definition for the parabolic structure in terms of left continuous filtrations given by Simpson \cite{SimpsonNonCompact}, a parabolic line bundle $\xi$ with jumps at weights $\beta_x$ for each $x\in D$ such that $\xi_{\beta_x,x}=\xi_x$ is the same as a trivial filtration for the bundle
\begin{equation}
\xi \left (\sum_{x\in D} \lfloor \beta_x \rfloor x \right)
\end{equation}
with parabolic weights $\{\beta_x-\lfloor \beta_x \rfloor\}_{x\in D}$.

Thus, the value of the jump $\beta_x$ completely defines the parabolic structure on $\xi$. By construction, we get that

$$\pdeg(\xi)=\deg(\xi) + \sum_{x\in D}\beta_x=\deg(\xi)+\sum_{x\in D}\sum_{i=1}^{l_x} \alpha_{x,i}=0$$

The line bundle $\xi$ can be given the structure of a parabolic Higgs bundle canonically taking a zero Higgs field. In fact, as the rank of $\xi$ is one, every traceless Higgs field over $\xi$ must be zero, so $\SM_{\op{Higgs}}(1,\beta,\xi)$ consists exactly of the point $(\xi,\xi_\beta,0)$.

Let $(E,E_\bullet,\Phi)$ be a traceless strongly parabolic $\op{SL}_{r}(\CC)$-Higgs bundle with parabolic system of weights $\alpha$ such that $\det(E)=\xi$. Taking the $r$-th exterior power, the morphism $\Phi$ induces a morphism $\bigwedge^r E\to \bigwedge^r E\otimes K(D)$ locally given by the trace of $\Phi$. As $\tr(\Phi)=0$, the induced morphism is the zero morphism.

Thus, taking the determinant, every parabolic Higgs bundles $[(E,E_\bullet,\Phi)]\in \SM_{\op{Higgs}}(r,\alpha,\xi)$ induces the same parabolic Higgs bundle $(\xi,\xi_\beta,0)$.

Using the Simpson correspondence \cite{SimpsonNonCompact} between parabolic Higgs bundles of parabolic degree 0 and parabolic connections of parabolic degree 0, the parabolic Higgs bundle $(\xi,\xi_\beta,0)$ corresponds to a parabolic connection $(\xi,\xi_\beta,\nabla_{\xi,\beta})$ with the same parabolic weights $\beta$, such that $\op{Res}(\nabla_{\xi,\beta},x)=\beta_x \op{Id}$ for every $x\in D$.

Let $(E',E'_\bullet,\nabla)$ be the parabolic connection corresponding to the Higgs bundle $(E,E_\bullet,\Phi)$ under the Simpson correspondence. Taking the $r$-th exterior power, $\nabla$ induces a morphism
$$\tilde{\nabla}:\bigwedge^r E\to \bigwedge^r E\otimes K(D) \quad .$$

As the Simpson correspondence is an equivalence of categories preserving the exterior product \cite[Theorem 2]{SimpsonNonCompact}, the wedge product of $(E',E'_\bullet,\nabla)$ must be the image of the wedge product of $(E,E_\bullet,\Phi)$. Therefore, the morphism $\tilde{\nabla}$ must coincide with $\nabla_{\xi,\beta}$. This leads up to the following definition of parabolic $\lambda$-connection for the group $\op{SL}_r(\CC)$.

\begin{definition}
\label{def:parabolicLambdaConnection}
For a fixed line bundle $\xi$, a system of weights $\alpha$ and a given $\lambda \in \CC$ a parabolic $\lambda$-connection on $C$ (for the group $\op{SL}_r(\CC)$) is a quadruple $(E,E_\bullet,\nabla,\lambda)$ where

\begin{enumerate}
\item $\lambda$ is a complex number.
\item $(E,E_\bullet)\longrightarrow C$ is a parabolic vector bundle of rank $r$ and weight system $\alpha$ together with an isomorphism $\bigwedge^r E\cong \xi$.
\item $\nabla : E\to E\otimes K(D)$ is a $\CC$-linear homomorphism of sheaves over the underlying vector space of $E$ satisfying the following conditions
\begin{enumerate}
\item If $f$ is a locally defined holomorphic function on $C$ and $s$ is a locally defined holomorphic section of $E$ then
$$\nabla (fs) = f\cdot \nabla(s) + \lambda \cdot s\otimes df$$
\item For each $x\in D$ the homomorphism induced in the filtration over the fiber $E_x$ satisfies
$$\nabla(E_{x,i}) \subseteq E_{x,i} \otimes K(D)|_x$$
\item For every $x\in D$ and every $i=1,\ldots, l_x$ the action of $\op{Res}(\nabla,x)$ on $E_{x,i}/E_{x,i+1}$ is the multiplication by $\lambda \alpha_{x,i}$. Since $\op{Res}(\nabla,x)$ preserves the filtration, it acts on each quotient.
\item The operator $\bigwedge^r E \longrightarrow \left(\bigwedge^r E \right) \otimes K(D)$ induced by $\nabla$ coincides with $\lambda \cdot \nabla_{\xi,\beta}$.
\end{enumerate}
\end{enumerate}
\end{definition}

We also have the following natural  notion of stability for $\lambda$-connections.

\begin{definition} A parabolic $\lambda$-connection $(E,E_\bullet,\nabla)$ is (semi-)stable if and only if for every parabolic subsheaf $(F,F_\bullet)\subseteq (E,E_\bullet)$ preserved by $\nabla$
$$\pmu(F)(\le) < \pmu(E)$$
\end{definition}

Given a parabolic $\lambda$-connection $(E,E_\bullet,\nabla,\lambda)$, the connection induces a parabolic morphism $\nabla': (K(D))^\lor \otimes E \to E$ that satisfies that for every local section $v$ of $(K(D))^\lor$, each locally defined holomorphic function on $C$ and each local section $f$ of $E$
\begin{equation}
\label{eq:lambdaLeibnitz}
\nabla'(v \otimes (fs)) = f \nabla'(v\otimes s) + \lambda df(v) \cdot s
\end{equation}
Then we get a morphism $(\id|_E \oplus \nabla'): (\SO_C\oplus (K(D))^\lor) \otimes E \to E$. Let us consider the bimodule structure on $\SO_C \oplus (K(D))^\lor$ given by
\begin{equation}
\label{eq:lambdaModProduct}
\begin{array}{ccc}
(g, v)\cdot f &=& (fg + \lambda df(v) , fv)\\
f\cdot (g,v)&=&(fg,fv)
\end{array}
\end{equation}
where $f,g$ are local sections of $\SO_C$ and $v$ is a local section of $(K(D))^\lor$ over the same open subset of $C$.

Then it becomes clear that requiring $\nabla'$ to satisfy equation \eqref{eq:lambdaLeibnitz} is equivalent to asking morphism $\nabla''=(\id|_E \oplus \nabla'): (\SO_C\oplus (K(D))^\lor) \otimes_{\SO_X} E \to E$ to be a $(\SO_X,\SO_X)$-module morphism for the previous bimodule structure of $\SO_C \oplus (K(D))^\lor$, as then for every local sections $f,g \in \SO_C(U)$, $v\in(K(D))^\lor(U)$ and $s\in E(U)$ over each an open set $U$.
\begin{multline*}
\nabla''( (g,v) \otimes (fs)) =  \nabla'' \left( ((g,v)\cdot f )\otimes s \right) = \nabla'' \left ((fg+\lambda df(v), fv) \otimes s \right )\\
= fgs + \lambda df(v) s + f\nabla'(v\otimes s)
\end{multline*}

From the previous explicit product formula, applying \cite[Theorem 2.11]{Si2} it becomes clear that for each $\lambda$, the sheaf $\Lambda_1^{\op{DR},\log D,\lambda}=\SO_C \oplus (K(D))^\lor$ with the given bimodule structure extends to a (split quasi-polynomial) sheaf of rings of differential operators $\Lambda^{\op{DR},\log D,\lambda}$ over $C$. Now let us consider the product $X=C\times \mathbb{A}^1$, and let $p_1:X\to C$ be the canonical projection. Then, $\Lambda^{\op{DR},\log D,R}_1 := \SO_X \oplus p_1^* ((K(D))^\lor)$ can be given a bimodule structure patching together the previous one for each value of $\lambda \in \mathbb{A}^1$. In  particular, as $\SO_X \cong \SO_C \otimes_{\CC} \SO_{\mathbb{A}^1}$, we define the right action by $\SO_X$ as
\begin{equation}
\label{eq:deformationGraduateProduct}
((g,\lambda),v)\cdot (f\otimes \nu) = (\nu f g + \lambda \nu df(v), \nu f v)
\end{equation}

Again, applying \cite[Theorem 2.11]{Si2} we can extend the bimodule structure to a (split quasi-polynomial) sheaf of rings of differential operators over $X$ flat over $\mathbb{A}^1$. By construction, it coincides to the deformation to the graduate of $\Lambda^{\op{DR},\log D}$. See \cite[Section 2, p. 41]{Si2} for the general construction of the deformation to the graduate for a split quasi-polynomial sheaf of rings of differential operators.

Conditions (2), (3.a) and (3.b) of the definition imply that a parabolic $\lambda$-connection is a parabolic $\Lambda^{\op{DR},\log D,R}$-module over $\Spec(\CC)$ with fixed determinant $\xi$. Let us fix once and for all an isomorphism $\SO_{\mathbb{A}^1} \cong \CC$. Let $\lambda_{x,i}=\lambda \cdot \alpha_{x,i} \in H^0(\mathbb{A}^1,\SO_X|_{\{x\}\times \mathbb{A}^1}) \cong H^0(\mathbb{A}^1, \SO_{\mathbb{A}^1})$. Let us denote $\overline{\lambda}=\{\lambda_{x,i}\}$. Let $z$ be a local coordinate around $x$. Then let
\begin{multline*}
R_{x,i}^\lambda = \left (-\lambda_{x,i},  \pi_C^* \left( z\frac{\partial}{\partial z} \right) \right) \in H^0\left(\mathbb{A}^1, \SO_{\mathbb{A}^1} \oplus \pi_C^*T_C(-D)|_{\{x\}\times \mathbb{A}^1}\right) \\
=H^0\left(\mathbb{A}^1,\Lambda_1^{\op{DR},\log D,R}|_{\{x\}\times \mathbb{A}^1}\right) \subsetneq H^0\left(\mathbb{A}^1,\Lambda^{\op{DR},\log D,R}|_{\{x\}\times \mathbb{A}^1}\right)
\end{multline*}
Then $\overline{R}^\lambda=\{R_{x,i}^\lambda\}$ is a residual condition for $\Lambda^{\op{DR},\log D,R}$. From the definition, it is clear that condition (3.c) is equivalent to requiring the parabolic $\Lambda^{\op{DR},\log D,R}$-module to be $\overline{R}$-residual.

If $\nabla'':\Lambda^{\op{DR},\log D,R}\otimes E \to E$ is a parabolic $\Lambda^{\op{DR},\log D,R}$-module, it induces a parabolic $\Lambda^{\op{DR},\log D,R}$-module
$$\tilde{\nabla}'':\Lambda^{\op{DR},\log D,R}\otimes \bigwedge^rE \to \bigwedge^rE$$
Let us consider the rank one $\lambda$-connection $\lambda\cdot\nabla_{\xi,\alpha} : \xi \to \xi \otimes K(D)$ over the fixed determinant bundle. As before, it induces a $\Lambda^{\op{DR},\log D,R}$-module
$$\lambda\cdot\nabla''_{\xi,\alpha}: \Lambda^{\op{DR},\log D,R}\otimes \xi \to \xi$$
Condition (3.d) is equivalent to requiring $\tilde{\nabla}''$ to coincide with morphism $\lambda\cdot\nabla''_{\xi,\alpha}$ under the isomorphism $\bigwedge^r E\cong \xi$.
\begin{eqnarray}
\label{eq:ConnectionDeterminantDiagram}
\xymatrixcolsep{3pc}
\xymatrixrowsep{3pc}
\xymatrix{
{\Lambda^{\op{DR},\log D,R}\otimes \bigwedge^r E} \ar[r]^-{\tilde{\nabla}''} \ar[d]^-{\wr} & {\bigwedge^r E} \ar[d]^-{\wr}\\
{\Lambda^{\op{DR},\log D,R}\otimes \xi} \ar[r]^-{\lambda\cdot\nabla''_{\xi,\alpha}} & \xi
}
\end{eqnarray}

Therefore, we can give the following alternative definition of a parabolic $\lambda$-connection

\begin{definition}
A parabolic $\lambda$-connection is a $\overline{R}^\lambda$-residual parabolic $\Lambda^{\op{DR},\log D,R}$-module over $\lambda:\Spec(\CC) \to \mathbb{A}^1$,  $\nabla'':\lambda^*\Lambda^{\op{DR},\log D,R} \otimes (E,E_\bullet) \to (E,E_\bullet)$ with $\bigwedge^r E \cong \xi$ such that the induced morphism
$$\tilde{\nabla}'':\lambda^*\Lambda^{\op{DR},\log D,R} \otimes \bigwedge^r E \to \bigwedge^r E$$
coincides with $\lambda\cdot\nabla''_{\xi,\alpha}: \lambda^*\Lambda^{\op{DR},\log D,R} \otimes \xi \to \xi$.
\end{definition}

Using this equivalent definition we can prove the following theorem.

\begin{theorem}
\label{thm:moduliLambdaConn}
Let $\SM_{\op{Hod}}(\xi,\alpha,\overline{r}):(\Sch_{\mathbb{A}^1}) \to (\Sets)$ denote the functor that associates each scheme $f:T\to \mathbb{A}^1$, the set of isomorphism classes of semistable parabolic $\lambda$-connections over $C\times T$ with determinant $\xi$ and the given parabolic type. Let $\SM_{\op{Hod}}^s(\xi,\alpha,\overline{r})$ be the subfunctor corresponding to classes of stable parabolic $\lambda$-connections. There exists a quasi-projective variety $M_{\op{Hod}}(\xi,\alpha,\overline{r})$ such that
\begin{enumerate}
\item $M_{\op{Hod}}(\xi,\alpha,\overline{r})$ corepresents the functor $\SM_{\op{Hod}}(\xi,\alpha,\overline{r})$.
\item The geometric points of $M_{\op{Hod}}(\xi,\alpha,\overline{r})$ are in bijection with the equivalence classes of semistable parabolic $\lambda$-connections with Hilbert polynomial $P$ and the given parabolic type on $C$ under the relation of $S$-equivalence.
\item There is an open subscheme $M_{\op{Hod}}^s(\xi,\alpha,\overline{r}) \subseteq M_{\op{Hod}}(\xi,\alpha,\overline{r})$ which is a coarse moduli space for the functor $\SM_{\op{Hod}}^s(\xi,\alpha,\overline{r})$.
\item Let $m_{x,i}=r_{x,i+1}-r_{x,i}$. If the great common divisor of $\{\deg(\xi),m_{x,i} | x\in D, 1<i\le l_x\}$ then there is a universal stable parabolic $\mathbb{A}^1$ family of $\lambda$-connections $(\SE,\SE_\bullet,\nabla^{\op{univ}})$ over $C\times M_{\op{Hod}}^s(\xi,\alpha,\overline{r})\times \mathbb{A}^1$. In particular, if $\overline{r}$ corresponds to a full flag parabolic type, $M_{\op{Hod}}^s(\xi,\alpha,\overline{r})$ is a fine moduli space.
\end{enumerate}
\end{theorem}

\begin{proof}
The parabolic structure $\overline{r}$ fixes the rank of the parabolic $\lambda$-connection and $\xi$ fixes its degree. As $C$ is a curve, this data uniquely determines the Hilbert polynomial $P$ of the $\lambda$-module.
Let $R_\lambda$ be the scheme constructed in the proof of Theorem \ref{thm:ResLambdaModModuli} for the given Hilbert polynomial and parabolic structure, taking $\Lambda=\Lambda^{\op{DR},\log D,R}$ over $C\times \mathbb{A}^1$ over $\mathbb{A}^1$ and the residual condition $\overline{R}^\lambda$. Let us consider the determinant morphism $\det:R_\lambda \to Jac(C)$ sending each geometric point of $R_\lambda$ $(E,E_\bullet,\varphi,\alpha)$ to $\det(E)=\bigwedge^r E$. We denote by $R_\lambda^{\xi}$ the pre-image of the point $\xi$ by this morphism. Therefore, it is a closed subscheme of $R_\lambda$ parameterizing locally free elements of $R_\lambda$ whose determinant is $\xi$. 

Let $f:T\to \mathbb{A}^1$ be a scheme. We say that a $T$-point $(E,E_\bullet,\varphi,\alpha)$ of $R_\lambda^{\xi}$,  satisfies condition $\det$ if the morphism
$$\tilde{\nabla}'': f^*\Lambda^{\op{DR},\log D,R} \otimes \bigwedge^r E \to \bigwedge^r E$$
induced by $\nabla'':\Lambda^{\op{DR},\log D,R}\otimes E \to E$ coincides with the morphism
$$f \cdot f^*\nabla''_{\xi,\alpha} : \pi^* \Lambda^{\op{DR},\log D, R} \otimes \pi_C^*\xi \to \pi_C^*\xi$$
under the isomorphism $\bigwedge^r E\cong \xi$.

Let $(\SE,\SE_\bullet)$ be the $\lambda$-residual parabolic $\Lambda^{\op{DR},\log D,R}$-module underlying the universal object in $R_\lambda^{\xi}$. Let $\pi:R_\lambda^{\xi} \to \mathbb{A}^1$. The action of $\Lambda^{\op{DR},\log D,R}$ on the universal object induces a morphism
$$\tilde{\nabla}''^{\op{univ}} : \pi^*\Lambda^{\op{DR},\log D,R} \otimes \bigwedge^r \SE \to \bigwedge^r \SE$$
On the other hand, by tanking the pullback to $R_\lambda^{\xi}$, we have a fixed morphism
$$\pi \cdot \pi^*\nabla''_{\xi,\alpha} : \pi^* \Lambda^{\op{DR},\log D,R} \otimes \pi_C^*\xi \to \pi_C^*\xi$$
Under the isomorphism $\bigwedge^r \SE\cong \pi_C^*\xi$, this induces a morphism
$$\pi \cdot \pi^*\nabla''_{\xi,\alpha} : \pi^* \Lambda^{\op{DR},\log D,R} \otimes \bigwedge^r \SE \to \bigwedge^r \SE$$

Then a $T$-point of $R_\lambda^{\xi}$ given by $e:T\to R_\lambda^{\xi}$ satisfies condition $\det$ if the pullback of
$$\tilde{\nabla}''^{\op{univ}}-\pi \cdot \pi^*\nabla''_{\xi,\alpha}  : \pi^* \Lambda^{\op{DR},\log D,R} \otimes \bigwedge^r \SE \to \bigwedge^r \SE$$
by $e$ is zero. By \cite[Lemma 4.3]{Yokogawa93}, there is a closed subscheme $R_\lambda^{\xi,\det}$ of $R_\lambda^{\xi}$ such that the pullback is zero if and only if $e$ factors through $R_\lambda^{\xi,\det}$.

Clearly, $R_\lambda^{\xi,\det}$ is a $\op{SL}_{P(N)}(\CC)$-invariant sub-scheme of $R_\lambda$. Let $R_\lambda^{\xi,\det,ss}=R_\lambda^{\xi,\det,ss}\cap R^{ss}$ and $R_\lambda^{\xi,\det,s}=R_\lambda^{\xi,\det,ss}\cap R^{s}$. Then the quotient of $R_\lambda^{ss}$ by $\op{SL}_{P(N)}(\CC)$ restricts to a quotient of $R_\lambda^{\xi,\det,ss}$. By Theorem \ref{thm:ResLambdaModModuli}, $R_\lambda \gitq \op{SL}_{P(N)}(\CC)$ corepresents $\SM(\Lambda,R,\lambda,P,\alpha, \overline{r})$, so $M_{\op{Hod}}(\xi,\alpha,\overline{r})=R_\lambda^{LF,\xi,det,ss}\gitq \op{SL}_{P(N)}(\CC)$ corepresents the sub-functor of $\SM(\Lambda,R,\lambda,P,\alpha, \overline{r})$ corresponding to locally free families with fixed determinant $\xi$ such that diagram \eqref{eq:ConnectionDeterminantDiagram} commutes. By the previous equivalent definition, this subfunctor coincides with $\SM_{\op{Hod}}(\xi,\alpha,\overline{r})$.

For each parabolic $\lambda$-connection $(E,E_\bullet,\varphi)$ in $R_\lambda^{ss}$, the closure of the orbit by the $\op{SL}_{P(N)}(\CC)$ action coincides with the set of S-equivalent parabolic $\lambda$-connections. The closure of the orbit always contain as a representative the graduate of its Jordan-H\"older filtration $\Gr(E)$, which is locally free by construction and has the same determinant bundle. Therefore, the local closures of two $\op{SL}_{P(N)}(\CC)$-orbits in $R_\lambda^{ss}$ intersect if and only if their closures intersect, because the intersection has at least a point in $R_\lambda^{\xi,\det,ss}$.

This proves that the set of closed points in $M_{\op{Hod}}(\xi,\alpha,\overline{r})$ is in correspondence with the desired set of isomorphism classes of parabolic $\lambda$-connections modulo S-equivalence.

By construction $M^s_{Hod}(\xi,\alpha,\overline{r})\subseteq M^s(\Lambda^{\op{DR},\log D,R},\overline{R}^\lambda,P,\alpha,\overline{r})$. If the coprimality condition of part (4) of the theorem holds, then by Theorem \ref{thm:fineModuli}, there exists a universal family on $M^s(\Lambda^{\op{DR},\log D,R},\overline{R}^\lambda,P,\alpha,\overline{r})$. Restricting it to $\SM^s _{Hod}(\xi,\alpha,\overline{r})$ we obtain the desired universal family.
\end{proof}

Now we will focus in developing the structure of the fibers of the moduli over $\mathbb{A}^1$. Let $\lambda$ be a closed point in $\mathbb{A}^1$. Let $f:T\to \mathbb{A}^1$ be any $\mathbb{A}^1$-scheme. The stability condition for parabolic $\Lambda$-modules over $S$ is stated point-wise on the base scheme. Therefore, by the base change formula for $\Lambda$-modules, any family of semistable parabolic $\Lambda^{\op{DR},\log D,R}$-modules which lie over $\lambda$, i.e., any $T$-point of $M_{\op{Hod}}(P,\alpha,\overline{r})$ over $\mathbb{A}^1$ such that the map $T\to \mathbb{A}^1$ is constant and identical to $\lambda$, is a family of semistable parabolic $\Lambda^{\op{DR},\log D,R}_\lambda \cong \Lambda^{\op{DR},\log D,\lambda}$-modules.

Therefore, the fiber of $M_{\op{Hod}}(P,\alpha,\overline{r})$ over $\lambda$ coincides with the moduli space of $\Lambda^{\op{DR},\log D,\lambda}$-modules. Moreover, if we restrict the right action \eqref{eq:deformationGraduateProduct} of $\SO_X \cong \SO_C \otimes_{\CC} \SO_{\mathbb{A}^1}$ on $\Lambda^{\op{DR},\log D,R}$ to $\SO_{\mathbb{A}^1}$, it induces an action of $\CC^*$ on the moduli which, by construction, preserves the fibers over $\mathbb{A}^1$. From equation \eqref{eq:deformationGraduateProduct}, this action coincides with the $\CC^*$ action
$$(E,E_\bullet,\nabla,\lambda) \cdot \mu = (E,E_\bullet ,\mu\nabla, \mu \lambda)$$
The action gives an explicit isomorphism between each fiber over $\lambda\ne 0$ and the fiber over $1\in \mathbb{A}^1$.

Let $\lambda=0$. Then, the $\SO_C$ bimodule structure \eqref{eq:lambdaModProduct} on $\Lambda^{\op{DR},\log D,0}$ reduces to
$$(g, v)\cdot f = (fg, fv)=f\cdot (g,v)$$
so the left and right $\SO_C$-actions are equal. As expected, $\Lambda^{\op{DR},\log D,0}$ coincides with the graduate of $\Lambda^{\op{DR},\log D}$, which is simply $\Lambda^{\op{Higgs},\log D}$. The residue restricts to $X_0\cong C$ as the null section, so the fiber of the moduli over $\lambda=0$ is a family of $0$-residual parabolic $\Lambda^{\op{Higgs},\log D}$-modules satisfying condition Therefore, the fiber over $\lambda=0$ coincides with the moduli space of parabolic Higgs bundles over $C$. On the other hand, from \eqref{eq:lambdaModProduct} it is clear that $\Lambda^{\op{DR},\log D,1}\cong \Lambda^{\op{DR},\log D}$, so the fiber over $\lambda=1$ (and therefore, over any nonzero $\lambda$) is isomorphic to the moduli space of vector bundles with a parabolic connection.

\bibliographystyle{alpha}
\bibliography{ParabolicLambdaModules}

\end{document}